\documentclass[sn-mathphys-num]{sn-jnl}

\usepackage{graphicx}%
\usepackage{multirow}%
\usepackage{amsmath,amssymb,amsfonts}%
\usepackage{amsthm}%
\usepackage{mathrsfs}%
\usepackage[title]{appendix}%
\usepackage{xcolor}%
\usepackage{textcomp}%
\usepackage{manyfoot}%
\usepackage{booktabs}%
\usepackage{algorithm}%
\usepackage{listings}%
\usepackage{algorithmic}
\usepackage{float}
\usepackage{textcomp}
\usepackage{mathtools}
\RequirePackage{subcaption}
\ifpdf
  \DeclareGraphicsExtensions{.eps,.pdf,.png,.jpg,.}
\else
  \DeclareGraphicsExtensions{.eps}
\fi

\theoremstyle{thmstyleone}%
\newtheorem{theorem}{Theorem}

\newtheorem{lemma}{Lemma}
\theoremstyle{thmstyletwo}%

\theoremstyle{thmstylethree}%
\newtheorem{definition}{Definition}%

\begin{document}

\title[Article Title]{A Shadow Sequence Technique for Last-Iterate Convergence Analysis of the Random Reshuffling Method: Relaxed Stepsize Conditions and Beyond}


\author[1]{\fnm{Hai} \sur{Liu}}\email{liuhai21@mails.ucas.ac.cn}

\author[1]{\fnm{Tiande} \sur{Guo}}\email{tdguo@ucas.ac.cn}

\author*[1]{\fnm{Congying} \sur{Han}}\email{hancy@ucas.ac.cn}

\affil[1]{\orgdiv{ School of Mathematical Sciences}, \orgname{University of Chinese Academy of Sciences (UCAS)}, \orgaddress{\city{Beijing}, \postcode{100049}, \country{China}}}


\abstract{We prove two last-iterate convergence bounds of the random reshuffling (RR) method in the smooth general convex finite-sum optimization. In the proved results, the stepsize sequence $\{\eta_{k}\}_{k}$ of RR is an arbitrary positive non-increasing sequence with $\eta_{1}\lesssim 1/(\operatorname{polylog}(n)L)$, where $n$ is the number of components of the finite-sum objective and $L$ is the smoothness parameter. Specifically, a last-iterate convergence rate $\widetilde{\mathcal{O}}\left(\frac{\Vert x_{1}^{1}-x_{*}\Vert^2}{n\sum_{k=1}^{K}\eta_{k}}+\sigma_{*}^{2}\sum_{k=1}^{K}\frac{\eta_{k}^{2}}{\sum_{l=k}^{K}\eta_{l}}\right)$ of RR is proved, where $x_{*}$ is a solution of the optimization problem and $\sigma_{*}^{2}$ is the variance at the solution $x_{*}$. This rate is novel to the last iterate of RR. Also, a last-iterate convergence rate $\widetilde{\mathcal{O}}\left(\frac{\Vert x_{1}^{1}-x_{*}\Vert^2}{n\sum_{k=1}^{K}\eta_{k}}+n\sigma_{*}^{2}\sum_{k=1}^{K}\frac{\eta_{k}^{3}}{\sum_{l=k}^{K}\eta_{l}} \right)$ of RR is proved. A similar rate exists in the literature, but only under a restrictive stepsize condition $\eta\lesssim 1/(n L)$ and does not allow varying stepsizes. At the heart of our proof for the latter rate is an auxiliary sequence named the shadow sequence, which has the same epoch endpoints as the iterate sequence generated by RR. The shadow sequence can be regarded as an iterate sequence generated by performing an inexact RR on an auxiliary interpolated problem or performing an inexact variant of random reshuffled SVRG on the original problem. The shadow sequence technique sheds light on deeper understanding of RR and may be useful in the analysis of other shuffling methods.}

\maketitle

\section{Introduction}

In this work, we consider the following finite-sum optimization problem:
\begin{equation}
\label{problem setting eq}
    \min_{x\in\mathbb{R}^{d}} f(x) \coloneqq \frac{1}{n}\sum_{i=1}^{n}f_{i}(x),
\end{equation}
where $f_{i}:\mathbb{R}^{d}\rightarrow\mathbb{R}$ are convex and have Lipschitz continuous gradients for all $i$. Problems of this type are ubiquitous in inverse problems \cite{jin2023convergence}, statistics  \cite{van2000asymptotic}, and machine learning \cite{bottou2018optimization}, etc. We are interested in the large-scale setting where $f$ has a large number of components $n$. Moreover, (\ref{problem setting eq}) often lives in a high dimensional space. Hence deterministic methods relying on calculating full gradients at each iteration are not suitable for solving (\ref{problem setting eq}).

Owing to its low per-iteration cost, scalability, and other benefits, the stochastic gradient descent (SGD) method \cite{robbins1951stochastic} is one of the most commonly used methods for solving (\ref{problem setting eq}). In its simplest form, SGD samples an index $i_{t}$ uniformly at random from $[n]\coloneqq \{1,2,\cdots, n\}$ at the $t^{\rm{th}}$ iteration and iterates as follows
\begin{equation}
    x_{t+1} = x_{t} - \eta_{t} \nabla f_{i_{t}}(x_{t}),
\end{equation}
where $\eta_{t}>0$ is the step size at the $t^{\rm{th}}$ iteration. 

However, the without-replacement version of SGD, a.k.a. the random reshuffling (RR) method \cite{bertsekas2015convex}, is used more frequently in practice. RR proceeds in epochs, with each epoch consisting of $n$ incremental updates according to a randomly generated permutation. Specifically, at the $k^{\rm{th}}$ epoch, RR samples a permutation $\pi_{k}$ uniformly at random from the set of all permutations of $[n]$, and iterates as follows
\begin{equation*}
    x_{k}^{t+1} = x_{k}^{t} - \eta_{k,t}\nabla f_{\pi_{k}^{t}}(x_{k}^{t}), \ t\in [n], \ x_{k+1}^{1} = x_{k}^{n+1},
\end{equation*}
where $\eta_{k,t}>0$ is the step size, and $\pi_{k}^{t}$ is the $t^{\rm{th}}$ element of the permutation $\pi_{k}$. The wide application of RR is due to its easier implementation and faster empirical convergence \cite{bottou2009curiously, bengio2012practical}, but the theoretical understanding of RR has been insufficient for a long period. The analysis of RR is significantly complicated by the fact that $\mathbb{E}_{\pi_{k}^{t}}[\nabla f_{\pi_{k}^{t}}(x_{k}^{t})] = \nabla f(x_{k}^{t})$ does not hold generally, and theoretical works on RR have been very limited until recently.

The first breakthrough is made by \cite{gurbuzbalaban2021random}, where it is proved that RR has the asymptotic convergence rate $\mathcal{O}(1/K^{2})$ ($K\in\mathbb{N}^{+}$ is the total number of epochs) in the smooth strongly convex setting, improving upon the $\mathcal{O} (1/K)$ rate of SGD under certain extra conditions\footnote{In this comparison, the $n$-dependence of the rates is omitted.}. Following this work, more researches further investigate RR under smooth strongly convex setting. \cite{haochen2019random} proves the non-asymptotic rate $\widetilde{\mathcal{O}}(1/(n K)^{2}+1/K^{3})$ of RR when the component functions in (\ref{problem setting eq}) are quadratic or have Lipschitz continuous Hessians\footnote{In the non-quadratic case, bounded iterates assumption and bounded gradients assumption are made.}, showing that RR converges faster than SGD after an explicitly given finite number of epochs. However, the aforementioned works both rely on the Hessian Lipschitz assumption, which is needed to extend the analysis of quadratic cases to the general strongly convex setting. Using a novel exchangeable pairs argument, \cite{nagaraj2019sgd} proves the rate $\widetilde{\mathcal{O}}(1/n K^{2})$ in the expected suboptimality of the objective value without the Hessian Lipschitz assumption. \cite{mishchenko2020random} proposes a novel analysis and the notion of shuffling variance, based on which it improves the dependence on the condition number in the convergence rate under different assumptions. Specifically, in \cite{mishchenko2020random}, the bounded gradients assumption used by \cite{haochen2019random} and \cite{nagaraj2019sgd} is removed, but an additional individual strong convexity assumption is made. The analytical techniques and tools developed in the aforementioned works extend beyond the smooth strongly convex setting and remain useful and inspiring in other settings. The reader could refer to \cite{shamir2016without, ahn2020sgd, nguyen2021unified, rajput2020closing, safran2021random, cha2023tighter} for more works on RR under this setting.

In the smooth nonconvex setting, \cite{mishchenko2020random} and \cite{nguyen2021unified} prove the convergence rate $\mathcal{O}(1/(n^{1/3}K^{2/3}))$ in the minimum expected squared gradient norm under a general variance bound. \cite{nguyen2021unified} also proves $\lim\inf_{k\rightarrow \infty} \Vert \nabla f(x_{k}^{1})\Vert = 0$ for stepsizes $\{\eta_{k}\}_{k}$ satisfying $\sum_{k=1}^{\infty}\eta_{k}=\infty$ and $\sum_{k=1}^{\infty}\eta_{k}^{3}<\infty$. Under similar stepsize conditions, \cite{li2023convergence} proves a stronger weak convergence result $\lim_{k\rightarrow\infty} \nabla f(x_{k}^{1}) = 0$. In addition, under the Kurdyka-{\L}ojasiewicz (KL) inequality, \cite{li2023convergence} proves that the iterate sequence generated by RR converges to a single stationary point almost surely and derives the corresponding asymptotic convergence rate which depends on the KL exponent and the stepsize selection. Also, the reader could refer to \cite{liu2025improved} and references therein for upper and lower bounds of RR in the nonsmooth setting.

In the smooth general convex setting, \cite{nagaraj2019sgd} first proves the convergence rate $\mathcal{O}(1/\sqrt{n K})$ of RR under the additional bounded gradients assumption in the constrained optimization setting where the constraint set is compact. \cite{mishchenko2020random} and \cite{nguyen2021unified} consider RR in the unconstrained setting and prove the convergence rate $\mathcal{O}(1/(n^{1/3}K^{2/3}))$ without the bounded gradients assumption. It is noteworthy that the theoretical guarantees of all aforementioned works \cite{nagaraj2019sgd}, \cite{mishchenko2020random}, and \cite{nguyen2021unified} are for the averaged iterate in the smooth general convex setting. However, returning the last iterate is commonly used in practice due to its lower computational and memory cost. It was unknown if a tight convergence bound can be proven for the last-iterate of RR in this setting. Hence, there existed a significant discrepancy between theory and practice, whose solution seemed elusive and was obtained only recently.

Leveraging the ideas from \cite{zamani2025exact}, a recent work \cite{liu2024last} establishes the last-iterate convergence rate $\widetilde{\mathcal{O}}(1/(n^{1/3}K^{2/3}))$ of RR in the expected suboptimality of the objective value, which nearly matches the last-iterate lower bound $\Omega(1/(n^{1/3}K^{2/3}))$ \cite{cha2023tighter}\footnote{We omit problem-dependent parameters in both the upper and lower bounds here for simplicity.}. Although the (almost) tight last-iterate convergence bound is proved, the corresponding constant stepsize $\eta$ is required to satisfy $\eta\lesssim 1/(n L)$. In the large-scale setting where $n$ is substantial, the threshold $1/(n L)$ can be extremely small or even vanish, which leaves a significant unresolved gap between theory and practice. The overly restrictive stepsize condition $\eta\lesssim 1/(n L)$ also appears in all other works addressing this setting, namely \cite{mishchenko2020random} and \cite{nguyen2021unified}\footnote{We do not include \cite{nagaraj2019sgd} in this discussion since an extra bounded gradients assumption is made there. However, even if we ignore stronger assumptions, a threshold $O(1/\sqrt{n})$ in terms of $n$-dependence still appears in \cite{nagaraj2019sgd}.}.

Very recently, \cite{liu2026random} proves the convergence rate $\mathcal{O}(D^{2}/(\eta n K)+\min\{1,\eta n L\}\eta\sigma_{*}^{2})$\footnote{$D\coloneqq \Vert x_{1}^{1}-x_{*} \Vert$ and $\sigma_{*}^{2}\coloneqq n^{-1}\sum_{i=1}^{n}\Vert \nabla f_{i}(x_{*})\Vert^{2}$ where $x_{*}$ is a solution of (\ref{problem setting eq}).} of RR for the stepsize satisfying $\eta\lesssim 1/L$, establishing the dominance of RR over SGD in the smooth general convex setting. The $\mathcal{O}(D^{2}/(\eta n K)+\eta\sigma_{*}^{2})$ branch is the same as the SGD rate and is novel in the RR literature. This branch guarantees that RR is never worse than SGD for $\eta\lesssim 1/L$. The other branch $\mathcal{O}(D^{2}/(\eta n K)+\eta^{2} n L \sigma_{*}^{2})$ shows that RR is faster than SGD when the stepsize is small, and this rate is known previously but only for the stepsize $\eta$ satisfying $\eta\lesssim 1/(n L)$. However, the theoretical guarantees of \cite{liu2026random} are for the averaged iterate rather than the last iterate, which is undesirable as discussed before.

Given the importance of both the last-iterate convergence guarantee and the mildness of the stepsize condition, it is natural to ask whether tight last-iterate convergence bounds of RR can be proved under mild stepsize conditions. We address this question in this work, and our contributions are briefly summarized as follows.

First, we prove a last-iterate convergence rate $\widetilde{\mathcal{O}}\left(\frac{\Vert x_{1}^{1}-x_{*}\Vert^2}{n\sum_{k=1}^{K}\eta_{k}}+\sigma_{*}^{2}\sum_{k=1}^{K}\frac{\eta_{k}^{2}}{\sum_{l=k}^{K}\eta_{l}}\right)$ of RR with an arbitrary positive non-increasing stepsize sequence $\{\eta_{k}\}_{k}$ with an initial stepsize satisfying $\eta_{1}\lesssim 1/(\operatorname{polylog(n)}L)$, where $x_{*}$ is a solution of (\ref{problem setting eq}) and $\sigma_{*}^{2}\coloneqq n^{-1}\sum_{i=1}^{n}\Vert \nabla f_{i}(x_{*})\Vert^2$. A similar convergence rate has only been proved to hold for the averaged iterate of RR previously \cite{liu2026random}. 

Second, we prove a last-iterate convergence rate $\widetilde{\mathcal{O}}\left(\frac{\Vert x_{1}^{1}-x_{*}\Vert^2}{n\sum_{k=1}^{K}\eta_{k}}+n\sigma_{*}^{2}\sum_{k=1}^{K}\frac{\eta_{k}^{3}}{\sum_{l=k}^{K}\eta_{l}} \right)$ of RR with an arbitrary positive non-increasing stepsize sequence $\{\eta_{k}\}_{k}$ with an initial stepsize satisfying $\eta_{1}\lesssim 1/(\operatorname{polylog(n)}L)$. A similar rate has been proved by \cite{liu2024last}, but only under a restrictive stepsize condition $\eta\lesssim 1/(n L)$, which is unrealistic in large-scale optimization. Also, the result in \cite{liu2024last} does not allow varying stepsizes for technical reasons in the proofs.

Third, in the proof of the latter bound, we introduce an auxiliary iterate sequence called the shadow sequence. The shadow sequence has the same epoch endpoints as the iterate sequence generated by RR, which allows us to investigate and understand the last-iterate convergence of RR by studying the shadow sequence. Two interpretations of the shadow sequence are given, which support and make clearer the existing viewpoint that RR has an implicit variance reduction effect.

\section{Notation, Definitions and Assumptions}

$\langle \cdot, \cdot\rangle$ is the standard Euclidean inner product, and $\Vert \cdot\Vert\coloneqq \sqrt{\langle \cdot, \cdot\rangle}$ is the induced Euclidean norm. The Bregman divergence associated with a differentiable function $h$ is defined as $D_{h}(y,x)\coloneqq h(y)-h(x)-\langle \nabla h(x),y-x\rangle$, which is nonnegative if $h$ is convex. The set of all positive integers is denoted as $\mathbb{N}^{+}$. Given $p\in\mathbb{N}^{+}$, denote $[p]\coloneqq \{1,2,\cdots,p\}$. Given a permutation $\pi$ of $[p]$, denote $\pi^{r}$ as the $r^{\rm{th}}$ element of $\pi$ for $r\in [p]$. If the optimization problem (\ref{problem setting eq}) has a nonempty solution set, then denote a fixed solution as $x_{*}$ and further define $f_{*}\coloneqq f(x_{*})$, $\sigma_{*}^{2}\coloneqq n^{-1}\sum_{i=1}^{n}\Vert \nabla f_{i}(x_{*})\Vert^2$. The symbol $\mathbb{E}$ denotes the expectation of a random element. Given a stochastic optimization method, the symbol $\mathbb{E}[\cdot|X]$ denotes an expectation conditioning on $X$ and all random variables (r.v.'s) appeared before $X$ in the optimization process. The notation $X\sim \mathcal{D}$ means that the r.v. $X$ follows the distribution $\mathcal{D}$. The discrete uniform distribution on the set $[n]$ is denoted as $\operatorname{Unif}[n]$. The notation $X \overset{\mathcal{D}}{=} Y$ means that the r.v.'s $X$ and $Y$ are equal in distribution.

In this work, the following standard assumptions will be made throughout.

\textbf{Assumption $1$}. The optimization problem (\ref{problem setting eq}) has a nonempty solution set.

\textbf{Assumption $2$}. Each $f_{i}$ in (\ref{problem setting eq}) is convex.

\textbf{Assumption $3$}. Each $f_{i}$ in (\ref{problem setting eq}) is $L$-smooth, i.e., each $f_{i}$ is differentiable and there exists an $L>0$ such that $\Vert \nabla f_{i}(x)-\nabla f_{i}(y)\Vert\leq L\Vert x-y\Vert, \ \forall x,y\in\mathbb{R}^{d}$.

In the statements of the propositions and the proofs, the term ``the algorithm" always refers to the RR method with positive non-increasing stepsizes which is defined as follows.

\section{Algorithm}

The algorithm under study is the RR method. Define the set of all permutations of $[n]$ as follows
\begin{equation*}
    \Gamma \coloneqq \left\{\pi \ :\pi \ \rm{is} \ \rm{a} \ \rm{permutation} \ \rm{of} \ [n] \right\}.
\end{equation*}
In each outer iteration $k\in\mathbb{N}^{+}$, RR draws a permutation $\pi_{k}$ independently and uniformly at random over $\Gamma$ and iterates as follows
\begin{equation*}
    x_{k}^{t+1} = x_{k}^{t} - \eta_{k} \nabla f_{\pi_{k}^{t}}(x_{k}^{t}), \ t = 1,2,\cdots, n,
\end{equation*}
where $\eta_{k}>0$ is the stepsize and the stepsize sequence $\{\eta_{k}\}_{k}$ is non-increasing. After $n$ incremental updates are completed in the $k^{\rm{th}}$ outer iteration, RR sets $x_{k+1}^{1} = x_{k}^{n+1}$ and enters the next outer iteration.

\section{Main Results and Related Discussions}

\subsection{Convergence Bound I}

\begin{theorem}
\label{theorem, the hard bound}
    Assume the stepsize sequence $\{\eta_{k}\}_{k}$ is non-increasing and the initial stepsize $\eta_{1}$ satisfies $\eta_{1}\leq \frac{1}{17 C^{3/2}\Lambda^{3/2}L}$, where $C = 3e^{3}$ is an absolute numerical constant and $\Lambda = \log(nK+1)+1$. Then after $K\in\mathbb{N}^{+}$ epochs, the algorithm guarantees that\footnote{We make no effort to optimize the numerical constants in the proved convergence bounds in this manuscript, especially in Lemma \ref{lemma, aux seq property}, to make the main ideas of the proof clearer.}
    \begin{equation}
    \label{theorem, the hard bound, conclusion eq}
        \mathbb{E} \left[f(x_{K}^{n+1})-f_{*} \right] \leq \frac{5C \Vert x_{1}^{1}-x_{*}\Vert^2}{n\sum_{k=1}^{K}\eta_{k}} + 1710C^{3}\Lambda^{2}L\sum_{k=1}^{K} \frac{n\sigma_{*}^{2}\eta_{k}^{3}}{\sum_{l=k}^{K}\eta_{l}}.
    \end{equation}
\end{theorem}

A similar rate has been established in \cite{liu2024last}, but only under a restrictive stepsize condition $\eta_{k}\lesssim 1/(n L)$. Notably, through a novel analysis, we relax this condition to $\eta_{k}\lesssim 1/(\operatorname{polylog}(n)L)$ while preserving the favourable convergence rate, hence broadening the applicability of this convergence result considerably.

By setting a constant stepsize
\begin{equation*}
    \eta_{k} = \eta_{1} = \min\left\{ \frac{1}{17 C^{3/2}\Lambda^{3/2}L} , \frac{\Vert x_{1}^{1}-x_{*}\Vert^{2/3}}{(342C^{2}\Lambda^{2}L\sigma_{*}^{2})^{1/3}n^{2/3}(K(1+\log K))^{1/3}}\right\},
\end{equation*}
we get the convergence rate $\widetilde {\mathcal{O}}\left(\frac{L^{1/3}\sigma_{*}^{2/3}\Vert x_{1}^{1}-x_{*}\Vert^{4/3}}{n^{1/3}K^{2/3}}\right)$ which matches the last-iterate lower bound $\Omega \left(\frac{L^{1/3}\sigma_{*}^{2/3}\Vert x_{1}-x_{*}\Vert^{4/3}}{n^{1/3}K^{2/3}}\right)$ \cite{cha2023tighter} up to logarithmic factors. This result demonstrates that the last-iterate RR is faster than the averaged-iterate SGD when the epoch number $K$ is large.

Another interesting observation is that the proved convergence bound in Theorem \ref{theorem, the hard bound} allows for an arbitrary positive non-increasing stepsize sequence if the initial stepsize is not too large, while the convergence theory in \cite{liu2024last} only considers the constant stepsize setting. 

This limitation of \cite{liu2024last} is arguably due to their key technical lemma, Lemma E.$2$, which considers (in their notation) a scalar sequence $\{d_{k}\}_{k=2}^{K+1}$ satisfying
\begin{equation*}
    d_{k+1} \leq \frac{a}{k} + b(1+\log k) + c\sum_{l=2}^{k}\frac{d_{l}}{k-l+2}, \ \forall k\in [K],
\end{equation*}
where $a$, $b$, and $c$ are positive constants. Lemma E.$2$ \cite{liu2024last} proves the following explicit upper bound for the last element $d_{K+1}$\footnote{By the word explicit, we mean that the upper bound does not contain elements in the sequence $\{d_{k}\}_{k=2}^{K+1}$.}
\begin{equation*}
    d_{K+1} \leq \left(\frac{a}{K} + b(1+\log K) \right)\sum_{i=0}^{K-1}\left( 2c(1+\log K)\right)^{i}.
\end{equation*}
The element $d_{K+1}$ in this technical lemma corresponds to the convergence measure $\mathbb{E}[f(x_{K}^{n+1})-f_{*}]$ where $x_{K}^{n+1}$ is the last iterate of RR. \cite{liu2024last} proves\footnote{In the following equation, $R_{l}\coloneqq \sum_{i=2}^{n}\frac{L}{n}\Vert \sum_{j=1}^{i-1}\nabla f_{\pi_{l}^{j}}(x_{*})\Vert^2$.}
\begin{equation}
\label{Liu's 2024 mid eq}
    f(x_{k+1}^{1}) - f_{*} \leq \frac{\Vert x_{1}^{1}-x_{*}\Vert^2}{2n\sum_{l=1}^{k}\eta_{l}} + \sum_{l=1}^{k}\frac{4\eta_{l}^{3}R_{l}}{\sum_{s=l}^{k}\eta_{s}} + \sum_{l=2}^{k} \frac{8n^{2}L^{2}\eta_{l-1}(\sum_{s=l}^{k}\eta_{s}^{3})}{(\sum_{s=l}^{k}\eta_{s})(\sum_{s=l-1}^{k}\eta_{s})}D_{f}(x_{l}^{1},x_{*}),
\end{equation}
and they adopt the constant stepsize setting to ensure the coefficients $a$, $b$, and $c$ are constants to apply the aforementioned technical lemma, so that the convergence rate of $\mathbb{E}[f(x_{K}^{n+1})-f_{*}]$ can be obtained. To allow an arbitrary non-increasing stepsize sequence, a natural resolution is to replace the stepsizes $\eta_{i}$'s in (\ref{Liu's 2024 mid eq}) with $\eta_{1}$ and $\eta_{k}$. In this way, one obtains a relaxed upper bound to which one can still apply Lemma E.$2$ \cite{liu2024last}. However, this natural approach introduces an $1/\eta_{k}$ factor in the final convergence bound, which causes the bound to explode when $\eta_{k}\downarrow 0$.

In contrast, the proved bound (\ref{theorem, the hard bound, conclusion eq}) does not explode when the stepsize adopted in the last epoch $\eta_{K}$ approaches $0$. To see why this is more reasonable, consider a thought experiment in which the algorithm is run for $K-1$ epochs and the stepsize in the $K^{\rm{th}}$ epoch is set to (almost) $0$, arguably making the iterate barely move (if it moves at all) in the $K^{\rm{th}}$ epoch. Then we should expect that $\mathbb{E}[f(x_{K}^{n+1})-f_{*}]$ satisfies a convergence bound similar to that for $\mathbb{E}[f(x_{K-1}^{n+1})-f_{*}]$, which is indeed the case in our theory. However, the convergence bound of $\mathbb{E}[f(x_{K}^{n+1})-f_{*}]$ explodes to infinity in the aforementioned natural extension of Theorem C.$1$ in \cite{liu2024last}.

If the aforementioned natural extension is not to be applied, then it seems elusive to extend Theorem $4.4$ \cite{liu2024last} to allow an arbitrary non-increasing stepsize sequence $\{\eta_{k}\}_{k}$. Hence, this observation suggests that our analysis is more favourable than that in \cite{liu2024last}, in the sense that our analysis allows more flexible stepsize strategies\footnote{However, \cite{liu2024last} considers composite optimization which is a broader setting.}.

\subsection{Convergence Bound II}

\begin{theorem}
\label{theorem, the other bound}
    Assume the stepsize sequence $\{\eta_{k}\}_{k}$ is non-increasing and the initial stepsize $\eta_{1}$ satisfies $\eta_{1}\leq \frac{1}{480C\Lambda^{2}L}$, where $C = 3e^{3}$ is an absolute numerical constant and $\Lambda = \log(nK+1)+1$. Then after $K\in\mathbb{N}^{+}$ epochs, the algorithm guarantees that
    \begin{equation}
    \label{theorem, the other bound, conclusion eq}
        \mathbb{E}\left[f(x_{K}^{n+1})-f_{*} \right] \leq \frac{2C\Lambda \Vert x_{1}^{1}-x_{*}\Vert^{2}}{n\sum_{k=1}^{K}\eta_{k}} + 257 C^{2}\Lambda^{3}\sigma_{*}^{2}\sum_{k=1}^{K}\frac{\eta_{k}^{2}}{\sum_{l=k}^{K}\eta_{l}}.
    \end{equation}
\end{theorem}

A sharper rate, which is better than the proved (\ref{theorem, the other bound, conclusion eq}) by logarithmic factors and requires a stepsize condition $\eta\lesssim 1/L$, is recently established in \cite{liu2026random}. However, their result is for the averaged iterate rather than the last iterate. To the best of our knowledge, a last-iterate convergence bound which is similar to (\ref{theorem, the other bound, conclusion eq}) has not been established in the literature until our work.

By setting a constant stepsize
\begin{equation*}
    \eta_{k} = \eta_{1} = \min\left\{\frac{1}{480C\Lambda^{2}L}, \frac{\Vert x_{1}^{1}-x_{*}\Vert}{16\sqrt{C}\Lambda \sigma_{*} \sqrt{n K (1+\log K)}}\right\},
\end{equation*}
we get the convergence rate $\tilde{\mathcal{O}}\left(\frac{L\Vert x_{1}^{1}-x_{*}\Vert^2}{n K}+\frac{\sigma_{*}\Vert x_{1}^{1}-x_{*}\Vert}{\sqrt{n K}}\right)$, which matches the averaged-iterate convergence rate $\mathcal{O}\left(\frac{L\Vert x_{1}-x_{*}\Vert^2}{n K} + \frac{\sigma_{*}\Vert x_{1}-x_{*}\Vert}{\sqrt{n K}} \right)$ of SGD up to logarithmic factors. This result demonstrates that last-iterate RR is comparable to SGD (even with the less commonly used averaged iterate) for all epoch numbers $K\in\mathbb{N}^{+}$ up to logarithmic factors.

During proof-reading and revision of this manuscript, we learned about a recent independent work \cite{chen2026large} which obtained a rate similar to (\ref{theorem, the other bound, conclusion eq}). Their result applies to every element of the iterate sequence, but as in \cite{liu2024last}, the convergence theory in \cite{chen2026large} does not allow varying stepsizes. More importantly, \cite{chen2026large} does not contain a rate similar to (\ref{theorem, the hard bound, conclusion eq}) under a mild stepsize condition. In fact, \cite{chen2026large} directly uses the convergence rate of Theorem $4.4$ in \cite{liu2024last}, which is established under the restrictive constant stepsize condition $\eta\lesssim 1/(n L)$, to complete their main theorem.

\subsection{The Shadow Sequence Technique}

At the heart of the proof of Theorem \ref{theorem, the hard bound} is an auxiliary iterate sequence $\{u_{k}^{i}\}_{k,i}$ defined in Definition \ref{def, basics}, which might be called the shadow sequence to distinguish it from other virtual sequences in the proof. Since this sequence may be of independent interest, we give some intuition behind its construction and two interpretations of it.

\subsubsection{Motivation}

One ingredient of its construction is the idea originally proposed in \cite{zamani2025exact} and later adapted by \cite{liu2024last} to analyze the last-iterate convergence of RR. We take the proof of \cite{liu2024last} as an example to demonstrate the aforementioned idea. \cite{liu2024last} designs an auxiliary sequence $\{z_{k,l}\}_{k,l}$ where $z_{k,l}$ is a convex combination of $x_{*},x_{1}^{1},x_{2}^{1},\cdots, x_{l-1}^{1}$ ($\{x_{k}^{i}\}_{k,i}$ is the iterate sequence generated by RR), i.e., $z_{k,l} = w_{k,0}x_{*}+\sum_{j=1}^{l-1}w_{k,j}x_{j}^{1}$. Instead of bounding $f(x_{j+1}^{1})-f(x_{*})$, \cite{liu2024last} bounds $f(x_{j+1}^{1})-f(z_{k,j})$. By the convex combination formulation of $\{z_{k,l}\}_{k,l}$ and convexity of $f$, one sees that $f(z_{k,j})\leq w_{k,0}f(x_{*})+\sum_{i=1}^{j-1}w_{k,i}f(x_{i}^{1})$.
Based on the above property, \cite{liu2024last} proves $\sum_{j=1}^{k}p_{j}(f(x_{j+1}^{1})-f(z_{k,j}))\geq \Omega (f(x_{k+1}^{1})-f(x_{*}))$ for all $k\in [K]$ and based on which \cite{liu2024last} obtains the convergence and convergence rate of $\mathbb{E}[f(x_{K}^{n+1})-f_{*}]$, where $\{p_{k}\}_{k}$ is a carefully designed scalar sequence.

Another ingredient is the idea originally proposed in \cite{mishchenko2020random} and later utilized and developed by \cite{liu2026random} to provide an averaged-iterate convergence rate of RR under an $n$-independent stepsize condition $\eta_{k}\lesssim 1/L$ in smooth general convex finite-sum optimization. We briefly introduce the proof of Theorem $3$ in \cite{liu2026random}. \cite{liu2026random} first introduces an auxiliary sequence, which is originally constructed in \cite{mishchenko2020random}, as follows
\begin{equation}
\label{motivation, the shifted optima eq}
    y_{k}^{i+1} \coloneqq y_{k}^{i} - \eta_{k}\nabla f_{\pi_{k}^{i}}(x_{*}),  \ y_{k+1}^{1} \coloneqq y_{k}^{n+1}, \ y_{1}^{1} \coloneqq x_{*}.
\end{equation}
Since $\sum_{i=1}^{n}\nabla f_{\pi_{k}^{i}}(x_{*}) = \sum_{i=1}^{n}\nabla f_{i}(x_{*}) = n\nabla f(x_{*}) = 0$, one sees that the first element and the last element of $\{y_{k}^{i}\}_{i}$ coincide with $x_{*}$, i.e.,
\begin{equation}
\label{motivation, a property of the y seq eq}
    y_{k}^{1} = y_{k}^{n+1} = x_{*}.
\end{equation}
Also inspired by \cite{mishchenko2020random}, \cite{liu2026random} recursively bounds $\Vert x_{k}^{i+1}-y_{k}^{i+1}\Vert^2$ using the recursive definitions of $\{x_{k}^{i}\}_{i}$ and $\{y_{k}^{i}\}_{i}$ as follows
\begin{equation}
\label{motivation, recursive relation btw two sequences eq}
    \Vert x_{k}^{i+1}-y_{k}^{i+1}\Vert^2 = \Vert x_{k}^{i}-y_{k}^{i}-\eta_{k}(\nabla f_{\pi_{k}^{i}}(x_{k}^{i})-\nabla f_{\pi_{k}^{i}}(x_{*}))\Vert^2,
\end{equation}
and then obtains the following inequality
\begin{equation*}
    \eta_{k}D_{f_{\pi_{k}^{i}}}(x_{k}^{i},x_{*}) \leq \Vert x_{k}^{i}-y_{k}^{i}\Vert^2 - \Vert x_{k}^{i+1}-y_{k}^{i+1}\Vert^2 + 2\eta_{k}D_{f_{\pi_{k}^{i}}}(y_{k}^{i},x_{*}).
\end{equation*}
Summing the above inequality from $i=1$ to $n$ gives
\begin{equation*}
    \eta_{k}\sum_{i=1}^{n} D_{f_{\pi_{k}^{i}}}(x_{k}^{i},x_{*})\leq \Vert x_{k}^{1}-y_{k}^{1}\Vert^2 - \Vert x_{k}^{n+1}-y_{k}^{n+1}\Vert^2 + 2\eta_{k}\sum_{i=1}^{n} D_{f_{\pi_{k}^{i}}}(y_{k}^{i},x_{*}).
\end{equation*}
By $x_{k}^{n+1}=x_{k+1}^{1}$ and the important property (\ref{motivation, a property of the y seq eq}), \cite{liu2026random} obtains
\begin{equation}
\label{motivation, one-step eq}
    \eta_{k}\sum_{i=1}^{n} D_{f_{\pi_{k}^{i}}}(x_{k}^{i},x_{*})\leq \Vert x_{k}^{1}-x_{*}\Vert^2 - \Vert x_{k+1}^{1}-x_{*}\Vert^2 + 2\eta_{k}\sum_{i=1}^{n} D_{f_{\pi_{k}^{i}}}(y_{k}^{i},x_{*}),
\end{equation}
where the terms $\Vert x_{k}^{1}-x_{*}\Vert^2 - \Vert x_{k+1}^{1}-x_{*}\Vert^2$ will telescope in the summation w.r.t. the epoch number $k$. \cite{liu2026random} proceeds to bound the difference between $D_{f_{\pi_{k}^{i}}}(x_{k}^{i},x_{*})$ and $D_{f}(x_{k}^{i},x_{*})$ and eventually obtains a bound for $\sum_{k=1}^{K}\sum_{i=1}^{n}\eta_{k} D_{f}(x_{k}^{i},x_{*})$.

We would like to apply the ideas from \cite{zamani2025exact, liu2024last} to modify the proofs in \cite{liu2026random} to prove the last-iterate convergence of RR under certain mild stepsize condition. Our target is to derive an inequality which is similar to (\ref{motivation, one-step eq}), with $D_{f_{\pi_{k}^{i}}}(x_{k}^{i},x_{*})$ replaced by $D_{f_{\pi_{k}^{i}}}(x_{k}^{i},\tilde{z}_{k})$ for a certain reference point $\tilde{z}_{k}$. However, this idea is faced with immediate obstacles. For example, if one mimics (\ref{motivation, the shifted optima eq}) to define
\begin{equation*}
    y_{k}^{i+1} \coloneqq y_{k}^{i} - \eta_{k}\nabla f_{\pi_{k}^{i}}(\tilde{z}_{k}), \ y_{k}^{1} \coloneqq \tilde{z}_{k},
\end{equation*}
then typically $y_{k}^{n+1}\neq \tilde{z}_{k}$ and we do not have $y_{k}^{1} = y_{k}^{n+1} = \tilde{z}_{k}$ to facilitate the telescoping across epochs (with the technique from \cite{zamani2025exact}). Other candidate remedies also face critical challenges, e.g., they may introduce the restrictive stepsize condition $\eta_{k}\lesssim 1/(n L)$ which obviously contradicts our purpose.

\subsubsection{Construction and Intuitions}

The proof of (\ref{motivation, one-step eq}) in \cite{liu2026random} critically depends on the philosophy of \cite{mishchenko2020random}, which argues that the iterate sequence $\{x_{k}^{i}\}_{i}$ of RR in the $k^{\rm{th}}$ epoch converges to a sequence of shifted points derived from $x_{*}$ defined as the sequence $\{y_{k}^{i}\}_{i}$ (see (\ref{motivation, the shifted optima eq})) rather than $x_{*}$ itself.

We benefit but also deviate from this idea and construct the shadow sequence as follows. We recall the unrolling step in the recursive bounding of the two sequences $\{x_{k}^{i}\}_{i}$ and $\{y_{k}^{i}\}_{i}$, equation (\ref{motivation, recursive relation btw two sequences eq}), and present it again as follows
\begin{equation}
    \Vert x_{k}^{i+1}-y_{k}^{i+1}\Vert^2 = \Vert x_{k}^{i}-y_{k}^{i}-\eta_{k}\nabla f_{\pi_{k}^{i}}(x_{k}^{i})+\eta_{k}\nabla f_{\pi_{k}^{i}}(x_{*})\Vert^2.
\end{equation}
Instead of attributing the term $\eta_{k}\nabla f_{\pi_{k}^{i}}(x_{*})$ to the unrolling of the update of $\{y_{k}^{i}\}_{i}$ as in \cite{mishchenko2020random,liu2026random}, we attribute this term, along with the term $-\eta_{k}\nabla f_{\pi_{k}^{i}}(x_{k}^{i})$, to the unrolling of the update of the sequence $\{x_{k}^{i}\}_{i}$\footnote{We write $\{x_{k}^{i}\}_{i}$ here with slight abuse of notation.}.

With this change, we are actually considering two sequences $\{\tilde{x}_{k}^{i}\}_{i}$ and $\{\tilde{y}_{k}^{i}\}_{i}$, where the latter sequence satisfies $\tilde{y}_{k}^{i} \equiv \tilde{z}_{k}$ and is constant, while the former sequence iterates as follows
\begin{equation}
\label{construction, def of the shadow seq}
    \tilde{x}_{k}^{i+1} \coloneqq \tilde{x}_{k}^{i} - \eta_{k}(\nabla f_{\pi_{k}^{i}}(x_{k}^{i})-\nabla f_{\pi_{k}^{i}}(x_{*})), \ \tilde{x}_{k}^{1} \coloneqq x_{k}^{1}.
\end{equation}
Since $\sum_{i=1}^{n}\nabla f_{\pi_{k}^{i}}(x_{*}) = \sum_{i=1}^{n}\nabla f_{i}(x_{*}) = n\nabla f(x_{*}) = 0$, we have
\begin{equation}
\label{construction, property of the shadow seq}
    \tilde{x}_{k}^{1} = x_{k}^{1}, \ \tilde{x}_{k}^{n+1} = x_{k}^{n+1}.
\end{equation}
In words, given a fixed epoch, the start point and the end point of $\{\tilde{x}_{k}^{i}\}_{i}$ in that epoch coincide with those of the RR generated sequence $\{x_{k}^{i}\}_{i}$, respectively. The sequence defined in (\ref{construction, def of the shadow seq}) is the shadow sequence, which is denoted as $\{u_{k}^{i}\}_{k,i}$ in our proofs.

Benefits of this construction are apparent. For example, since $\tilde{y}_{k}^{i} \equiv \tilde{z}_{k}, i\in [n+1]$, the difference between the squared distances in our one-step descent inequality (see Lemma \ref{lemma, one step lemma})  will telescope in the weighted summation w.r.t. the epoch number $k$ (with a technique similar to that of \cite{zamani2025exact}). For another example, although we are bounding a weighted sum of $\{D_{f}(\tilde{x}_{k}^{i},\tilde{z}_{k})\}_{k,i}$ and finally obtain a bound for $\mathbb{E} D_{f}(\tilde{x}_{K}^{n+1},x_{*}) = \mathbb{E}[f(\tilde{x}_{K}^{n+1})-f_{*}]$, we nevertheless obtain an upper bound on the quantity of interest $\mathbb{E}[f(x_{K}^{n+1})-f_{*}]$ due to the property (\ref{construction, property of the shadow seq}).

\textit{In a nutshell, a perturbed version of the RR-generated iterate sequence with vanished perturbation at endpoints of every epoch, namely the shadow sequence, converges to the non-perturbed solution of the optimization problem.}

\subsubsection{Two Interpretations}

We have two interpretations for the shadow sequence. To illustrate the first interpretation, we use the notation in Definition \ref{def, basics} and define 
\begin{equation*}
    B(x)\coloneqq f(x)-f(x_{*}), \ b_{i}(x)\coloneqq f_{i}(x)-f_{i}(x_{*})-\langle \nabla f_{i}(x_{*}),x-x_{*}\rangle, \ i\in [n].
\end{equation*}
Note that $\nabla b_{i}(x) = \nabla f_{i}(x)-\nabla f_{i}(x_{*})$, then by (\ref{construction, def of the shadow seq}), the iteration of the shadow sequence can be reformulated as follows
\begin{equation}
\label{interpretation, reformulation of the shadow seq}
    \tilde{x}_{k}^{i+1} = \tilde{x}_{k}^{i} - \eta_{k}\nabla b_{\pi_{k}^{i}}(\tilde{x}_{k}^{i}+(x_{k}^{i}-\tilde{x}_{k}^{i})).
\end{equation}
Consider the following auxiliary finite-sum problem
\begin{equation}
\label{interpretation, aux optimization problem}
    \min_{x\in\mathbb{R}^{d}} B(x) = \frac{1}{n}\sum_{i=1}^{n}b_{i}(x),
\end{equation}
where the equality is by definitions of $b_{i}$ and $B$. It is apparent that all $b_{i}$ ($i\in [n]$) share the same minimizer $x_{*}$, i.e., the problem (\ref{interpretation, aux optimization problem}) satisfies the interpolation condition. Hence, regarding the term $x_{k}^{i}-\tilde{x}_{k}^{i}$ as a stochastic error, the shadow sequence is generated by an inexact RR performed on an auxiliary interpolated finite-sum problem.

For the second interpretation, we recall the definition of the shadow sequence
\begin{equation}
\label{interpretation, def of the shadow seq again}
    \tilde{x}_{k}^{i+1} \coloneqq \tilde{x}_{k}^{i} - \eta_{k}(\nabla f_{\pi_{k}^{i}}(x_{k}^{i})-\nabla f_{\pi_{k}^{i}}(x_{*})).
\end{equation}
Since the optimization problem (\ref{problem setting eq}) is unconstrained, we have $\nabla f(x_{*}) = 0$. Then (\ref{interpretation, def of the shadow seq again}) can be equivalently written as
\begin{equation}
    \tilde{x}_{k}^{i+1} \coloneqq \tilde{x}_{k}^{i} - \eta_{k}(\nabla f_{\pi_{k}^{i}}(\tilde{x}_{k}^{i}+(x_{k}^{i}-\tilde{x}_{k}^{i}))-\nabla f_{\pi_{k}^{i}}(x_{*})+\nabla f(x_{*})).
\end{equation}
Interestingly, this iterative scheme can be regarded as an inexact variant of the SVRG-type methods \cite{johnson2013accelerating, allen2018katyusha, lan2019unified, liu2026relating} with the checkpoint (a.k.a. the full gradient evaluation point) replaced by the solution $x_{*}$ and the random reshuffling sampling strategy.

Some researchers argued that the theoretical advantage of RR over SGD is (at least partially) related to an implicit variance reduction effect of RR \cite{gurbuzbalaban2021random, mishchenko2020random}. Our proof reveals and utilizes the fact that the last iterate of RR coincides with the last iterate of the shadow sequence, which can be regarded as an iterate sequence produced by performing an inexact RR on an auxiliary interpolated problem or performing an inexact variant of random reshuffled SVRG on the original problem. The implicit variance reduction effect of RR is arguably supported and made clearer by our analysis.

\clearpage

\appendix

\section{Technical lemmas for the proof of Theorem \ref{theorem, the hard bound}}

\subsection{An auxiliary interpolated problem, the shadow sequence, and their properties}

\begin{definition}
\label{def, basics}
    For notational convenience, define $q_{i}\coloneqq \nabla f_{i}(x_{*}), \forall i\in [n]$,
    \begin{equation*}
        b_{i}(x)\coloneqq f_{i}(x)-f_{i}(x_{*})-\langle \nabla f_{i}(x_{*}),x-x_{*}\rangle, \ B(x)\coloneqq f(x)-f(x_{*}).
    \end{equation*}
    Also, $\Omega_{i}(x)\coloneqq b_{i}(x)-B(x)$. In the $k^{\rm{th}}$ epoch ($k\in \mathbb{N}^{+}$), define
    \begin{equation*}
        d_{k}^{i}\coloneqq \eta_{k}\sum_{j=1}^{i-1} q_{\pi_{k}^{j}}, \ u_{k}^{i}\coloneqq x_{k}^{i}+d_{k}^{i}, \forall i\in [n+1],
    \end{equation*}
    where we adopt the convention that $\sum_{j=1}^{0} q_{\pi_{k}^{j}}\coloneqq 0$.
\end{definition}

As the notation indicates, $b_{i}$ and $B$ are Bregman divergences associated with functions $f_{i}$ and $f$, respectively. Since the component functions $f_{i}$ are convex, all $b_{i}$ and $B$ are nonnegative. Note that functions $B(x)$ and $b_{i}(x)$ (where $i\in [n]$) share the same property that they achieve the minimal function value at $x_{*}$ and the minimum is $0$. These properties lead to inequalities that will be frequently used in the proof and are summarized as follows.

\begin{lemma}
\label{lemma, special func property}
    Let $h:\mathbb{R}^{d}\rightarrow \mathbb{R}$ be a convex and $L-$smooth function which satisfies $\nabla h(x_{*})=0$ and $h(x_{*})=0$, where $x_{*}\in\mathbb{R}^{d}$. Then for any $x,y\in \mathbb{R}^{d}$,
    \begin{align}
    \label{lemma, special func property, eq 1}
        \Vert \nabla h(x)\Vert^2 &\leq 2Lh(x), \\
    \label{lemma, special func property, eq 2}
        \left|h(x)-h(y)\right| &\leq \Vert \nabla h(y)\Vert \Vert x-y\Vert + \frac{L}{2}\Vert x-y\Vert^2.
    \end{align}
    Specifically, by Definition \ref{def, basics}, the functions $B$ and $b_{i}, i\in [n]$ satisfy (\ref{lemma, special func property, eq 1}) and (\ref{lemma, special func property, eq 2}).
\end{lemma}
\begin{proof}
    By convexity and $L-$smoothness of $h$,
    \begin{equation*}
        \Vert \nabla h(x)\Vert^2 \leq 2L\left(h(x)-h(x_{*}) \right) = 2Lh(x),
    \end{equation*}
    where the last equality is by $h(x_{*})=0$. This proves (\ref{lemma, special func property, eq 1}). By $L-$smoothness of $h$, we have
    \begin{equation}
    \label{lemma, special func property, eq 3}
        h(x)-h(y) \leq \langle \nabla h(y),x-y\rangle + \frac{L}{2}\Vert x-y\Vert^2 \leq \Vert \nabla h(y)\Vert \Vert x-y\Vert + \frac{L}{2}\Vert x-y\Vert^2.
    \end{equation}
    By convexity of $h$, we have
    \begin{equation}
    \label{lemma, special func property, eq 4}
        h(y)-h(x)\leq \langle \nabla h(y), y-x\rangle \leq \Vert \nabla h(y)\Vert \Vert x-y\Vert.
    \end{equation}
    Combining (\ref{lemma, special func property, eq 3}) and (\ref{lemma, special func property, eq 4}) gives equation (\ref{lemma, special func property, eq 2}).
\end{proof}

Some properties of the sequence $\{u_{k}^{i}\}_{i\in [n+1]}$ are listed as follows.

\begin{lemma}
    For any fixed $k\in\mathbb{N}^{+}$, we have 
    \begin{align}
    \label{virtual seq, property 1}
        u_{k}^{1} &= x_{k}^{1}, \  u_{k}^{n+1} = x_{k}^{n+1}, \\
    \label{virtual seq, property 2}
        u_{k}^{i+1} &= u_{k}^{i} - \eta_{k}\nabla b_{\pi_{k}^{i}}(u_{k}^{i}-d_{k}^{i}).
    \end{align}
\end{lemma}

\begin{proof}
    Since $d_{k}^{1}=0$ and $d_{k}^{n+1}=\eta_{k} \sum_{j=1}^{n} q_{\pi_{k}^{j}}=\eta_{k}\sum_{i=1}^{n}\nabla f_{i}(x_{*})=0$, equation (\ref{virtual seq, property 1}) holds. By Definition \ref{def, basics},
    \begin{align*}
        u_{k}^{i+1} &= x_{k}^{i+1}+\eta_{k}\sum_{j=1}^{i}q_{\pi_{k}^{j}} = x_{k}^{i} - \eta_{k}\nabla f_{\pi_{k}^{i}}(x_{k}^{i})+\eta_{k}\sum_{j=1}^{i}q_{\pi_{k}^{j}} \\
        &= u_{k}^{i}-\eta_{k}\sum_{j=1}^{i-1}q_{\pi_{k}^{j}}- \eta_{k}\nabla f_{\pi_{k}^{i}}(x_{k}^{i})+\eta_{k}\sum_{j=1}^{i}q_{\pi_{k}^{j}} \\
        &= u_{k}^{i}-\eta_{k} (\nabla f_{\pi_{k}^{i}}(x_{k}^{i})-\nabla f_{\pi_{k}^{i}}(x_{*})) \\
        &= u_{k}^{i}-\eta_{k}\nabla b_{\pi_{k}^{i}}(x_{k}^{i})  \\
        &= u_{k}^{i}-\eta_{k}\nabla b_{\pi_{k}^{i}}(u_{k}^{i}-d_{k}^{i}).
    \end{align*}
\end{proof}

\subsection{The lemma for sampling without replacement}

\begin{lemma}
\label{lemma, SWR lemma}
    Let $X_{1},\cdots, X_{n}\in \mathbb{R}^{d}$ be fixed vectors with $\sum_{i=1}^{n}X_{i}=0$. Given a fixed $p\in [n]$, let $X_{\pi^{1}},\cdots, X_{\pi^{p}}$ be sampled uniformly without replacement from $\{X_{1},\cdots, X_{n}\}$. Then\footnote{This lemma is essentially due to \cite{mishchenko2020random}.}
    \begin{equation}
    \label{SWR lemma, eq 1}
        \mathbb{E}\left\Vert \sum_{j=1}^{p}X_{\pi^{j}}\right\Vert^2 = \frac{p(n-p)}{n-1} \frac{1}{n}\sum_{i=1}^{n} \Vert X_{i}\Vert^2.
    \end{equation}
    Specifically, by Definition \ref{def, basics},
    \begin{equation}
    \label{SWR lemma, eq 2}
        \mathbb{E}\Vert d_{k}^{i}\Vert^2 = \eta_{k}^{2} \frac{(i-1)(n-i+1)}{n-1}\sigma_{*}^{2},
    \end{equation}
    where $\sigma_{*}^{2}\coloneqq n^{-1}\sum_{i=1}^{n}\Vert \nabla f_{i}(x_{*})\Vert^2$.
\end{lemma}

\begin{proof}
    Denote $\sigma^{2}\coloneqq n^{-1}\sum_{i=1}^{n}\Vert X_{i}\Vert^2$. Then for any fixed $j\in [n]$, we have $\mathbb{E}\Vert X_{\pi^{j}}\Vert^2=\sigma^{2}$. For $r,s\in [n]$ and $r\neq s$, 
    \begin{equation*}
        \mathbb{E} \langle X_{\pi^{r}},X_{\pi^{s}} \rangle = \frac{1}{n(n-1)}\sum_{i\neq j} \langle X_{i},X_{j}\rangle = -\frac{\sigma^{2}}{n-1},
    \end{equation*}
    where the last equality holds since $0=\Vert \sum_{i}X_{i}\Vert^2=\sum_{i}\Vert X_{i}\Vert^2+\sum_{i\neq j}\langle X_{i},X_{j}\rangle$. Hence
    \begin{equation*}
        \mathbb{E}\Vert \sum_{j=1}^{p}X_{\pi^{j}}\Vert^2 = p\sigma^{2}-p(p-1)\frac{\sigma^{2}}{n-1} = \frac{p(n-p)}{n-1} \sigma^{2},
    \end{equation*}
    and this proves (\ref{SWR lemma, eq 1}). By Definition \ref{def, basics} and (\ref{SWR lemma, eq 1}), equation (\ref{SWR lemma, eq 2}) holds.
\end{proof}

\subsection{A one-step descent inequality}

\begin{lemma}
\label{lemma, one step lemma}
    Consider $x=u-d$, $u^{+}=u-\alpha \nabla b_{i}(x)$, where $i\in [n]$ and $L\alpha\leq 1/3$. Then for any $z\in \mathbb{R}^{d}$ and $\rho \in (0,1/4)$, we have
    \begin{equation}
    \label{one step lemma, eq 1}
        B(u)-B(z)\leq \frac{\Vert u-z\Vert^2-\Vert u^{+}-z\Vert^2}{2\alpha}+\Omega_{i}(z)-\Omega_{i}(x)+(2L\alpha+\rho)b_{i}(z)+2\rho B(u)+\frac{3L}{\rho} \Vert d\Vert^2.
    \end{equation}
\end{lemma}

\begin{proof}
    Since $u^{+}=u-\alpha \nabla b_{i}(x)$,
    \begin{equation*}
        \Vert u^{+}-z\Vert^2 = \Vert u-\alpha \nabla b_{i}(x)-z\Vert^2=\Vert u-z\Vert^2-2\alpha \langle \nabla b_{i}(x),u-z\rangle+\alpha^2 \Vert \nabla b_{i}(x)\Vert^2.
    \end{equation*}
    By convexity and $L$-smoothness of $b_{i}$,
    \begin{equation*}
        b_{i}(x)-b_{i}(z)\leq \langle \nabla b_{i}(x),x-z\rangle-\frac{1}{2L}\Vert \nabla b_{i}(x)-\nabla b_{i}(z)\Vert^2.
    \end{equation*}
    Combining the above two equations gives
    \begin{align*}
        b_{i}(x)-b_{i}(z) &\leq \frac{\Vert u-z\Vert^2-\Vert u^{+}-z\Vert^2}{2\alpha}-\langle \nabla b_{i}(x),d\rangle + \frac{\alpha}{2}\Vert \nabla b_{i}(x)\Vert^2 \\
        &\quad - \frac{1}{2L}\Vert \nabla b_{i}(x)-\nabla b_{i}(z)\Vert^2.
    \end{align*}
    Since $b_{i}(x)=B(x)+\Omega_{i}(x)$ and $b_{i}(z)=B(z)+\Omega_{i}(z)$, we have
    \begin{align*}
        B(x)-B(z) &\leq \frac{\Vert u-z\Vert^2-\Vert u^{+}-z\Vert^2}{2\alpha}-\langle \nabla b_{i}(x),d\rangle+\Omega_{i}(z)-\Omega_{i}(x) \\ 
        & \quad-\frac{1}{2L}\Vert \nabla b_{i}(x)-\nabla b_{i}(z)\Vert^2+\frac{\alpha}{2}\Vert \nabla b_{i}(x)\Vert^2.
    \end{align*}
    By $L$-smoothness of $B$, we have $B(u)-B(x)\leq \langle \nabla B(x),d\rangle+\frac{L}{2}\Vert d\Vert^2$. Then 
    \begin{equation}
    \label{one step lemma, eq 2}
        B(u)-B(z) \leq \frac{\Vert u-z\Vert^2-\Vert u^{+}-z\Vert^2}{2\alpha} + \Omega_{i}(z)-\Omega_{i}(x)+R,
    \end{equation}
    where
    \begin{equation*}
        R \coloneqq -\langle \nabla b_{i}(x)-\nabla B(x),d\rangle + \frac{\alpha}{2}\Vert \nabla b_{i}(x)\Vert^2 + \frac{L}{2}\Vert d\Vert^2 - \frac{1}{2L}\Vert \nabla b_{i}(x)-\nabla b_{i}(z)\Vert^2.
    \end{equation*}
    We bound the inner product term and certain non-negative terms in $R$ as follows. By Young's inequality, we have
    \begin{align*}
        -\langle \nabla b_{i}(x)-\nabla B(x),d\rangle &= -\langle (\nabla b_{i}(x)-\nabla b_{i}(z))+(\nabla b_{i}(z)-\nabla B(x)),d\rangle \\
        &\leq \frac{1}{16L} \Vert \nabla b_{i}(x)-\nabla b_{i}(z)\Vert^2 + 4L\Vert d\Vert^2 + \frac{L}{\rho} \Vert d\Vert^2 \\
        & \quad + \frac{\rho}{4L}\Vert \nabla b_{i}(z)-\nabla B(x)\Vert^2 \\
        &\leq \frac{1}{16L} \Vert \nabla b_{i}(x)-\nabla b_{i}(z)\Vert^2 + 4L\Vert d\Vert^2 + \frac{L}{\rho} \Vert d\Vert^2 \\
        & \quad + \frac{\rho}{2L}\Vert \nabla b_{i}(z)\Vert^2 + \frac{\rho}{2L}\Vert \nabla B(x)\Vert^2 \\
        &\leq \frac{1}{16L} \Vert \nabla b_{i}(x)-\nabla b_{i}(z)\Vert^2 + 4L\Vert d\Vert^2 + \frac{L}{\rho} \Vert d\Vert^2 \\
        & \quad + \rho b_{i}(z) + \rho B(x),
    \end{align*}
    where the last inequality is by Lemma \ref{lemma, special func property}. To further bound the $\rho B(x)$ term, note that
    \begin{align*}
        B(x) &\leq B(u)+\langle \nabla B(u),-d\rangle + \frac{L}{2}\Vert d\Vert^2 \\
        &\leq B(u) + \frac{\rho}{2L}\Vert \nabla B(u)\Vert^2 + \frac{L}{2\rho}\Vert d\Vert^2 + \frac{L}{2}\Vert d\Vert^2 \\
        &\leq (1+\rho)B(u) + \frac{L}{\rho} \Vert d\Vert^2,
    \end{align*}
    where the last inequality is by Lemma \ref{lemma, special func property} and the setting $\rho\in (0,1/4)$. Also note that 
    \begin{equation*}
        \frac{\alpha}{2}\Vert \nabla b_{i}(x)\Vert^2 \leq \alpha \Vert \nabla b_{i}(x)-\nabla  b_{i}(z)\Vert^2 + 2L\alpha b_{i}(z),
    \end{equation*}
    where Lemma \ref{lemma, special func property} is applied. Combining the above inequalities gives
    \begin{align*}
        R &\leq (2L\alpha+\rho)b_{i}(z) + 2\rho B(u) + (4L+L/\rho+L+L/2)\Vert d\Vert^2 \\
        & \quad + (\alpha+1/(16L)-1/(2L))\Vert \nabla b_{i}(x)-\nabla b_{i}(z)\Vert^2 \\
        &\leq (2L\alpha+\rho)b_{i}(z) + 2\rho B(u) + \frac{3L}{\rho}\Vert d\Vert^2,
    \end{align*}
    where the last inequality is by $\rho\in (0,1/4)$ and $L\alpha \leq 1/3$. Plugging the above inequality into (\ref{one step lemma, eq 2}) gives the desired equation (\ref{one step lemma, eq 1}).
\end{proof}

\subsection{Two indexing conventions, a scalar sequence, and its properties}

We will adopt two perspectives in the subsequent proof. From the first perspective, the algorithm is flattened with the iterate  indexed by $t\in\mathbb{N}^{+}$. From the second perspective, the algorithm is nested with the iterate indexed by a pair $(k,i)$ where $k\in\mathbb{N}^{+}$ and $i\in [n]$. Lemma \ref{lemma, one step lemma} adopts the first perspective, and later we will switch to the other to facilitate certain estimates by interchanging the order of summation. Hence the correspondence of the two index systems is needed.

Also note that we will apply Lemma \ref{lemma, one step lemma} to each iterate of the sequence $\{u_{k}^{i}\}_{k,i}$ until the algorithm is run for a prescribed number of iterations. In the last application of Lemma \ref{lemma, one step lemma}, the desired last iterate $u_{\rm{obj}}$ should have a successor $u_{\rm{obj}}^{+}$ but the (virtual) algorithm ends at $u_{\rm{obj}}$. To formally apply Lemma \ref{lemma, one step lemma} to $u_{\rm{obj}}$, we need to define a ghost iteration starting at $u_{\rm{obj}}$ which is not performed by the (virtual) algorithm.

Motivated by the above needs, we make the following definition.

\begin{definition}
\label{def, global index}
    Assume the algorithm is run for $K\in\mathbb{N}^{+}$ epochs. Denote $T\coloneqq nK$. For any $t\in [T]$, there exists a unique pair $(k,i)$ such that $t=(k-1)n+i$ where $k\in\mathbb{N}^{+}, i\in [n]$. Based on this representation, define
    \begin{equation*}
        \alpha_{t}\coloneqq \eta_{k}, \ i_{t}\coloneqq \pi_{k}^{i}, \ x_{t}\coloneqq x_{k}^{i}, \ d_{t}\coloneqq d_{k}^{i}, \ u_{t}\coloneqq u_{k}^{i}.
    \end{equation*}
    Then (\ref{virtual seq, property 2}) can be written as 
    \begin{equation*}
        u_{t+1} = u_{t} - \alpha_{t}\nabla b_{i_{t}}(x_{t}), \ t\in [T].
    \end{equation*}
    We also denote $\tau_{k,i}\coloneqq (k-1)n+i$ for $k\in [K], i\in [n]$.
    
    Denoting $M\coloneqq T+1$, define\footnote{The following $\beta$ notation is introduced to emphasize that the stepsize $\alpha_{M}$ corresponds to a ghost iteration which does not actually happen in the optimization process.}
    \begin{equation*}
        \alpha_{M}\coloneqq \eta_{1}, \ i_{M}\sim \operatorname{Unif}[n], \ u_{M+1}\coloneqq u_{M}-\alpha_{M} \nabla b_{i_{M}}(x_{M}), \ \beta\coloneqq \alpha_{M} = \eta_{1}.
    \end{equation*}
    Since $x_{M}=u_{M}$, we have $d_{M}=0$ and $u_{M+1}=u_{M}-\eta_{1}\nabla b_{i_{M}}(u_{M})$.
\end{definition}

Following the proof strategy of \cite{zamani2025exact}, the point $z$ in Lemma \ref{lemma, one step lemma} will be defined as a convex combination of $x_{*}$ and elements of $\{u_{t}\}_{t}$. Coefficients of this convex combination will be determined by the sequence $\{v_{t}\}_{t}$ defined as follows.

\begin{definition}
\label{def, auxiliary seq}
    Define
    \begin{align*}
        H_{t} &\coloneqq \sum_{r=t}^{M}\alpha_{r}, t\in [M]; \\
        \Lambda &\coloneqq 1+\log M, \ \lambda \coloneqq \Lambda^{-1}, \ a\coloneqq 1-\lambda.
    \end{align*}
    With these notations, define 
    \begin{equation*}
        v_{M+1}\coloneqq 1, \ v_{t}\coloneqq v_{t+1}\left(1-a\frac{\alpha_{t}}{H_{t}} \right),  \ \Delta v_{t}\coloneqq v_{t+1}-v_{t}=a\frac{\alpha_{t}}{H_{t}}v_{t+1}.
    \end{equation*}
    We also define $r_{t}\coloneqq \alpha_{t}/H_{t}$. Direct calculation gives
    \begin{equation*}
        \alpha_{t}v_{t+1}-\Delta v_{t}H_{t} = \lambda \alpha_{t} v_{t+1}.
    \end{equation*} 
\end{definition}

To prove certain important properties of the scalar sequence $\{v_{t}\}_{t}$, we typically encounter a target inequality with multiple elements of $\{v_{t}\}_{t}$ in it, which is troublesome to handle due to the recursive definition of $\{v_{t}\}_{t}$. 

Without loss of generality, suppose the target inequality contains $v_{t_{1}}, v_{t_{1}+1}, \cdots, v_{t_{2}}$. Informally, to make the inequality less implicit, one possible approach is to write
\begin{equation}
\label{appendix eq, connect two v's}
    v_{t^{\prime}} = v_{t^{\prime}+1} (1-ar_{t^{\prime}}) = \cdots = v_{t_{2}} \prod_{j=t^{\prime}}^{t_{2}-1}(1-ar_{j})
\end{equation}
for all $t_{1}\leq t^{\prime}<t_{2}$ with the hope that we can treat $v_{t_{2}}$ as a constant. An immediate problem is that $\prod_{j=t^{\prime}}^{t_{2}-1}(1-ar_{j})$ does not seem to admit a compact formulation. As a remedy, a helpful technical observation is that the product of $1-r_{j}$ telescopes as follows
\begin{equation*}
    \prod_{j=t^{\prime}}^{t_{2}-1} (1-r_{j}) = \prod_{j=t^{\prime}}^{t_{2}-1} (1-\frac{\alpha_{j}}{H_{j}}) = \prod_{j=t^{\prime}}^{t_{2}-1} \frac{H_{j+1}}{H_{j}} = \frac{H_{t_{2}}}{H_{t^{\prime}}}.
\end{equation*}
Hence we need certain inequalities connecting $1-ar_{j}$ and $1-r_{j}$ to make the aforementioned idea work. Such elementary tools are provided in the following lemma.

\begin{lemma}
\label{elementary ineq lemma}
    Assume that $a\in (0,1)$ and $r\in [0,1/2]$. Then 
    \begin{align}
    \label{elementary ineq lemma, eq 1}
        \log(1-a r) &\leq a\log(1-r)+2\lambda r, \\
    \label{elementary ineq lemma, eq 2}
        \log(1-a r) &\leq a\log(1-r)+\frac{r^{2}}{2(1-r)},  \\
    \label{elementary ineq lemma, eq 3}
        \log(1-a r) &\geq a\log(1-r).
    \end{align}
\end{lemma}
\begin{proof}
    To prove (\ref{elementary ineq lemma, eq 1}), define
    \begin{equation*}
        h_{1}(r)\coloneqq a\log(1-r)-\log (1-a r)+2\lambda r, \ r\in (0,1).
    \end{equation*}
    Direct calculation gives
    \begin{align*}
        h_{1}^{'}(r) &= (1-a r)^{-1}(1-r)^{-1}\left[2\lambda (1-r-a r+a r^{2})+a r(a-1) \right] \\
        &= (1-a r)^{-1}(1-r)^{-1}(1-a)\left[2(a r^{2}-(a+1)r+1)-a r \right] \\
        &= (1-a r)^{-1}(1-r)^{-1}(1-a)\left[2 a r^{2}-(3a+2)r+2\right],
    \end{align*}
    where the second equality is by $a=1-\lambda$. The sign of $h_{1}^{\prime}(r)$ is the same as that of the function $h_{1,1}(r)\coloneqq 2ar^{2}-(3a+2)r+2$. The minimizer of $h_{1,1}(r)$ is $(3a+2)/(4a)\in \left[1,+\infty\right)$, hence the minimum of $h_{1,1}(r)$ over $r\in [0,1/2]$ is $h_{1,1}(1/2)=1-a>0$. Hence $h_{1}^{\prime}(r)>0$. Noting that $h_{1}(0)=0$, equation (\ref{elementary ineq lemma, eq 1}) is proved.

    To prove (\ref{elementary ineq lemma, eq 2}), first note that
    \begin{equation*}
        \log(1-a r)-a\log(1-r)\leq -a r - a\log(1-r)\leq -\log(1-r)-r,
    \end{equation*}
    where the inequalities are by repeatedly using $\log(1-r)\leq -r, r\in (0,1)$ and the fact that $a\in (0,1)$. Define
    \begin{equation*}
        h_{2}(r)=2r\log r -r^{2}+1, \ r\in (0,1).
    \end{equation*}
    Direct calculation gives
    \begin{equation*}
        h_{2}^{\prime}(r)=2(\log r+1-r ) = 2(\log(1-(1-r))+(1-r))\leq 0,
    \end{equation*}
    hence $h_{2}(r)$ is decreasing w.r.t. $r$. Since $\lim_{r\rightarrow 1^{-}}h_{2}(r)=0$, we have $h_{2}(r)\geq 0, r\in (0,1)$. Recall the definition of $h_{2}(r)$ and replace $r$ by $1-r$, the equation (\ref{elementary ineq lemma, eq 2}) is proved.

    To prove (\ref{elementary ineq lemma, eq 3}), define
    \begin{equation*}
        h_{3}(r)\coloneqq \log(1-a r) - a\log(1-r), \ r\in (0,1).
    \end{equation*}
    Direct calculation gives
    \begin{align*}
        h_{3}^{'}(r) &= (1-a r)^{-1}(1-r)^{-1}\left[-a(1-r)+a(1-a r) \right] \\
        &= a r (1-a r)^{-1}(1-r)^{-1}(1-a),
    \end{align*}
    and it follows that $h_{3}^{'}(r)>0$. Since $h_{3}(0)=0$, then equation (\ref{elementary ineq lemma, eq 3}) is proved.
\end{proof}

With Lemma \ref{elementary ineq lemma}, the product in (\ref{appendix eq, connect two v's}) admits an estimate with a compact form, which leads to a tractable inequality connecting two elements of $\{v_{t}\}_{t}$ presented as follows.

\begin{lemma}
\label{lemma, aux seq property 1}
    Assume the algorithm is run for $K\in\mathbb{N}^{+}$ epochs. Given a fixed epoch $k\in [K]$,  denote $\mathcal{R}_{k}\coloneqq n A_{k+1}+\beta$ where $A_{k}\coloneqq \sum_{j=k}^{K}\eta_{j}$. For any $r,j\in [n]$ and $r\geq j$, we have
    \begin{equation}
    \label{lemma, aux seq property 1, eq 1}
        \frac{1}{2}\left(\frac{\mathcal{R}_{k}+(n-j)\eta_{k}}{\mathcal{R}_{k}+(n-r)\eta_{k}}\right)^{a}\leq \frac{v_{\tau_{k,r}+1}}{v_{\tau_{k,j}+1}}\leq \left(\frac{\mathcal{R}_{k}+(n-j)\eta_{k}}{\mathcal{R}_{k}+(n-r)\eta_{k}}\right)^{a}.
    \end{equation}
\end{lemma}
\begin{proof}
    For any $t=(k-1)n+i$ where $k\in [K]$ and $i\in [n]$, by Definition \ref{def, auxiliary seq},
    \begin{equation*}
        r_{t} = \alpha_{t}/H_{t} \leq \alpha_{t}/(\alpha_{t}+\beta) = \eta_{k}/(\eta_{k}+\eta_{1})\leq 1/2,
    \end{equation*}
    hence $r_{t}\in [0,1/2]$. Then Lemma \ref{elementary ineq lemma} is applicable with the variable $r$ replaced by $r_{t}$, $t\in [T]$. By Definition \ref{def, auxiliary seq},
    \begin{align*}
        \log \frac{v_{\tau_{k,r}+1}}{v_{\tau_{k,j}+1}} &= \log \frac{v_{\tau_{k,r}+1}}{v_{\tau_{k,r}+1}\prod_{s=j+1}^{r}(1-a r_{\tau_{k,s}})} = - \sum_{s=j+1}^{r}\log(1-a r_{\tau_{k,s}}) \\
        &\leq -a \sum_{s=j+1}^{r}\log(1-r_{\tau_{k,s}}) = -a\log \prod_{s=j+1}^{r}\frac{H_{\tau_{k,s}+1}}{H_{\tau_{k,s}}} \\
        &= -a\log \frac{H_{\tau_{k,r}+1}}{H_{\tau_{k,j+1}}} = -a\log \frac{\mathcal{R}_{k}+(n-r)\eta_{k}}{\mathcal{R}_{k}+(n-j)\eta_{k}},
    \end{align*}
    where the first inequality is by (\ref{elementary ineq lemma, eq 3}), the third equality is by $1-r_{t}=H_{t+1}/H
    _{t}$. This proves the latter part of equation (\ref{lemma, aux seq property 1, eq 1}). On the other hand, for $j\in [n-1]$, we have
    \begin{align*}
        \log \frac{v_{\tau_{k,r}+1}}{v_{\tau_{k,j}+1}} &= \log \frac{v_{\tau_{k,r}+1}}{v_{\tau_{k,r}+1}\prod_{s=j+1}^{r}(1-a r_{\tau_{k,s}})} = \log\prod_{s=j+1}^{r}(1-a r_{\tau_{k,s}})^{-1} \\
        &= -\sum_{s=j+1}^{r}\log (1-ar_{\tau_{k,s}}) \geq -a\sum_{s=j+1}^{r}\log(1-r_{\tau_{k,s}}) - \frac{1}{2}\sum_{s=j+1}^{r}\frac{r_{\tau_{k,s}}^{2}}{1-r_{\tau_{k,s}}} \\
        &= -a\log \frac{\mathcal{R}_{k}+(n-r)\eta_{k}}{\mathcal{R}_{k}+(n-j)\eta_{k}} - \frac{1}{2}\sum_{s=j+1}^{r} \frac{1}{(\mathcal{R}_{k}/\eta_{k}+n-s+1)(\mathcal{R}_{k}/\eta_{k}+n-s)} \\
        &= -a\log \frac{\mathcal{R}_{k}+(n-r)\eta_{k}}{\mathcal{R}_{k}+(n-j)\eta_{k}} - \frac{1}{2}\sum_{s=j+1}^{r} \left(\frac{1}{\mathcal{R}_{k}/\eta_{k}+n-s}-\frac{1}{\mathcal{R}_{k}/\eta_{k}+n-s+1} \right) \\
        &\geq -a\log \frac{\mathcal{R}_{k}+(n-r)\eta_{k}}{\mathcal{R}_{k}+(n-j)\eta_{k}} - \frac{1}{2\mathcal{R}_{k}/\eta_{k}} \geq -a\log \frac{\mathcal{R}_{k}+(n-r)\eta_{k}}{\mathcal{R}_{k}+(n-j)\eta_{k}}-\frac{1}{2},
    \end{align*}
    where the first inequality is by (\ref{elementary ineq lemma, eq 2}), the fourth equality is by
    \begin{equation*}
        r_{\tau_{k,s}}=\frac{\alpha_{\tau_{k,s}}}{H_{\tau_{k,s}}}=\frac{\eta_{k}}{\mathcal{R}_{k}+(n-s+1)\eta_{k}} = \frac{1}{\mathcal{R}_{k}/\eta_{k}+n-s+1},
    \end{equation*}
    the second inequality is by $j\in [n-1]$, and the last inequality is by $\mathcal{R}_{k}/\eta_{k}=nA_{k+1}/\eta_{k}+\beta/\eta_{k}\geq nA_{k+1}/\eta_{k}+1\geq 1$. This proves the former part of equation (\ref{lemma, aux seq property 1, eq 1}) for $j\in [n-1]$. When $j=n$, we have $r=j=n$ (since $r\geq j$), then the former part of equation (\ref{lemma, aux seq property 1, eq 1}) holds vacuously since $\{v_{t}\}_{t}$ is increasing w.r.t. $t$.
\end{proof}

With Lemma \ref{lemma, aux seq property 1}, we can prove the following properties of $\{v_{t}\}_{t}$. We provide an overview of the proofs to facilitate understanding. Proving the first two equations reduces to relating the elements $v_{1}$ and $v_{t+1}$ to $v_{M}$ (or equivalently, $v_{M+1}$). Proving the third equation reduces to relating $v_{\tau_{k,j}}$ to $v_{\tau_{k,r}+1}, r\in\{j,j+1,\cdots, n\}$. For the last equation, by relating the $v_{t^{\prime}}$'s in the LHS to $v_{\tau_{k,1}+1}$ and relating the $v_{t^{\prime\prime}}$'s in the RHS to $v_{\tau_{k,1}+1}$ respectively, we can  derive a more tractable upper bound of the LHS and a lower bound of the RHS. Proving the last equation then reduces to proving the aforementioned upper bound does not exceed the lower bound.

\begin{lemma}
\label{lemma, aux seq property}
    The sequence $\{v_{t}\}_{t}$ has the following properties:
    \begin{align}
    \label{lemma, aux seq property, eq 1}
        \frac{v_{1}}{\lambda \beta} \leq \frac{C}{H_{1}}, \ &\frac{v_{t+1}}{\lambda \beta} \leq \frac{C}{H_{t}}, \ t\in [T], \\
    \label{lemma, aux seq property, eq 2}
        \frac{1}{n}\sum_{r=j}^{n} v_{\tau_{k,r}+1} &\leq C\Lambda v_{\tau_{k,j}+1}, \ j\in [n], \\
    \label{lemma, aux seq property, eq 3}
        \sum_{j=1}^{n}v_{\tau_{k,j}+1}(j-1) &\leq C\Lambda \sum_{j=1}^{n} v_{\tau_{k,j}+1}\frac{(j-1)(n-j+1)}{n-1},
    \end{align}
    where $C$ is an absolute numerical constant defined as follows
    \begin{equation*}
        C \coloneqq \max \left\{3e^{3},2,50 \right\} = 3e^{3}.
    \end{equation*}
\end{lemma}
\begin{proof}
    By Definition \ref{def, auxiliary seq}, $v_{M}=v_{M+1}(1-a \alpha_{M}/H_{M})=\lambda v_{M+1}=\lambda$. Then
    \begin{equation}
    \label{lemma, aux seq property, eq 4}
        v_{t} = v_{t+1}(1-a r_{t}) = v_{M} \prod_{s=t}^{M-1}(1- a r_{s}) = \lambda \prod_{s=t}^{M-1}(1-a r_{s}).
    \end{equation}
    We have
    \begin{equation}
    \label{lemma, aux seq property, eq 5}
        \log\left(\prod_{s=t}^{M-1}(1-a r_{s}) \right) = \sum_{s=t}^{M-1}\log(1-a r_{s})\leq a\sum_{s=t}^{M-1}\log(1-r_{s})+2\lambda \sum_{s=t}^{M-1}r_{s},
    \end{equation}
    where the inequality is by (\ref{elementary ineq lemma, eq 1}). For the first term in the RHS of (\ref{lemma, aux seq property, eq 5}), we have
    \begin{equation}
    \label{lemma, aux seq property, eq 6}
        a\sum_{s=t}^{M-1}\log(1-r_{s})\leq a\log\left(\prod_{s=t}^{M-1}\frac{H_{s+1}}{H_{s}} \right) = \log\left(\frac{\beta}{H_{t}} \right)^{a}.
    \end{equation}
    For the second term in the RHS of (\ref{lemma, aux seq property, eq 5}), we have
    \begin{equation}
    \label{lemma, aux seq property, eq 7}
        2\lambda \sum_{s=t}^{M-1} r_{s} \leq 2\lambda \sum_{s=t}^{M-1}-\log (1-r_{s}) = 2\lambda \log \left(\frac{H_{t}}{\beta}\right),
    \end{equation}
    where the inequality is by $r\leq -\log(1-r), r\in [0,1/2]$. By (\ref{lemma, aux seq property, eq 4}), (\ref{lemma, aux seq property, eq 5}), (\ref{lemma, aux seq property, eq 6}), and (\ref{lemma, aux seq property, eq 7}),
    \begin{equation}
    \label{lemma, aux seq property, eq 8}
        v_{t} \leq \lambda \left(\frac{\beta}{H_{t}} \right)^{a} \exp \left(2\lambda \log\frac{H_{t}}{\beta} \right).
    \end{equation}
    Since the stepsize sequence $\{\eta_{k}\}_{k}$ is non-increasing, we have $H_{t}/\beta = H_{t}/\eta_{1}\leq M$. By Definition \ref{def, auxiliary seq}, $\lambda \log M = (1+\log M)^{-1}\log M\leq 1$. Hence the exponential term in (\ref{lemma, aux seq property, eq 8}) is upper bounded by $e^{2}$. Then
    \begin{equation}
    \label{lemma, aux seq property, eq 9}
        v_{t} \leq e^{2} \lambda \left(\frac{\beta}{H_{t}} \right)^{a}.
    \end{equation}
    When $t=1$, (\ref{lemma, aux seq property, eq 9}) gives
    \begin{equation}
    \label{lemma, aux seq property, eq 10}
        \frac{v_{1}}{\lambda \beta}\leq \frac{e^{2}}{H_{1}}\frac{H_{1}}{\beta} \left(\frac{\beta}{H_{1}} \right)^{a} = \frac{e^{2}}{H_{1}} \left(\frac{H_{1}}{\beta} \right)^{\lambda}\leq \frac{e^{2}}{H_{1}} \left(\frac{M\beta}{\beta} \right)^{\lambda }\leq \frac{e^{3}}{H_{1}},
    \end{equation}
    where the last inequality is by $M^{\lambda}\leq e$. Equation (\ref{lemma, aux seq property, eq 9}) also gives
    \begin{equation}
    \label{lemma, aux seq property, eq 11}
        \frac{v_{t+1}}{\lambda \beta}\leq \frac{e^{2}}{H_{t}}\frac{H_{t}}{\beta} \left(\frac{\beta}{H_{t}} \right)^{a} \left(\frac{H_{t}}{H_{t+1}} \right)^{a}.
    \end{equation}
    Note that
    \begin{equation*}
        \frac{H_{t+1}}{H_{t}} = 1-\frac{\alpha_{t}}{H_{t}} = 1-r_{t} \geq \frac{1}{2},
    \end{equation*}
    hence $H_{t}/H_{t+1}\leq 2$. Then (\ref{lemma, aux seq property, eq 11}) gives
    \begin{equation}
    \label{lemma, aux seq property, eq 12}
        \frac{v_{t+1}}{\lambda \beta} \leq 2^{a}e \frac{e^{2}}{H_{t}} \leq \frac{2e^{3}}{H_{t}}.
    \end{equation}
    Hence (\ref{lemma, aux seq property, eq 1}) holds by (\ref{lemma, aux seq property, eq 10}) and (\ref{lemma, aux seq property, eq 12}).
    
    Equation (\ref{lemma, aux seq property, eq 2}) is equivalent to
    \begin{equation*}
        \frac{1}{n}\sum_{r=j}^{n} \frac{v_{\tau_{k,r}+1}}{v_{\tau_{k,j}+1}}\leq C \Lambda.
    \end{equation*}
    By (\ref{lemma, aux seq property 1, eq 1}),
    \begin{equation}
    \label{lemma, aux seq property, eq 13}
        \frac{1}{n}\sum_{r=j}^{n} \frac{v_{\tau_{k,r}+1}}{v_{\tau_{k,j}+1}}\leq \frac{1}{n} \sum_{r=j}^{n} \left(\frac{\mathcal{R}_{k}+(n-j)\eta_{k}}{\mathcal{R}_{k}+(n-r)\eta_{k}} \right)^{a} = \frac{(\mathcal{R}_{k}+(n-j)\eta_{k})^{a}}{n}\sum_{q=0}^{n-j}\left(\mathcal{R}_{k}+q\eta_{k} \right)^{-a}.
    \end{equation}
    Define $h_{3}(x)\coloneqq (\mathcal{R}_{k}+x\eta_{k})^{-a}$. Since $h_{3}(x)$ is decreasing on $[0,+\infty)$, we have
    \begin{equation*}
        \sum_{q=0}^{n-j} h_{3}(q)\leq h_{3}(0)+\int_{0}^{n-j} h_{3}(x) \ {\rm d}x.
    \end{equation*}
    Then
    \begin{align*}
        \sum_{q=0}^{n-j}\left(\mathcal{R}_{k}+q\eta_{k} \right)^{-a} &\leq \mathcal{R}_{k}^{-a}+\int_{0}^{n-j} \left(\mathcal{R}_{k}+x\eta_{k} \right)^{-a}  {\rm d}x \\
        &= \mathcal{R}_{k}^{-a}+\frac{\left(\mathcal{R}_{k}+(n-j)\eta_{k}\right)^{\lambda}-\mathcal{R}_{k}^{\lambda}}{\lambda \eta_{k}}.
    \end{align*}
    Plugging the above inequality into (\ref{lemma, aux seq property, eq 13}) gives
    \begin{equation}
    \label{lemma, aux seq property, eq 14}
        \frac{1}{n}\sum_{r=j}^{n}\frac{v_{\tau_{k,r}+1}}{v_{\tau_{k,j}+1}}\leq I_{1}+I_{2},
    \end{equation}
    where
    \begin{equation*}
        I_{1}\coloneqq \frac{\left(\mathcal{R}_{k}+(n-j)\eta_{k} \right)^{a} \mathcal{R}_{k}^{-a}}{n}, \ I_{2}\coloneqq \frac{\left(\mathcal{R}_{k}+(n-j)\eta_{k} \right)^{a}\left[\left(\mathcal{R}_{k}+(n-j)\eta_{k}\right)^{\lambda}-\mathcal{R}_{k}^{\lambda}\right]}{n\lambda \eta_{k}}.
    \end{equation*}
    We have
    \begin{equation}
    \label{lemma, aux seq property, eq 15}
        I_{1} = \frac{1}{n}\left(1+\frac{(n-j)\eta_{k}}{\mathcal{R}_{k}} \right)^{a} \leq \frac{(1+n-j)^{a}}{n} \leq \frac{n^{a}}{n} \leq 1,
    \end{equation}
    where the first inequality is by $\mathcal{R}_{k}\geq \beta \geq \eta_{k}$. Let $\tilde{\mathcal{R}}_{k}\coloneqq \mathcal{R}_{k}+(n-j)\eta_{k}$, then
    \begin{equation}
    \label{lemma, aux seq property, eq 16}
        I_{2} = \frac{\tilde{R}_{k}^{a+\lambda} \left[1-\left(\mathcal{R}_{k}/\tilde{\mathcal{R}_{k}}\right)^{\lambda} \right]}{n\lambda \eta_{k}} \leq \frac{\tilde{\mathcal{R}_{k}}\left(1-\mathcal{R}_{k}/\tilde{\mathcal{R}_{k}}\right)}{n\lambda \eta_{k}} = \frac{(n-j)\eta_{k}}{n\lambda \eta_{k}} \leq \lambda^{-1} = \Lambda,
    \end{equation}
    where the first inequality is by $\tilde{\mathcal{R}}_{k}=\mathcal{R}_{k}+(n-j)\eta_{k}\geq \mathcal{R}_{k}$ and $\lambda \in (0,1)$, and the second equality is by $\tilde{\mathcal{R}}_{k}-\mathcal{R}_{k}=(n-j)\eta_{k}$. Noting that $1+\Lambda \leq 2\Lambda$, equations (\ref{lemma, aux seq property, eq 14}), (\ref{lemma, aux seq property, eq 15}), and (\ref{lemma, aux seq property, eq 16}) imply (\ref{lemma, aux seq property, eq 2}). 
    
    Dividing both sides of (\ref{lemma, aux seq property, eq 3}) by $v_{\tau_{k,1}+1}$ and using (\ref{lemma, aux seq property 1, eq 1}), one sees that it suffices to prove the following inequality to prove (\ref{lemma, aux seq property, eq 3})
    \begin{equation*}
        \sum_{j=1}^{n} \left(\mathcal{R}_{k}+(n-j)\eta_{k} \right)^{-a}(j-1) \leq 2^{-1}C\Lambda \sum_{j=1}^{n} \left(\mathcal{R}_{k}+(n-j)\eta_{k} \right)^{-a} \frac{(j-1)(n-j+1)}{n-1},
    \end{equation*}
    where the common factor $(\mathcal{R}_{k}+(n-1)\eta_{k})^{a}$ of both sides is canceled and hence omitted. Denote $d_{k}\coloneqq \mathcal{R}_{k}/\eta_{k}$, then the above inequality is equivalent to
    \begin{equation}
    \label{lemma, aux seq property, eq 17}
        \sum_{j=1}^{n} \left(d_{k}+n-j\right)^{-a}(j-1) \leq 2^{-1}C\Lambda \sum_{j=1}^{n} \left(d_{k}+n-j \right)^{-a} \frac{(j-1)(n-j+1)}{n-1}.
    \end{equation}
    For notational convenience, define
    \begin{equation*}
        I_{3}\coloneqq \sum_{j=1}^{n} \left(d_{k}+n-j\right)^{-a}(j-1), \ I_{4}\coloneqq  \sum_{j=1}^{n} \left(d_{k}+n-j \right)^{-a} \frac{(j-1)(n-j+1)}{n-1}.
    \end{equation*}
    We consider two cases to prove (\ref{lemma, aux seq property, eq 17}). In the case where $d_{k}\geq n-1$, 
    we have $d_{k}\leq d_{k}+n-j\leq 2d_{k}$, then
    \begin{equation}
    \label{lemma, aux seq property, eq 18}
        I_{3} \leq d_{k}^{-a}\sum_{j=1}^{n}(j-1)=d_{k}^{-a}\frac{n(n-1)}{2}.
    \end{equation}
    Also note that
    \begin{align*}
        I_{4} &\geq (2d_{k})^{-a}\frac{1}{n-1}\sum_{j=1}^{n} \left[n(j-1)-(j-1)^{2}\right] \\
        &= (2d_{k})^{-a}(n-1)^{-1}\left(\frac{n^{2}(n-1)}{2}-\frac{n(n-1)(2n-1)}{6} 
        \right) \\
        &\geq (1/2) d_{k}^{-a}(n-1)^{-1} \frac{1}{6}n(n-1)(n+1)\geq \frac{d_{k}^{-a}}{6} \frac{n(n-1)}{2} \\
        &\geq \frac{1}{6}\sum_{j=1}^{n} \left(d_{k}+n-j \right)^{-a}(j-1),
    \end{align*}
    where the last inequality is by (\ref{lemma, aux seq property, eq 18}). Hence in this case, we have
    \begin{equation}
    \label{lemma, aux seq property, eq 19}
        I_{3} \leq 6 I_{4}, \ {\rm when} \ d_{k}\geq n-1.
    \end{equation}
    Since $\mathcal{R}_{k}$ contains the summand $\beta=\eta_{1}\geq \eta_{k}$, we always have $d_{k}\geq 1$. Consider the case where $1\leq d_{k}<n-1$. In this case,
    \begin{align*}
        I_{3} &\leq n \sum_{j=1}^{n} (d_{k}+n-j)^{-a} = n d_{k}^{-a} + n\sum_{q=1}^{n-1}(d_{k}+q)^{-a} \\
        &\leq nd_{k}^{-a} + n\int_{0}^{n-1}\left(d_{k}+x \right)^{-a} {\rm d}x = nd_{k}^{-a}+n\frac{(d_{k}+n-1)^{\lambda}-d_{k}^{\lambda}}{\lambda} \\
        &\leq nd_{k}^{-a}+n\Lambda \left(d_{k}+n-1 \right)^{\lambda}.
    \end{align*}
    Note that $(d_{k}+n-j)^{-a}\geq (d_{k}+n-1)^{-a}$ for $j\in [n]$, then
    \begin{align*}
        I_{4} &\geq (d_{k}+n-1)^{-a} \sum_{j=1}^{n} \frac{(j-1)(n-j+1)}{n-1}  \\
        &\geq (d_{k}+n-1)^{-a} \sum_{j=1}^{n} \frac{1}{n-1} \left[n(j-1)-(j-1)^{2}\right] \\
        &= \frac{1}{6}(d_{k}+n-1)^{-a}n(n+1).
    \end{align*}
    To bound $I_{3}$ with $I_{4}$, it suffices to consider the following quantity
    \begin{equation}
    \label{lemma, aux seq property, eq 20}
        \frac{nd_{k}^{-a}+n\Lambda \left(d_{k}+n-1 \right)^{\lambda}}{6^{-1}\left(d_{k}+n-1 \right)^{-a}n(n+1)} = 6\left(I_{5}+I_{6} \right),
    \end{equation}
    where
    \begin{equation}
    \label{lemma, aux seq property, eq 21}
        I_{5}\coloneqq \left(\frac{d_{k}+n-1}{d_{k}} \right)^{a}\frac{1}{n+1}, \ I_{6}\coloneqq \Lambda \frac{d_{k}+n-1}{n+1}.
    \end{equation}
    Note that in this case $1\leq d_{k} < n-1$, we have
    \begin{equation}
    \label{lemma, aux seq property, eq 22}
        I_{5} = \left(1+\frac{n-1}{d_{k}} \right)^{a}\frac{1}{n+1} \leq 2^{a}\frac{\left(n-1 \right)^{a}}{n+1} \leq 2 \leq 2\Lambda.
    \end{equation}
    Also note that
    \begin{equation}
    \label{lemma, aux seq property, eq 23}
        I_{6} \leq \Lambda \frac{2(n-1)}{n+1} \leq 2\Lambda.
    \end{equation}
    Combining (\ref{lemma, aux seq property, eq 20}), (\ref{lemma, aux seq property, eq 21}), (\ref{lemma, aux seq property, eq 22}), and (\ref{lemma, aux seq property, eq 23}) gives 
    \begin{equation}
    \label{lemma, aux seq property, eq 24}
        I_{3} \leq 24\Lambda I_{4}, \ {\rm when} \ 1\leq d_{k}<n-1.
    \end{equation}
    Combining (\ref{lemma, aux seq property, eq 19}) and (\ref{lemma, aux seq property, eq 24}) gives
    \begin{equation*}
        I_{3} \leq 24\Lambda I_{4}.
    \end{equation*}
    Substituting the definitions of $I_{3}$ and $I_{4}$ into the above equation, we have
    \begin{equation*}
        \sum_{j=1}^{n} \left(d_{k}+n-j \right)^{-a}(j-1) \leq 24\Lambda \sum_{j=1}^{n}\left(d_{k}+n-j \right)^{-a}\frac{(j-1)(n-j+1)}{n-1},
    \end{equation*}
    which proves (\ref{lemma, aux seq property, eq 3}).
\end{proof}

\subsection{Permutations generated by swapping elements and the non-expansiveness of the gradient descent step}

The following notions are needed in subsequent proofs. To the best of our knowledge, these notions can be traced back to \cite{nagaraj2019sgd}, and are used by subsequent works, e.g., \cite{liu2026random}. Consider a permutation $\pi$ of $[n]$ and two indices $i,j\in [n]$ with $j\leq i$, a new permutation is defined as follows
    \begin{equation*}
        \pi(i,j) \coloneqq \left(\pi^{1},\cdots, \pi^{j-1},\pi^{i},\pi^{j+1},\cdots, \pi^{i-1},\pi^{j},\pi^{i+1},\cdots, \pi^{n} \right).
    \end{equation*}
Following the description of \cite{liu2026random}, $\pi(i,j)$ is the permutation generated by exchanging the elements $\pi^{i}$ and $\pi^{j}$ in $\pi$. Similarly, for $i,j\in [n]$ with $j>i$, $\pi(i,j)$ is defined as the permutation generated by exchanging the elements $\pi^{i}$ and $\pi^{j}$ in $\pi$. With the notion of $\pi(i,j)$ defined above, a virtual sequence is defined as follows 
    \begin{equation}
    \label{def, in line, virtual seq, eq 1}
        x_{k}^{l+1}(i,j) \coloneqq x_{k}^{l}(i,j)-\eta_{k}\nabla f_{\pi_{k}^{l}(i,j)}(x_{k}^{l}(i,j)), \ \forall l\in [n], \ {\rm where} \ x_{k}^{1}(i,j)\coloneqq x_{k}^{1}.
    \end{equation}
In words, the sequence $\{x_{k}^{l}(i,j)\}_{l\in [n+1]}$ is generated by the RR method, but under the permutation $\pi_{k}(i,j)$. Based on this virtual sequence, we further define
    \begin{equation}
    \label{def, in line, virtual seq, eq 2}
        d_{k}^{l}(i,j) \coloneqq \eta_{k}\sum_{m=1}^{l-1}q_{\pi_{k}^{m}(i,j)}, \ u_{k}^{l}(i,j) \coloneqq x_{k}^{l}(i,j)+d_{k}^{l}(i,j), \forall l\in [n+1].
    \end{equation}
Informally, this sequence can be regarded as the adjoint shadow sequence of the virtual sequence $\{x_{k}^{l}(i,j)\}_{l\in [n+1]}$.

The following lemma estimates the distance between $x_{k}^{i}(i,j)$ and $x_{k}^{i}$, whose proof relies on the nonexpansiveness of the update rule in the gradient descent method. To be specific, we perform one gradient descent step starting at $x\in\mathbb{R}^{d}$ and $y\in\mathbb{R}^{d}$ respectively with the same stepsize $\eta\leq 2/L$\footnote{In the following iterations, the function $g:\mathbb{R}^{d}\rightarrow \mathbb{R}$ is convex and $L$-smooth.}
\begin{equation*}
    x^{+} = x - \eta\nabla g(x), \ y^{+} = y - \eta\nabla g(y).
\end{equation*}
Then the distance between $x$ and $y$ does not increase after the iteration, i.e., 
\begin{equation*}
    \Vert x^{+}-y^{+}\Vert \leq \Vert x-y\Vert.
\end{equation*}
Note that $x_{k}^{i}(i,j)$ and $x_{k}^{i}$ have the same starting point $x_{k}^{1}(i,j)=x_{k}^{1}$ (hence the initial distance is $0$), and the distance between the iterates of the corresponding sequences only increases when the gradient descent step is performed w.r.t. different functions, i.e., when the sampled component functions differ. Then the conclusion of the following lemma is obvious, and the reader could refer to the proof of Lemma $3$ in \cite{liu2026random} for a formal proof.

\begin{lemma}
\label{lemma, diff btw virtual and real seq}
     Assume the algorithm is run for $K\in\mathbb{N}^{+}$ epochs and $\eta_{k}\leq 2/L, \forall k\in [K]$. Then for any $k\in [K], i\in [n]$, and $j\in [i-1]$, we have
     \begin{equation*}
         \Vert x_{k}^{i}(i,j)-x_{k}^{i}\Vert \leq \eta_{k} \Vert \nabla f_{\pi_{k}^{i}}(x_{k}^{j})-\nabla f_{\pi_{k}^{j}}(x_{k}^{j})\Vert,
     \end{equation*}
     where $x_{k}^{i}(i,j)$ is defined in (\ref{def, in line, virtual seq, eq 1}).
\end{lemma}

\subsection{Bounds on component-function deviations}

Equipped with Lemma \ref{lemma, diff btw virtual and real seq}, we are able to prove the following lemma. This lemma bounds $\Omega_{\pi_{k}^{r}}(u_{k}^{p})$,  the difference between a component function $b_{\pi_{k}^{r}}$ and the corresponding average function $B$ evaluated at the same point $u_{k}^{p}$, where the point $u_{k}^{p}$ and the index $\pi_{k}^{r}$ are dependent r.v.'s. Such quantities are classical in the analysis of RR \cite{nagaraj2019sgd, liu2025improved} and the estimation of them constitutes one of the major challenges in these analyses.

\begin{lemma}
\label{lemma, RR component vs finite-sum}
    Assume the algorithm is run for $K\in\mathbb{N}^{+}$ epochs, $\eta_{k}\leq 2/L$, and consider a fixed epoch  $k\in [K]$. For any given $\varepsilon\in (0,1/4)$, we have
    \begin{equation*}
        \left| \mathbb{E}\Omega_{\pi_{k}^{r}}(u_{k}^{p}) \right| \leq 2\varepsilon \mathbb{E}B(u_{k}^{p}) + \frac{64L^{2}\eta_{k}^{2}}{\varepsilon n} \sum_{1\leq m<p} \mathbb{E} b_{\pi_{k}^{m}}(x_{k}^{m}) + \frac{48L\eta_{k}^{2}(p-1)}{\varepsilon n}\sigma_{*}^{2},
    \end{equation*}
    where $r,p\in [n], r\ge p$.
\end{lemma}
\begin{proof}
    Since $B(u_{k}^{p})=n^{-1}\sum_{m=1}^{n}b_{m}(u_{k}^{p})$, we have
    \begin{equation}
    \label{lemma, RR component vs finite-sum, eq 1}
        B(u_{k}^{p})-b_{\pi_{k}^{r}}(u_{k}^{p}) = \frac{1}{n}\sum_{m=1}^{n}\left(b_{m}(u_{k}^{p})-b_{\pi_{k}^{r}}(u_{k}^{p}) \right) = \frac{1}{n} \sum_{m=1}^{n}\left(b_{\pi_{k}^{m}}(u_{k}^{p})-b_{\pi_{k}^{r}}(u_{k}^{p}) \right),
    \end{equation}
    where the last equality is by $\{\pi_{k}^{m} | m\in [n] \}=[n]$. Note that 
    \begin{equation*}
        \pi_{k}(r,m) \overset{\mathcal{D}}{=} \pi_{k}, \ \left(\pi_{k}(r,m), u_{k}^{p}(r,m) \right) \overset{\mathcal{D}}{=} \left(\pi_{k}, u_{k}^{p} \right), \ p\in [n],
    \end{equation*}
    then
    \begin{equation}
    \label{lemma, RR component vs finite-sum, eq 2}
        \mathbb{E} b_{\pi_{k}^{m}}(u_{k}^{p}) = \mathbb{E} b_{\pi_{k}^{m}(r,m)}(u_{k}^{p}(r,m)) = \mathbb{E} b_{\pi_{k}^{r}}(u_{k}^{p}(r,m)).
    \end{equation}
    By definitions (\ref{def, in line, virtual seq, eq 1}) and (\ref{def, in line, virtual seq, eq 2}),
    \begin{equation}
    \label{lemma, RR component vs finite-sum, eq 3}
        u_{k}^{p}(r,m) = u_{k}^{p}, \ \mathbb{E}b_{\pi_{k}^{r}}(u_{k}^{p}(r,m)) = \mathbb{E}b_{\pi_{k}^{r}}(u_{k}^{p}), \ m \geq p.
    \end{equation}
    Taking expectation in (\ref{lemma, RR component vs finite-sum, eq 1}) and using (\ref{lemma, RR component vs finite-sum, eq 2}) and (\ref{lemma, RR component vs finite-sum, eq 3}) lead to
    \begin{equation*}
        \mathbb{E} \Omega_{\pi_{k}^{r}}(u_{k}^{p}) = \frac{-1}{n}\sum_{1\leq m<p}\mathbb{E}\left[b_{\pi_{k}^{r}}(u_{k}^{p}(r,m))-b_{\pi_{k}^{r}}(u_{k}^{p}) \right].
    \end{equation*}
    By Jensen's inequality, we have
    \begin{equation}
    \label{lemma, RR component vs finite-sum, eq 4}
        \left|\mathbb{E}\Omega_{\pi_{k}^{r}}(u_{k}^{p})\right| \leq \frac{1}{n}\sum_{1\leq m<p} \mathbb{E} \left| b_{\pi_{k}^{r}}(u_{k}^{p}(r,m))-b_{\pi_{k}^{r}}(u_{k}^{p}) \right|.
    \end{equation}
    By Lemma \ref{lemma, special func property}, we have
    \begin{align}
    \notag
        \left| b_{\pi_{k}^{r}}(u_{k}^{p}(r,m))-b_{\pi_{k}^{r}}(u_{k}^{p}) \right| &\leq \Vert \nabla b_{\pi_{k}^{r}}(u_{k}^{p})\Vert \Vert u_{k}^{p}(r,m)-u_{k}^{p}\Vert + \frac{L}{2} \Vert u_{k}^{p}(r,m)-u_{k}^{p}\Vert^2 \\
        &\leq \varepsilon b_{\pi_{k}^{r}}(u_{k}^{p}) + \frac{L}{\varepsilon} \Vert u_{k}^{p}(r,m)-u_{k}^{p}\Vert^2,
        \label{lemma, RR component vs finite-sum, eq 5}
    \end{align}
    where the first inequality is by (\ref{lemma, special func property, eq 2}), and the second inequality is by Young's inequality, (\ref{lemma, special func property, eq 1}), and the setting $\varepsilon\in (0,1/4)$. The quantity $\Vert u_{k}^{p}(r,m)-u_{k}^{p}\Vert^2$ is further bounded as follows. By (\ref{def, in line, virtual seq, eq 2}), we have
    \begin{align*}
        \Vert u_{k}^{p}(r,m)-u_{k}^{p}\Vert^2 &= \Vert \left( x_{k}^{p}(r,m)-x_{k}^{p}\right) + \eta_{k}\left(q_{\pi_{k}^{r}}-q_{\pi_{k}^{m}}\right)\Vert^2 \\
        &\leq 2\eta_{k}^{2}\Vert \nabla f_{\pi_{k}^{r}}(x_{k}^{m})-\nabla f_{\pi_{k}^{m}}(x_{k}^{m})\Vert^2 + 2\eta_{k}^{2} \Vert q_{\pi_{k}^{r}}-q_{\pi_{k}^{m}}\Vert^2\\
        &= 2\eta_{k}^{2} \Vert \left(\nabla b_{\pi_{k}^{r}}(x_{k}^{m})-\nabla b_{\pi_{k}^{m}}(x_{k}^{m})\right)+\left(q_{\pi_{k}^{r}}-q_{\pi_{k}^{m}} \right)\Vert^2+2\eta_{k}^{2} \Vert q_{\pi_{k}^{r}}-q_{\pi_{k}^{m}}\Vert^2 \\
        &\leq 8\eta_{k}^{2}\Vert \nabla b_{\pi_{k}^{r}}(x_{k}^{m})\Vert^2 + 8\eta_{k}^{2}\Vert \nabla b_{\pi_{k}^{m}}(x_{k}^{m})\Vert^2 + 12\eta_{k}^{2}\Vert q_{\pi_{k}^{r}}\Vert^2 + 12\eta_{k}^{2}\Vert q_{\pi_{k}^{m}}\Vert^2 \\
        &\leq 16L\eta_{k}^{2} b_{\pi_{k}^{r}}(x_{k}^{m}) + 16L\eta_{k}^{2} b_{\pi_{k}^{m}}(x_{k}^{m}) + 12\eta_{k}^{2}\Vert q_{\pi_{k}^{r}}\Vert^2 + 12\eta_{k}^{2}\Vert q_{\pi_{k}^{m}}\Vert^2,
    \end{align*}
    where the first inequality is by Lemma \ref{lemma, diff btw virtual and real seq}, the second equality is by Definition \ref{def, basics}, and the last inequality is by Lemma \ref{lemma, special func property}. Taking expectation in the above inequality gives
    \begin{equation}
    \label{lemma, RR component vs finite-sum, eq 6}
        \mathbb{E} \Vert u_{k}^{p}(r,m)-u_{k}^{p}\Vert^2 \leq 32L\eta_{k}^{2} \mathbb{E}b_{\pi_{k}^{m}}(x_{k}^{m}) + 24\eta_{k}^{2}\sigma_{*}^{2},
    \end{equation}
    where we used $\mathbb{E} \Vert q_{\pi_{k}^{r}}\Vert^2=\mathbb{E}\Vert q_{\pi_{k}^{m}}\Vert^2=\sigma_{*}^{2}$ and
    \begin{equation*}
        \mathbb{E}b_{\pi_{k}^{r}}(x_{k}^{m}) = \mathbb{E}\left\{ \mathbb{E} \left[ b_{\pi_{k}^{r}}(x_{k}^{m})| x_{k}^{m}\right]\right\} = \mathbb{E} b_{\pi_{k}^{m}}(x_{k}^{m}), \  r>m,
    \end{equation*}
    where the first equality is by the tower property of conditional expectation. Combining (\ref{lemma, RR component vs finite-sum, eq 4}), (\ref{lemma, RR component vs finite-sum, eq 5}), and (\ref{lemma, RR component vs finite-sum, eq 6}) gives
    \begin{align*}
        \left|\mathbb{E}\Omega_{\pi_{k}^{r}}(u_{k}^{p}) \right| &\leq \frac{1}{n}\sum_{1\leq m<p} \left(\varepsilon \mathbb{E}b_{\pi_{k}^{r}}(u_{k}^{p}) + 32L^{2}\eta_{k}^{2}\varepsilon^{-1}\mathbb{E}b_{\pi_{k}^{m}}(x_{k}^{m}) + 24L\eta_{k}^{2}\varepsilon^{-1} \sigma_{*}^{2} \right) \\
        &= \frac{1}{n} \sum_{1\leq m<p} \left(\varepsilon \mathbb{E}\Omega_{\pi_{k}^{r}}(u_{k}^{p}) + \varepsilon \mathbb{E}B(u_{k}^{p}) + 32L^{2}\eta_{k}^{2}\varepsilon^{-1}\mathbb{E}b_{\pi_{k}^{m}}(x_{k}^{m}) + 24L\eta_{k}^{2}\varepsilon^{-1} \sigma_{*}^{2} \right) \\
        &\leq \varepsilon \left| \mathbb{E}\Omega_{\pi_{k}^{r}}(u_{k}^{p})\right| + \varepsilon \mathbb{E}B(u_{k}^{p}) + \frac{32L^{2}\eta_{k}^{2}}{\varepsilon n}
        \sum_{1\leq m<p} \mathbb{E}b_{\pi_{k}^{m}}(x_{k}^{m}) + \frac{24L\eta_{k}^{2}(p-1)}{\varepsilon n} \sigma_{*}^{2}.
    \end{align*}
    In the above inequality, moving the term $\varepsilon |\mathbb{E}\Omega_{\pi_{k}^{r}}(u_{k}^{p})|$ to the LHS and using $1-\varepsilon\geq 3/4\geq 1/2$ give us
    \begin{equation*}
        \left| \mathbb{E}\Omega_{\pi_{k}^{r}}(u_{k}^{p}) \right| \leq 2\varepsilon \mathbb{E}B(u_{k}^{p}) + \frac{64L^{2}\eta_{k}^{2}}{\varepsilon n} \sum_{1\leq m<p} \mathbb{E} b_{\pi_{k}^{m}}(x_{k}^{m}) + \frac{48L\eta_{k}^{2}(p-1)}{\varepsilon n}\sigma_{*}^{2},
    \end{equation*}
    which is the desired inequality.
\end{proof}

\begin{lemma}
\label{lemma, RR component vs finite-sum v2}
    Assume the algorithm is run for $K\in\mathbb{N}^{+}$ epochs, $\eta_{k}\leq 2/L$, and consider a fixed epoch  $k\in [K]$. For any given $\varepsilon\in (0,1/4)$, we have
    \begin{equation*}
        \left|\mathbb{E}\Omega_{\pi_{k}^{r}}(x_{k}^{r}) \right| \leq 6\varepsilon \mathbb{E} B(u_{k}^{r}) + \frac{128L^{2}\eta_{k}^{2}}{\varepsilon n}\sum_{1\leq m<r} \mathbb{E}b_{\pi_{k}^{m}}(x_{k}^{m}) + \frac{96L\eta_{k}^{2}(r-1)}{\varepsilon n}\sigma_{*}^{2} + \frac{4L}{\varepsilon}\mathbb{E}\Vert d_{k}^{r}\Vert^2,
    \end{equation*}
    where $r\in [n]$.
\end{lemma}
\begin{proof}
    We have
    \begin{align*}
        \left| \mathbb{E}\Omega_{\pi_{k}^{r}}(x_{k}^{r})\right| &\leq \left|\mathbb{E}\Omega_{\pi_{k}^{r}}(u_{k}^{r}) \right| + \left|\mathbb{E}\Omega_{\pi_{k}^{r}}(x_{k}^{r})-\mathbb{E}\Omega_{\pi_{k}^{r}}(u_{k}^{r}) \right| \\
        &\leq \left|\mathbb{E}\Omega_{\pi_{k}^{r}}(u_{k}^{r}) \right| + \mathbb{E} \left| \Omega_{\pi_{k}^{r}}(x_{k}^{r})-\Omega_{\pi_{k}^{r}}(u_{k}^{r})\right| \\
        &\leq \left|\mathbb{E}\Omega_{\pi_{k}^{r}}(u_{k}^{r}) \right| + \mathbb{E} \left| b_{\pi_{k}^{r}}(x_{k}^{r})-b_{\pi_{k}^{r}}(u_{k}^{r})\right| + \mathbb{E} \left| B(x_{k}^{r})-B(u_{k}^{r})\right| \\
        &\leq \left|\mathbb{E}\Omega_{\pi_{k}^{r}}(u_{k}^{r}) \right| + \mathbb{E}\Vert \nabla b_{\pi_{k}^{r}}(u_{k}^{r})\Vert \Vert x_{k}^{r}-u_{k}^{r}\Vert + \frac{L}{2}\mathbb{E}\Vert x_{k}^{r}-u_{k}^{r}\Vert^2 \\
        &\quad +\mathbb{E}\Vert \nabla B(u_{k}^{r})\Vert \Vert x_{k}^{r}-u_{k}^{r}\Vert + \frac{L}{2}\mathbb{E}\Vert x_{k}^{r}-u_{k}^{r}\Vert^2 \\
        &\leq \left|\mathbb{E}\Omega_{\pi_{k}^{r}}(u_{k}^{r}) \right| + \varepsilon \mathbb{E} b_{\pi_{k}^{r}}(u_{k}^{r}) + \frac{L}{\varepsilon}\mathbb{E} \Vert x_{k}^{r}-u_{k}^{r}\Vert^2 + \varepsilon \mathbb{E}B(u_{k}^{r}) + \frac{L}{\varepsilon} \mathbb{E}\Vert x_{k}^{r}-u_{k}^{r}\Vert^2 \\
        &\leq \left|\mathbb{E}\Omega_{\pi_{k}^{r}}(u_{k}^{r}) \right| + \varepsilon \left| \mathbb{E}\Omega_{\pi_{k}^{r}}(u_{k}^{r})\right| + 2\varepsilon \mathbb{E}B(u_{k}^{r}) + \frac{2L}{\varepsilon}\mathbb{E}\Vert d_{k}^{r}\Vert^2 \\
        &\leq 2\left|\mathbb{E}\Omega_{\pi_{k}^{r}}(u_{k}^{r}) \right| + 2\varepsilon \mathbb{E}B(u_{k}^{r}) + \frac{2L}{\varepsilon}\mathbb{E}\Vert d_{k}^{r}\Vert^2,
    \end{align*}
    where the second inequality is by Jensen's inequality, the fourth inequality is by Lemma \ref{lemma, special func property}, the fifth inequality is by Lemma \ref{lemma, special func property} and $\varepsilon\in (0,1/4)$, and in the sixth inequality we use Definition \ref{def, basics}. Using Lemma \ref{lemma, RR component vs finite-sum} with $r=p$ to upper bound the term $2\left|\mathbb{E}\Omega_{\pi_{k}^{r}}(u_{k}^{r}) \right|$ of the above inequality gives the desired equation.
\end{proof}

With Lemma \ref{lemma, RR component vs finite-sum} and Lemma \ref{lemma, RR component vs finite-sum v2}, we can prove the following equation which bounds a quantity arising when we consider the weighted sum of the one-step inequality from Lemma \ref{lemma, one step lemma}. Lemma \ref{lemma, RR component vs finite-sum}, Lemma \ref{lemma, RR component vs finite-sum v2}, and the following lemma upper bound their target terms with (a weighted sum of) $B(u_{t})$, which will be absorbed by similar terms in the LHS of the weighted sum of the aforementioned one-step inequality.

\begin{lemma}
\label{lemma, bound the overall bias}
    Define
    \begin{equation*}
        \mathcal{B} \coloneqq \sum_{s=1}^{M}\Delta v_{s}\sum_{t=s}^{M} \alpha_{t}\mathbb{E}\Omega_{i_{t}}(u_{s}) - \sum_{t=1}^{M} \alpha_{t}v_{t+1}\mathbb{E}\Omega_{i_{t}}(x_{t}),
    \end{equation*}
    and set
    \begin{equation*}
        \varepsilon = \frac{\lambda}{C}, \ \eta_{1}\leq \frac{1}{17 C^{3/2}\Lambda^{3/2}L}.
    \end{equation*}
    Then we have
    \begin{equation*}
        \mathcal{B} \leq \left|\mathcal{B} \right| \leq  \frac{\lambda}{2}\sum_{k=1}^{K}\sum_{p=1}^{n}\eta_{k}v_{\tau_{k,p}+1}\mathbb{E}B(u_{k}^{p}) + 153C^{2}\Lambda^{2}L \sum_{k=1}^{K}\sum_{p=1}^{n}\eta_{k}v_{\tau_{k,p}+1}\mathbb{E}\Vert d_{k}^{p}\Vert^2,
    \end{equation*}
    and we denote the above upper bound of $|\mathcal{B}|$ as $|\mathcal{B}|_{\rm{up}}$.
\end{lemma}
\begin{proof}
    For notational convenience, define $c_{k,p}\coloneqq \eta_{k}v_{\tau_{k,p}+1}$.
    We also define
    \begin{equation*}
        \mathcal{B}_{1} \coloneqq \sum_{s=1}^{M}\Delta v_{s}\sum_{t=s}^{M} \alpha_{t}\left| \mathbb{E}\Omega_{i_{t}}(u_{s})\right|, \ \mathcal{B}_{2} \coloneqq \sum_{t=1}^{M} \alpha_{t}v_{t+1} \left|\mathbb{E}\Omega_{i_{t}}(x_{t})\right|,
    \end{equation*}
    then $\mathcal{B}\leq |\mathcal{B}|\leq \mathcal{B}_{1}+\mathcal{B}_{2}$. In the inner summation $\sum_{t=s}^{M} \alpha_{t}|\mathbb{E}\Omega_{i_{t}}(u_{s})|$ of $\mathcal{B}_{1}$, when $t$ and $s$ are not in the same epoch, we have (note that $t\geq s$)
    \begin{equation*}
        \mathbb{E}\Omega_{i_{t}}(u_{s}) = \mathbb{E}\left\{\mathbb{E}\left[\Omega_{i_{t}}(u_{s}) | u_{\lceil s/n\rceil}  \right] \right\} = 0,
    \end{equation*}
    where the second equality is because $i_{t}\sim \operatorname{Unif} [n]$ conditioned on all the history up to the end of the epoch containing $u_{s}$. Hence $\mathcal{B}_{1}$ can be reformulated as follows
    \begin{equation*}
        \mathcal{B}_{1} = \sum_{k=1}^{K} \eta_{k} \sum_{p=1}^{n} \Delta v_{\tau_{k,p}} \sum_{r=p}^{n} \left|\mathbb{E}\Omega_{\pi_{k}^{r}}(u_{k}^{p})\right|.
    \end{equation*}
    Denote $\gamma_{k,p}\coloneqq \eta_{k}\Delta v_{\tau_{k,p}}(n-p+1)$. Note that $\eta_{k}>0$ and $\Delta v_{\tau_{k,p}}>0$ (since the sequence $\{v_{t}\}_{t}$ is increasing), by Lemma \ref{lemma, RR component vs finite-sum}, we have
    \begin{align}
    \notag
        \mathcal{B}_{1} &\leq \sum_{k=1}^{K}\eta_{k} \sum_{p=1}^{n} \Delta v_{\tau_{k,p}} \sum_{r=p}^{n} \left( 2\varepsilon \mathbb{E}B(u_{k}^{p}) + \frac{64L^{2}\eta_{k}^{2}}{\varepsilon n} \sum_{1\leq m<p} \mathbb{E} b_{\pi_{k}^{m}}(x_{k}^{m}) + \frac{48L\eta_{k}^{2}(p-1)}{\varepsilon n}\sigma_{*}^{2} \right) \\
    \label{lemma, bound the overall bias, eq 1}
        &\leq \sum_{k=1}^{K}\sum_{p=1}^{n} \gamma_{k,p} \left( 2\varepsilon \mathbb{E}B(u_{k}^{p}) + \frac{64L^{2}\eta_{k}^{2}}{\varepsilon n} \sum_{1\leq m<p} \mathbb{E} b_{\pi_{k}^{m}}(x_{k}^{m}) + \frac{48L\eta_{k}^{2}(p-1)}{\varepsilon n}\sigma_{*}^{2} \right),
    \end{align}
    where the second inequality is by the fact that the upper bound of $|\mathbb{E}\Omega_{\pi_{k}^{r}}(u_{k}^{p})|$ in Lemma \ref{lemma, RR component vs finite-sum} does not depend on $r$. We then upper bound $\gamma_{k,p}$ with $c_{k,p}$ as follows. By Definition \ref{def, auxiliary seq}, $\Delta v_{t} = a\alpha_{t}v_{t+1}/H_{t}$. When $t=\tau_{k,p}$, we have
    \begin{equation*}
        \Delta v_{\tau_{k,p}} = a\eta_{k}v_{\tau_{k,p}+1}/H_{\tau_{k,p}} = ac_{k,p}/H_{\tau_{k,p}}.
    \end{equation*}
    Then 
    \begin{equation}
    \label{lemma, bound the overall bias, eq 2}
        \gamma_{k,p} = ac_{k,p}\frac{(n-p+1)\eta_{k}}{H_{\tau_{k,p}}} \leq ac_{k,p} \leq c_{k,p},
    \end{equation}
    where the first inequality is by
    \begin{equation*}
        H_{\tau_{k,p}} = \sum_{s=p}^{n} \alpha_{\tau_{k,s}} + \sum_{t=\tau_{k+1,1}}^{M}\alpha_{t} \geq  \sum_{s=p}^{n} \alpha_{\tau_{k,s}} = (n-p+1)\eta_{k}.
    \end{equation*}
    Define
    \begin{align*}
        S_{1,1} &\coloneqq \varepsilon \sum_{k=1}^{K}\sum_{p=1}^{n} c_{k,p}\mathbb{E}B(u_{k}^{p}), \ S_{1,2}\coloneqq \sum_{k=1}^{K} \frac{L^{2}\eta_{k}^{2}}{\varepsilon n} \sum_{p=1}^{n} c_{k,p} \sum_{1\leq m<p} \mathbb{E}b_{\pi_{k}^{m}}(x_{k}^{m}), \\
        S_{1,3} &\coloneqq \sum_{k=1}^{K}\frac{L\eta_{k}^{2}}{\varepsilon n}\sum_{p=1}^{n} c_{k,p}(p-1) \sigma_{*}^{2}.
    \end{align*}
    Then combining (\ref{lemma, bound the overall bias, eq 1}) and (\ref{lemma, bound the overall bias, eq 2}) gives
    \begin{equation}
    \label{lemma, bound the overall bias, eq 3}
        \mathcal{B}_{1} \leq 2S_{1,1} + 64S_{1,2} + 48S_{1,3}.
    \end{equation}
    Define
    \begin{equation*}
        \mathcal{S}_{U}\coloneqq \sum_{k=1}^{K}\sum_{p=1}^{n} c_{k,p}\mathbb{E}B(u_{k}^{p}), \ \mathcal{S}_{X}\coloneqq \sum_{k=1}^{K}\sum_{p=1}^{n}c_{k,p} \mathbb{E}b_{\pi_{k}^{p}}(x_{k}^{p}), \ \mathcal{S}_{D}\coloneqq \sum_{k=1}^{K}\sum_{p=1}^{n} c_{k,p}\mathbb{E}\Vert d_{k}^{p}\Vert^2.
    \end{equation*}
    It is clear that
    \begin{equation}
    \label{lemma, bound the overall bias, eq 4}
        S_{1,1} = \varepsilon \mathcal{S}_{U}.
    \end{equation}
    For $S_{1,2}$, we have
    \begin{align}
    \notag
        S_{1,2} &= \sum_{k=1}^{K} \frac{L^{2}\eta_{k}^{2}}{\varepsilon n} \sum_{p=1}^{n}c_{k,p} \sum_{1\leq m<p} \mathbb{E}b_{\pi_{k}^{m}}(x_{k}^{m})
        = \sum_{k=1}^{K} \frac{L^{2}\eta_{k}^{2}}{\varepsilon n} \sum_{m=1}^{n-1} \mathbb{E}b_{\pi_{k}^{m}}(x_{k}^{m})\sum_{p=m+1}^{n}c_{k,p}\\
        \notag
        &= \sum_{k=1}^{K}\frac{L^{2}\eta_{k}^{2}}{\varepsilon} \sum_{m=1}^{n-1}\mathbb{E}b_{\pi_{k}^{m}}(x_{k}^{m}) \eta_{k} \sum_{p=m+1}^{n}\frac{v_{\tau_{k,p}+1}}{n} \\
        \notag
        &\leq C\Lambda \sum_{k=1}^{K}\frac{L^{2}\eta_{k}^{2}}{\varepsilon} \sum_{m=1}^{n} \eta_{k}v_{\tau_{k,m}+1} \mathbb{E}b_{\pi_{k}^{m}}(x_{k}^{m}) \\
        &= C\Lambda \sum_{k=1}^{K}\frac{L^{2}\eta_{k}^{2}}{\varepsilon} \sum_{m=1}^{n} c_{k,m}\mathbb{E}b_{\pi_{k}^{m}}(x_{k}^{m})
        \leq \frac{C\Lambda L^2\eta_{1}^{2}}{\varepsilon} \mathcal{S}_{X},
        \label{lemma, bound the overall bias, eq 5}
    \end{align}
    where the second equality is by interchanging the order of summation, the third equality is by the definition $c_{k,p}=\eta_{k}v_{\tau_{k,p}+1}$, the first inequality is by Lemma \ref{lemma, aux seq property}, and the last inequality is by $\eta_{k}\leq \eta_{1}$. For $S_{1,3}$, we have
    \begin{align}
    \notag
        S_{1,3} &= \sum_{k=1}^{K} \frac{L\eta_{k}^{2}}{\varepsilon n} \sum_{p=1}^{n} c_{k,p}(p-1)\sigma_{*}^{2} =  \sum_{k=1}^{K} \frac{L\eta_{k}^{3}\sigma_{*}^{2}}{\varepsilon n} \sum_{p=1}^{n} v_{\tau_{k,p}+1}(p-1) \\
    \notag
        &\leq C\Lambda \sum_{k=1}^{K} \frac{L\eta_{k}^{3}\sigma_{*}^{2}}{\varepsilon n}\sum_{p=1}^{n} v_{\tau_{k,p}+1} \frac{(p-1)(n-p+1)}{n-1} \\
    \notag
        &= C\Lambda \sum_{k=1}^{K}\frac{L}{\varepsilon n} \sum_{p=1}^{n}\left(\eta_{k} v_{\tau_{k,p}+1} \right) \eta_{k}^{2}\frac{(p-1)(n-p+1)}{n-1}\sigma_{*}^{2} \\
    \label{lemma, bound the overall bias, eq 6}
        &= \frac{CL\Lambda}{\varepsilon n}\sum_{k=1}^{K}\sum_{p=1}^{n}c_{k,p}\mathbb{E}\Vert d_{k}^{p}\Vert^2 = \frac{CL\Lambda }{\varepsilon n}\mathcal{S}_{D},
    \end{align}
    where the second equality is by the definition of $c_{k,p}$, the first inequality is by Lemma \ref{lemma, aux seq property}, and the fourth equality is by Lemma \ref{lemma, SWR lemma}. Plugging (\ref{lemma, bound the overall bias, eq 4}), (\ref{lemma, bound the overall bias, eq 5}), and (\ref{lemma, bound the overall bias, eq 6}) into (\ref{lemma, bound the overall bias, eq 3}) gives
    \begin{equation}
    \label{lemma, bound the overall bias, eq 7}
        \mathcal{B}_{1} \leq 2\varepsilon \mathcal{S}_{U} + \frac{64 C\Lambda L^{2}\eta_{1}^{2}}{\varepsilon} \mathcal{S}_{X} + \frac{48CL\Lambda}{\varepsilon n} \mathcal{S}_{D}.
    \end{equation}
    To bound $\mathcal{B}_{2}=\sum_{t=1}^{M}\alpha_{t}v_{t+1}|\mathbb{E} \Omega_{i_{t}}(x_{t})|$, we rewrite $\mathcal{B}_{2}$ as follows
    \begin{equation*}
        \mathcal{B}_{2} = \sum_{k=1}^{K}\eta_{k}\sum_{p=1}^{n}v_{\tau_{k,p}+1} \left|\mathbb{E}\Omega_{\pi_{k}^{p}}(x_{k}^{p})\right|.
    \end{equation*}
    By Lemma \ref{lemma, RR component vs finite-sum v2}, we have
    \begin{align}
    \notag
        \mathcal{B}_{2}  &\leq \sum_{k=1}^{K}\sum_{p=1}^{n} 6\varepsilon \eta_{k}v_{\tau_{k,p}+1} \mathbb{E}B(u_{k}^{p}) + \sum_{k=1}^{K} \frac{128L^{2}\eta_{k}^{2}}{\varepsilon n} \sum_{p=1}^{n}\eta_{k}v_{\tau_{k,p}+1} \sum_{1\leq m<p}\mathbb{E}b_{\pi_{k}^{m}}(x_{k}^{m}) \\
    \notag
        & \quad +\sum_{k=1}^{K}\sum_{p=1}^{n}\eta_{k}v_{\tau_{k,p}+1}\frac{96L\eta_{k}^{2}(p-1)}{\varepsilon n} \sigma_{*}^{2} + \frac{4L}{\varepsilon}\sum_{k=1}^{K}\sum_{p=1}^{n} \eta_{k}v_{\tau_{k,p}+1}\mathbb{E}\Vert d_{k}^{p}\Vert^2 \\
    \notag
        &\leq 6\varepsilon \mathcal{S}_{U} + \frac{128L^{2}\eta_{1}^{2}}{\varepsilon} \sum_{k=1}^{K} \eta_{k} \sum_{m=1}^{n}\mathbb{E}b_{\pi_{k}^{m}}(x_{k}^{m}) \sum_{p=m+1}^{n} \frac{v_{\tau_{k,p}+1}}{n} \\
    \notag
        & \quad + \frac{96L}{\varepsilon n}\sum_{k=1}^{K}\eta_{k}^{3}\sigma_{*}^{2}\sum_{p=1}^{n} v_{\tau_{k,p}+1}(p-1) + \frac{4L}{\varepsilon} \mathcal{S}_{D} \\
    \notag
        &\leq 6\varepsilon \mathcal{S}_{U} + \frac{128C\Lambda L^{2}\eta_{1}^{2}}{\varepsilon}\sum_{k=1}^{K}\sum_{p=1}^{n}\eta_{k}v_{\tau_{k,p}+1}\mathbb{E}b_{\pi_{k}^{p}}(x_{k}^{p}) + \frac{4L}{\varepsilon}\mathcal{S}_{D} \\
    \notag
        & \quad + \frac{96C\Lambda L}{\varepsilon n} \sum_{k=1}^{K}\sum_{p=1}^{n} \eta_{k}v_{\tau_{k,p}+1} \eta_{k}^{2}\frac{(p-1)(n-p+1)}{n-1}\sigma_{*}^{2} \\
    \label{lemma, bound the overall bias, eq 8}
        &\leq 6\varepsilon \mathcal{S}_{U} + \frac{128C\Lambda L^{2}\eta_{1}^{2}}{\varepsilon}\mathcal{S}_{X} + \frac{96C\Lambda L}{\varepsilon n} \mathcal{S}_{D} + \frac{4L}{\varepsilon}\mathcal{S}_{D},
    \end{align}
    where the second inequality is by $\eta_{k}\leq \eta_{1}$, the third inequality is by Lemma \ref{lemma, aux seq property}, and the last inequality is by Lemma \ref{lemma, SWR lemma}. Note that
    \begin{equation*}
        \frac{48C\Lambda L}{\varepsilon n}+\frac{96C\Lambda L}{\varepsilon n} + \frac{4L}{\varepsilon} \leq \frac{144C\Lambda L}{\varepsilon} + \frac{4C\Lambda L}{\varepsilon} = \frac{148C\Lambda L}{\varepsilon},
    \end{equation*}
    where the inequality is by $n\geq 1$, $C\geq 1$, and $\Lambda\geq 1$. Combining the above inequality, (\ref{lemma, bound the overall bias, eq 7}), and (\ref{lemma, bound the overall bias, eq 8}) gives 
    \begin{equation}
    \label{lemma, bound the overall bias, eq 9}
        \mathcal{B} \leq |\mathcal{B}| \leq \mathcal{B}_{1} + \mathcal{B}_{2} \leq 8\varepsilon \mathcal{S}_{U} + \frac{192C\Lambda L^{2}\eta_{1}^{2}}{\varepsilon}\mathcal{S}_{X} + \frac{148C\Lambda L}{\varepsilon}\mathcal{S}_{D}.
    \end{equation}
    We upper bound $\mathcal{S}_{X}$ with $\mathcal{S}_{U}$ and $\mathcal{S}_{D}$ as follows. By $L-$smoothness of $b_{\pi_{k}^{p}}$, 
    \begin{align*}
        b_{\pi_{k}^{p}}(x_{k}^{p}) & \leq b_{\pi_{k}^{p}}(u_{k}^{p}) + \langle \nabla b_{\pi_{k}^{p}}(u_{k}^{p}),x_{k}^{p}-u_{k}^{p}\rangle + \frac{L}{2}\Vert x_{k}^{p}-u_{k}^{p}\Vert^2 \\
        &\leq b_{\pi_{k}^{p}}(u_{k}^{p}) + \frac{\varepsilon}{2L}\Vert \nabla b_{\pi_{k}^{p}}(u_{k}^{p})\Vert^2 + \frac{L}{\varepsilon}\Vert x_{k}^{p}-u_{k}^{p}\Vert^2 + \frac{L}{2}\Vert x_{k}^{p}-u_{k}^{p}\Vert^2 \\
        &\leq (1+\varepsilon)b_{\pi_{k}^{p}}(u_{k}^{p}) + \frac{2L}{\varepsilon}\Vert x_{k}^{p}-u_{k}^{p}\Vert^2 \\
        &= (1+\varepsilon) B(u_{k}^{p}) + (1+\varepsilon)\Omega_{\pi_{k}^{p}}(u_{k}^{p}) + \frac{2L}{\varepsilon}\Vert x_{k}^{p}-u_{k}^{p}\Vert^2,
    \end{align*}
    where the second inequality is by Young's inequality, the third inequality is by Lemma \ref{lemma, special func property}, and the last equality is by Definition \ref{def, basics}. Then
    \begin{align*}
        \mathcal{S}_{X} &= \sum_{k=1}^{K}\sum_{p=1}^{n} c_{k,p}\mathbb{E}b_{\pi_{k}^{p}}(x_{k}^{p}) \\
        &\leq \sum_{k=1}^{K}\sum_{p=1}^{n} \left[(1+\varepsilon )c_{k,p}\mathbb{E}B(u_{k}^{p})  + 2L c_{k,p}\varepsilon^{-1}\mathbb{E}\Vert x_{k}^{p}-u_{k}^{p}\Vert^2 \right] \\
        & \quad + (1+\varepsilon)\sum_{k=1}^{K}\sum_{p=1}^{n}  c_{k,p}\left[2\varepsilon \mathbb{E}B(u_{k}^{p}) + \frac{64L^{2}\eta_{k}^{2}}{\varepsilon n} \sum_{1\leq m<p} \mathbb{E} b_{\pi_{k}^{m}}(x_{k}^{m}) + \frac{48L\eta_{k}^{2}(p-1)}{\varepsilon n}\sigma_{*}^{2}\right] \\
        &\leq 3(1+\varepsilon)\mathcal{S}_{U} + \frac{2L}{\varepsilon}\mathcal{S}_{D} + (1+\varepsilon) \sum_{k=1}^{K} \frac{64L^{2}\eta_{k}^{3}}{\varepsilon} \sum_{m=1}^{n-1}\mathbb{E}b_{\pi_{k}^{m}}(x_{k}^{m})\sum_{p=m+1}^{n}\frac{v_{\tau_{k,p}+1}}{n} \\
        & \quad + (1+\varepsilon) \sum_{k=1}^{K} \frac{48L\eta_{k}^{3}}{\varepsilon n}\sigma_{*}^{2}\sum_{p=1}^{n}v_{\tau_{k,p}+1}(p-1) \\
        &\leq 3(1+\varepsilon)\mathcal{S}_{U} + \frac{2L}{\varepsilon}\mathcal{S}_{D} + (1+\varepsilon)C\Lambda\sum_{k=1}^{K}\frac{64L^{2}\eta_{k}^{2}}{\varepsilon}\sum_{m=1}^{n} \eta_{k}v_{\tau_{k,m}+1}\mathbb{E}b_{\pi_{k}^{m}}(x_{k}^{m}) \\
        & \quad + (1+\varepsilon)C\Lambda\sum_{k=1}^{K} \frac{48L}{\varepsilon n}\sum_{p=1}^{n} (\eta_{k}v_{\tau_{k,p}+1})\eta_{k}^{2}\frac{(p-1)(n-p+1)}{n-1}\sigma_{*}^{2} \\
        &\leq 3(1+\varepsilon)\mathcal{S}_{U} + \frac{2L}{\varepsilon}\mathcal{S}_{D} + \frac{90 C \Lambda L^{2}\eta_{1}^{2}}{\varepsilon}\mathcal{S}_{X} + \frac{60 C\Lambda L}{\varepsilon n}\mathcal{S}_{D},
    \end{align*}
    where the first inequality is by Lemma \ref{lemma, RR component vs finite-sum}, the second inequality is by interchanging the order of summation, the third inequality is by Lemma \ref{lemma, aux seq property}, and the last inequality is by Lemma \ref{lemma, SWR lemma}. Since $C\geq 1, \Lambda \geq 1$, and $n\geq 1$, we have 
    \begin{equation*}
        \mathcal{S}_{X} \leq 3(1+\varepsilon)\mathcal{S}_{U} + \frac{90C\Lambda L^{2}\eta_{1}^{2}}{\varepsilon}\mathcal{S}_{X} + \frac{62C\Lambda L}{\varepsilon}\mathcal{S}_{D}.
    \end{equation*}
    By the settings
    \begin{equation*}
        \varepsilon = \frac{\lambda}{C}, \ \eta_{1}\leq \frac{1}{17 C^{3/2}\Lambda^{3/2}L},
    \end{equation*}
    we have $90C\Lambda L^{2}\eta_{1}^{2}\varepsilon^{-1}\leq \lambda/(2C)\leq 1/2$. Hence
    \begin{equation}
    \label{lemma, bound the overall bias, eq 10}
        \mathcal{S}_{X} \leq 6(1+\varepsilon)\mathcal{S}_{U} + \frac{124C\Lambda L}{\varepsilon}\mathcal{S}_{D}\leq 8\mathcal{S}_{U} + 124C^{2}\Lambda^{2}L\mathcal{S}_{D}.
    \end{equation}
    By the settings of $\varepsilon$ and $\eta_{1}$, we also have $192C\Lambda L^{2}\eta_{1}^{2}\varepsilon^{-1}\leq 3\lambda/(2C)$. Hence (\ref{lemma, bound the overall bias, eq 9}) implies
    \begin{equation*}
        \left| \mathcal{B}\right| \leq \frac{8\lambda}{C} \mathcal{S}_{U} + \frac{3\lambda}{2C}\mathcal{S}_{X} + 148C^{2}\Lambda^{2}L\mathcal{S}_{D}.
    \end{equation*}
    Plugging (\ref{lemma, bound the overall bias, eq 10}) into the above inequality gives
    \begin{align*}
        \left|\mathcal{B}\right| &\leq \frac{20\lambda}{C}\mathcal{S}_{U} + \frac{3}{2\cdot 40}124C^{2}\Lambda^{2}L\mathcal{S}_{D} + 148C^{2}\Lambda^{2}L\mathcal{S}_{D} \\
        &\leq \frac{20\lambda}{C}\mathcal{S}_{U} + 153C^{2}\Lambda^{2}L\mathcal{S}_{D} \leq \frac{\lambda}{2}\mathcal{S}_{U} + 153C^{2}\Lambda^{2}L\mathcal{S}_{D},
    \end{align*}
    which is the desired inequality.
\end{proof}

\subsection{A bound on the inexactness of the shadow sequence iteration}

\begin{lemma}
\label{lemma, weighted variance at optimum}
    We have
    \begin{equation*}
        \sum_{t=1}^{T}\frac{\alpha_{t}\mathbb{E}\Vert d_{t}\Vert^2}{H_{t}} \leq \sum_{k=1}^{K}\frac{n\eta_{k}^{3}\sigma_{*}^{2}}{A_{k}},
    \end{equation*}
    where $A_{k}$ is previously defined as $A_{k}=\sum_{j=k}^{K}\eta_{j}$.
\end{lemma}
\begin{proof}
    We have
    \begin{align*}
        \sum_{t=1}^{T} \frac{\alpha_{t}\mathbb{E}\Vert d_{t}\Vert^2}{H_{t}} &= \sum_{k=1}^{K}\eta_{k}\sum_{p=1}^{n}\frac{\mathbb{E}\Vert d_{k}^{p}\Vert^2}{H_{\tau_{k,p}}} = \sigma_{*}^{2}\sum_{k=1}^{K}\eta_{k} \sum_{p=1}^{n} \frac{(p-1)(n-p+1)}{n-1}\frac{\eta_{k}^{2}}{H_{\tau_{k,p}}} \\
        &\leq \sigma_{*}^{2}\sum_{k=1}^{K}\eta_{k} \sum_{p=1}^{n} \frac{(p-1)(n-p+1)}{n-1}\frac{\eta_{k}^{2}}{nA_{k+1}+(n-p+1)\eta_{k}} \\
        &= \sigma_{*}^{2}\sum_{k=1}^{K}\eta_{k}^{2}\sum_{p=1}^{n}\frac{p-1}{n-1}\frac{n-p+1}{n(A_{k+1}/\eta_{k})+(n-p+1)} \\
        &\leq \sigma_{*}^{2}\sum_{k=1}^{K}\eta_{k}^{2}\sum_{p=1}^{n}\frac{p-1}{n-1}\frac{1}{(A_{k+1}/\eta_{k})+1} \\
        &= \sigma_{*}^{2}\sum_{k=1}^{K}\eta_{k}^{2}\frac{n}{2(A_{k+1}/\eta_{k})+2} \leq \sum_{k=1}^{K}\frac{n\eta_{k}^{3}\sigma_{*}^{2}}{A_{k}},
    \end{align*}
    where the second equality is by Lemma \ref{lemma, SWR lemma}.
\end{proof}

\section{Proof of Theorem \ref{theorem, the hard bound}}

\begin{proof}
    By Lemma \ref{lemma, one step lemma}, we have
    \begin{align}
    \notag
        B(u_{t})-B(z_{t+1}) &\leq \frac{\Vert u_{t}-z_{t+1}\Vert^2 - \Vert u_{t+1}-z_{t+1}\Vert^2}{2\alpha_{t}} + \Omega_{i_{t}}(z_{t+1})-\Omega_{i_{t}}(x_{t}) \\
    \label{theorem, the hard bound, eq 0}
        & \quad +(2L\alpha_{t}+\rho) b_{i_{t}}(z_{t+1}) + 2\rho B(u_{t}) + \frac{3L}{\rho} \Vert d_{t}\Vert^2,
    \end{align}
    where $\rho \coloneqq \frac{\lambda}{20}$ and the sequence $\{z_{t}\}_{t}$ is defined as follows
    \begin{equation}
    \label{theorem, the hard bound, eq 1}
        z_{t+1} \coloneqq \frac{v_{t}}{v_{t+1}}z_{t}+\left(1-\frac{v_{t}}{v_{t+1}} \right)u_{t}, \ z_{1} \coloneqq x_{*}.
    \end{equation}
    Multiplying both sides of (\ref{theorem, the hard bound, eq 0}) by $c_{t}\coloneqq\alpha_{t}v_{t+1}$ gives
    \begin{align}
    \notag
        c_{t}\left(B(u_{t})-B(z_{t+1})\right) &\leq \frac{v_{t+1}\Vert u_{t}-z_{t+1}\Vert^2 - v_{t+1}\Vert u_{t+1}-z_{t+1}\Vert^2}{2} + c_{t}\Omega_{i_{t}}(z_{t+1}) \\
    \label{theorem, the hard bound, eq 2}
        & \quad -c_{t}\Omega_{i_{t}}(x_{t})+(2L\alpha_{t}+\rho)c_{t} b_{i_{t}}(z_{t+1}) + 2\rho c_{t} B(u_{t}) + \frac{3L c_{t}}{\rho} \Vert d_{t}\Vert^2.
    \end{align}
    By (\ref{theorem, the hard bound, eq 1}), we have $z_{t+1}-u_{t}=(v_{t}/v_{t+1})(z_{t}-u_{t})$, then
    \begin{equation*}
        v_{t+1}\Vert u_{t}-z_{t+1}\Vert^2 = v_{t+1}\frac{v_{t}^{2}}{v_{t+1}^{2}}\Vert u_{t}-z_{t}\Vert^2 \leq v_{t}\Vert u_{t}-z_{t}\Vert^2,
    \end{equation*}
    where the inequality is by $v_{t+1}\geq v_{t}$. Hence
    \begin{equation*}
        \frac{v_{t+1}\Vert u_{t}-z_{t+1}\Vert^2 - v_{t+1}\Vert u_{t+1}-z_{t+1}\Vert^2}{2} \leq \frac{v_{t}\Vert u_{t}-z_{t}\Vert^2 - v_{t+1}\Vert u_{t+1}-z_{t+1}\Vert^2}{2},
    \end{equation*}
    and the RHS of the above inequality will telescope in the summation w.r.t. $t$. By (\ref{theorem, the hard bound, eq 1}), 
    \begin{equation}
    \label{theorem, the hard bound, eq 3}
        z_{t+1} = \frac{v_{1}}{v_{t+1}}x_{*} + \sum_{s=1}^{t}\frac{\Delta v_{s}}{v_{t+1}} u_{s},
    \end{equation}
    i.e., $z_{t+1}$ is a convex combination of the points $x_{*}, u_{1}, u_{2}, \cdots, u_{t}$. Hence
    \begin{align}
    \notag
        c_{t}\Omega_{i_{t}}(z_{t+1}) &= \alpha_{t} v_{t+1}\Omega_{i_{t}}(z_{t+1}) = \alpha_{t}v_{t+1}b_{i_{t}}(z_{t+1}) - c_{t}B(z_{t+1}) \\
    \notag
        &\leq \alpha_{t}v_{1}b_{i_{t}}(x_{*}) + \alpha_{t}\sum_{s=1}^{t}\Delta v_{s} b_{i_{t}}(u_{s}) - c_{t}B(z_{t+1}) \\
    \label{theorem, the hard bound, eq 4}
        &\leq \alpha_{t}\sum_{s=1}^{t}\Delta v_{s} b_{i_{t}}(u_{s}) - c_{t}B(z_{t+1}).
    \end{align}
    Summing (\ref{theorem, the hard bound, eq 2}) from $t=1$ to $M$ and taking expectation give
    \begin{align}
    \notag
        \sum_{t=1}^{M} c_{t}\mathbb{E}B(u_{t}) &\leq \frac{v_{1}\Vert x_{1}-x_{*}\Vert^2}{2} + \sum_{t=1}^{M}\alpha_{t} \sum_{s=1}^{t}\Delta v_{s} \mathbb{E} b_{i_{t}}(u_{s}) - \sum_{t=1}^{M} c_{t}\mathbb{E}\Omega_{i_{t}}(x_{t}) \\
    \label{theorem, the hard bound, eq 5}
        & \quad + \tilde{R}_{Z} + 2\rho \sum_{t=1}^{M}c_{t}\mathbb{E}B(u_{t}) + \tilde{R}_{D}, 
    \end{align}
    where the quantities $\tilde{R}_{Z}$ and $\tilde{R}_{D}$ are defined as follows
    \begin{equation*}
        \tilde{R}_{Z} \coloneqq \sum_{t=1}^{M} (2L\alpha_{t}+\rho)c_{t}\mathbb{E}b_{i_{t}}(z_{t+1}), \ \tilde{R}_{D} \coloneqq \frac{3L}{\rho} \sum_{t=1}^{M} c_{t}\mathbb{E} \Vert d_{t}\Vert^2.
    \end{equation*}
    For the second term in the RHS of (\ref{theorem, the hard bound, eq 5}), we have
    \begin{align*}
        \sum_{t=1}^{M}\alpha_{t}\sum_{s=1}^{t}\Delta v_{s}b_{i_{t}}(u_{s}) &= \sum_{t=1}^{M} \alpha_{t}\sum_{s=1}^{t}\Delta v_{s}B(u_{s}) + \sum_{t=1}^{M}\alpha_{t}\sum_{s=1}^{t}\Delta v_{s}\Omega_{i_{t}}(u_{s}) \\
        &= \sum_{s=1}^{M} \Delta v_{s}B(u_{s})\sum_{t=s}^{M}\alpha_{t} + \sum_{t=1}^{M}\alpha_{t}\sum_{s=1}^{t}\Delta v_{s}\Omega_{i_{t}}(u_{s}) \\
        &=\sum_{s=1}^{M}\Delta v_{s}H_{s}B(u_{s}) + \sum_{t=1}^{M}\alpha_{t}\sum_{s=1}^{t}\Delta v_{s}\Omega_{i_{t}}(u_{s}) \\
        &= \sum_{t=1}^{M} a c_{t} B(u_{t}) + \sum_{t=1}^{M}\alpha_{t}\sum_{s=1}^{t}\Delta v_{s}\Omega_{i_{t}}(u_{s}),
    \end{align*}
    where the second equality is by interchanging the order of summation, and the last equality is by Definition \ref{def, auxiliary seq}. Hence
    \begin{equation}
    \label{theorem, the hard bound, eq 6}
        \mathbb{E} \sum_{t=1}^{M}\alpha_{t}\sum_{s=1}^{t}\Delta v_{s}b_{i_{t}}(u_{s}) = a \sum_{t=1}^{M}c_{t}\mathbb{E}B(u_{t}) + \sum_{t=1}^{M}\alpha_{t}\sum_{s=1}^{t}\Delta v_{s}\mathbb{E}\Omega_{i_{t}}(u_{s}).
    \end{equation}
    Denote $D\coloneqq \Vert x_{1}-x_{*}\Vert$, recall the definitions in Lemma \ref{lemma, bound the overall bias} and define 
    \begin{equation*}
        \mathcal{S}_{U}^{+}\coloneqq \sum_{t=1}^{M}c_{t}\mathbb{E}B(u_{t}) = \mathcal{S}_{U}+c_{M}\mathbb{E}B(u_{M}),
    \end{equation*}
    plug (\ref{theorem, the hard bound, eq 6}) into (\ref{theorem, the hard bound, eq 5}) and note that $\lambda = 1-a$, we have
    \begin{equation}
    \label{theorem, the hard bound, eq 7}
        \lambda \mathcal{S}_{U}^{+} \leq \frac{v_{1}D^{2}}{2} + \mathcal{B} +\tilde{R}_{Z} +2\rho \mathcal{S}_{U}^{+} + \tilde{R}_{D}.
    \end{equation}
    Denote $\tilde{\rho} \coloneqq \max_{1\leq t\leq M} (2L\alpha_{t}+\rho)=2L\eta_{1}+\rho$. We have
    \begin{align*}
        \tilde{R}_{Z} &= \sum_{t=1}^{M} (2L\alpha_{t}+\rho) c_{t}\mathbb{E} b_{i_{t}}(z_{t+1}) \leq \tilde{\rho}\sum_{t=1}^{M}\alpha_{t} v_{t+1}\mathbb{E}b_{i_{t}}(z_{t+1}) \\
        &\leq \tilde{\rho} \sum_{t=1}^{M} \alpha_{t} \sum_{s=1}^{t} \Delta v_{s}\mathbb{E}b_{i_{t}}(u_{s}) = \tilde{\rho} \sum_{t=1}^{M}\alpha_{t} \sum_{s=1}^{t}\Delta v_{s} \left[\mathbb{E}B(u_{s})+ \mathbb{E}\Omega_{i_{t}}(u_{s})\right] \\
        &= \tilde{\rho} \sum_{s=1}^{M}\Delta v_{s}\mathbb{E}B(u_{s})\sum_{t=s}^{M} \alpha_{t} + \tilde{\rho} \sum_{s=1}^{M}\Delta v_{s}\sum_{t=s}^{M}\alpha_{t} \mathbb{E}\Omega_{i_{t}}(u_{s}) \\
        &\leq \tilde{\rho} \sum_{s=1}^{M} \Delta v_{s} H_{s} \mathbb{E}B(u_{s}) + \tilde{\rho} \left| \mathcal{B} \right|_{\rm{up}} \leq \tilde{\rho} \sum_{t=1}^{M}c_{t}\mathbb{E}B(u_{t}) + \tilde{\rho}\left| \mathcal{B} \right|_{\rm{up}},
    \end{align*}
    where the second inequality is by (\ref{theorem, the hard bound, eq 3}), the third equality is by interchanging the order of summation, and the last inequality is by Definition \ref{def, auxiliary seq} (where it is shown that $\Delta v_{s} H_{s} = a\alpha_{s}v_{s+1} \leq \alpha_{s}v_{s+1} =c_{s}$)\footnote{For the third inequality, note that $|\mathcal{B}|_{\rm{up}}$ in Lemma \ref{lemma, bound the overall bias} contains a summand which is the upper bound of $\sum_{s=1}^{M}\Delta v_{s}\sum_{t=s}^{M}\alpha_{t}|\mathbb{E}\Omega_{i_{t}}(u_{s})|$.}. Recall the definition of $\mathcal{S}_{U}^{+}$, the above inequality can be written as
    \begin{equation}
    \label{theorem, the hard bound, eq 8}
        \tilde{R}_{Z} \leq \tilde{\rho} \mathcal{S}_{U}^{+} + \tilde{\rho} \left| \mathcal{B} \right|_{\rm{up}}.
    \end{equation}
    Plugging (\ref{theorem, the hard bound, eq 8}) into (\ref{theorem, the hard bound, eq 7}) gives
    \begin{align*}
        \lambda \mathcal{S}_{U}^{+} &\leq \frac{v_{1}D^{2}}{2} + (1+\tilde{\rho})\left| \mathcal{B}\right|_{\rm{up}} + 3\tilde{\rho}\mathcal{S}_{U}^{+} + \tilde{R}_{D} \\
        &\leq \frac{v_{1}D^{2}}{2} + \left[2^{-1}(1+\tilde{\rho})\lambda + 3\tilde{\rho} \right]\mathcal{S}_{U}^{+} + \left[153(1+\tilde{\rho})C^{2}\Lambda^{2}L + 3L\rho^{-1} \right]\mathcal{S}_{D} \\
        &\leq \frac{v_{1}D^{2}}{2} + \frac{17\lambda}{20}\mathcal{S}_{U}^{+} + 171C^{2}\Lambda^{2}L\mathcal{S}_{D},
    \end{align*}
    where the second inequality is by Lemma \ref{lemma, bound the overall bias}, and the last inequality is by the settings of $\eta_{1}$ and $\rho$. Rearranging the above inequality gives
    \begin{equation}
    \label{theorem, the hard bound, eq 9}
        \lambda \sum_{t=1}^{M}\alpha_{t}v_{t+1}\mathbb{E}B(u_{t}) = \lambda \mathcal{S}_{U}^{+} \leq 5v_{1}D^{2} + 1710C^{2}\Lambda^{2}L\mathcal{S}_{D}.
    \end{equation}
    By Definition \ref{def, auxiliary seq}, $\alpha_{M}=\beta, v_{M+1}=1$. By Definition \ref{def, basics} and Definition \ref{def, global index}, $u_{M}=x_{M}$ is the last iterate of the algorithm. Dropping all terms except for $\lambda \alpha_{M} v_{M+1}\mathbb{E}B(u_{M})$ in the LHS of (\ref{theorem, the hard bound, eq 9}) gives
    \begin{equation*}
        \lambda \beta \mathbb{E}B(x_{M}) \leq 5v_{1}D^{2} + 1710C^{2}\Lambda^{2}L \sum_{t=1}^{M}\alpha_{t}v_{t+1}\mathbb{E}\Vert d_{t}\Vert^2.
    \end{equation*}
    Dividing both sides by $\lambda \beta$ and using Lemma \ref{lemma, aux seq property} (note that $d_{M}=0$ by Definition \ref{def, global index}), we have
    \begin{equation*}
        \mathbb{E}B(x_{M}) \leq \frac{5CD^{2}}{H_{1}} + 1710C^{3}\Lambda^{2}L\sum_{t=1}^{T}\frac{\alpha_{t}}{H_{t}}\mathbb{E}\Vert d_{t}\Vert^2.
    \end{equation*}
    By Lemma \ref{lemma, weighted variance at optimum}, we further have
    \begin{equation*}
        \mathbb{E} \left[f(x_{M})-f_{*} \right] \leq \frac{5CD^{2}}{H_{1}} + 1710C^{3}\Lambda^{2}L\sum_{k=1}^{K} \frac{n\sigma_{*}^{2}\eta_{k}^{3}}{A_{k}},
    \end{equation*}
    which is the desired inequality.
\end{proof}

\section{Technical lemmas for the proof of Theorem \ref{theorem, the other bound}}

In the subsequent proofs, we will adopt many notation introduced earlier. The reader could refer to Definition \ref{def, basics}, Definition \ref{def, global index}, and Definition \ref{def, auxiliary seq} for them. Also, we mention some notations introduced in the proceeding proofs that will still be used subsequently. The reference point sequence $\{z_{t}\}_{t}$ is still defined as (\ref{theorem, the hard bound, eq 1}) with $u_{t}$ replaced by $x_{t}$, and $A_{k} = \sum_{j=k}^{K}\eta_{j}$ is the sum of the stepsizes used from the  $k^{\rm{th}}$ epoch to the $K^{\rm{th}}$ epoch. Also, the following new definition is needed.

\begin{definition}
\label{def, the other bound, basics}
    Define\footnote{We use notation $G$ rather than $B$ here to emphasize that $G$ is the finite-sum counterpart of $g_{i}$.}
    \begin{equation*}
        g_{i}(x)\coloneqq f_{i}(x)-f_{i}(x_{*}), \ G(x)\coloneqq f(x)-f(x_{*}), \ \tilde{\Omega}_{i}(x)\coloneqq g_{i}(x)-G(x),
    \end{equation*}
    where $i\in [n]$. It is immediate from the definitions that $G(x)=n^{-1}\sum_{i=1}^{n}g_{i}(x)$.
\end{definition}

Since the subsequent proofs share many common arguments with former proofs, we omit corresponding comments for brevity.

\subsection{A one-step descent inequality}

\begin{lemma}
\label{lemma, the other bound, one step inequality}
    Consider $x^{+}=x-\alpha \nabla f_{i}(x)$ where $0<\alpha\leq 1/(6L)$. For any $z\in\mathbb{R}^{d}$, we have
    \begin{equation*}
        G(x)-G(z) \leq \frac{\Vert x-z\Vert^2 - \Vert x^{+}-z\Vert^2}{2\alpha} + \tilde{\Omega}_{i}(z) - \tilde{\Omega}_{i}(x) -\frac{1}{2}b_{i}(z,x) + 3L\alpha b_{i}(z) + \frac{3\alpha}{2}\Vert q_{i}\Vert^2,
    \end{equation*}
    where $b_{i}(x,y)\coloneqq D_{f_{i}}(x,y)$ is the Bregman divergence associated with the component function $f_{i}$.
\end{lemma}
\begin{proof}
    Since $x^{+}=x-\alpha\nabla f_{i}(x)$,
    \begin{equation*}
        \Vert x^{+}-z\Vert^2 = \Vert x-\alpha\nabla f_{i}(x)-z\Vert^2 = \Vert x-z\Vert^2 -2\alpha \langle \nabla f_{i}(x),x-z\rangle + \alpha^2 \Vert \nabla f_{i}(x)\Vert^2.
    \end{equation*}
    By the definition of the Bregman divergence,
    \begin{equation*}
        b_{i}(z,x) = f_{i}(z) - f_{i}(x) -\langle \nabla f_{i}(x),z-x\rangle.
    \end{equation*}
    Combining the above two equations gives
    \begin{equation}
    \label{lemma, the other bound, one step inequality, eq 1}
        f_{i}(x) - f_{i}(z) = \frac{\Vert x-z\Vert^2 - \Vert x^{+}-z\Vert^2}{2\alpha} - b_{i}(z,x) + \frac{\alpha}{2} \Vert \nabla f_{i}(x)\Vert^2.
     \end{equation}
     Note that
     \begin{align*}
         \Vert \nabla f_{i}(x)\Vert^2 &\leq 3\Vert \nabla f_{i}(x)-\nabla f_{i}(z)\Vert^2 + 3\Vert \nabla f_{i}(z)-\nabla f_{i}(x_{*})\Vert^2 + 3\Vert \nabla f_{i}(x_{*})\Vert^2 \\
         &\leq 6L b_{i}(z,x) + 3\Vert \nabla b_{i}(z)\Vert^2 + 3\Vert q_{i}\Vert^2 
         \leq 6L b_{i}(z,x) + 6L b_{i}(z) +3\Vert q_{i}\Vert^2,
     \end{align*}
     where the second inequality is by convexity and $L$-smoothness of $f_{i}$ and Definition \ref{def, basics}, and the last inequality is by Lemma \ref{lemma, special func property}. Plugging the above inequality into (\ref{lemma, the other bound, one step inequality, eq 1}) gives
     \begin{align}
     \notag
         f_{i}(x) - f_{i}(z) &\leq \frac{\Vert x-z\Vert^2 - \Vert x^{+}-z\Vert^2}{2\alpha} + \left(3L\alpha - 1 \right)b_{i}(z,x) + 3L\alpha b_{i}(z) + \frac{3\alpha}{2}\Vert q_{i}\Vert^2 \\
    \label{lemma, the other bound, one step inequality, eq 2}
         &\leq \frac{\Vert x-z\Vert^2 - \Vert x^{+}-z\Vert^2}{2\alpha} - \frac{1}{2}b_{i}(z,x) + 3L\alpha b_{i}(z) +\frac{3\alpha}{2}\Vert q_{i}\Vert^2,
     \end{align}
     where the second inequality is by the stepsize condition $6L\alpha\leq 1$. Also note that
     \begin{equation}
     \label{lemma, the other bound, one step inequality, eq 3}
         f_{i}(x)-f_{i}(z) = g_{i}(x)-g_{i}(z) = G(x)-G(z)+\tilde{\Omega}_{i}(x)-\tilde{\Omega}_{i}(z).
     \end{equation}
     Plugging (\ref{lemma, the other bound, one step inequality, eq 3}) into (\ref{lemma, the other bound, one step inequality, eq 2}) gives the desired equation.
\end{proof}

Define the ghost step as follows
\begin{equation*}
    x_{M+1} = x_{M} - \eta_{1}\nabla f(x_{M}) = x_{M} - \beta \nabla f(x_{M}).
\end{equation*}
We then have the following one-step descent inequality for this ghost update.

\begin{lemma}
\label{lemma, the other bound, one step inequality, ghost update}
    Consider $x_{M+1}=x_{M}-\beta \nabla f(x_{M})$ where $0<\beta\leq 1/(6L)$. For any $z\in\mathbb{R}^{d}$, we have
    \begin{equation*}
        G(x_{M})-G(z) \leq \frac{\Vert x_{M}-z\Vert^2-\Vert x_{M+1}-z\Vert^2}{2\beta}+\beta L G(x_{M}).
    \end{equation*}
\end{lemma}
\begin{proof}
    Since $x_{M+1}=x_{M}-\beta \nabla f(x_{M})$,
    \begin{equation*}
        \Vert x_{M+1}-z\Vert^2 = \Vert x_{M}-\beta \nabla f(x_{M})-z\Vert^2 = \Vert x_{M}-z\Vert^2 -2\beta \langle \nabla f(x_{M}),x_{M}-z\rangle + \beta^{2}\Vert \nabla f(x_{M})\Vert^2.
    \end{equation*}
    By the definition of the Bregman divergence,
    \begin{equation*}
        D_{f}(z,x_{M}) = f(z)-f(x_{M})-\langle \nabla f(x_{M}),z-x_{M}\rangle.
    \end{equation*}
    Combining the above two equations gives
    \begin{equation*}
        f(x_{M})-f(z) = \frac{\Vert x_{M}-z\Vert^2-\Vert x_{M+1}-z\Vert^2}{2\beta} - D_{f}(z,x_{M}) + \frac{\beta}{2}\Vert \nabla f(x_{M})\Vert^2.
    \end{equation*}
    Since
    \begin{equation*}
        f(x_{M})-f(z)=G(x_{M})-G(z), \ \Vert \nabla f(x_{M})\Vert^2 \leq 2L(f(x_{M})-f_{*})=2LG(x_{M}),
    \end{equation*}
    we have
    \begin{equation*}
        G(x_{M})-G(z) \leq \frac{\Vert x_{M}-z\Vert^2-\Vert x_{M+1}-z\Vert^2}{2\beta}+\beta L G(x_{M}),
    \end{equation*}
    which is the desired equation.
\end{proof}

\subsection{Bounds on component-function deviations}

\begin{lemma}
    \label{lemma, the other bound, RR component vs finite-sum}
    Assume $1\leq p \leq r \leq n$ and $L\eta_{k}\leq 1$ in the algorithm. Then the following inequalities hold
    \begin{align*}
        \left| \mathbb{E}\tilde{\Omega}_{\pi_{k}^{r}}(x_{k}^{p})\right| &\leq 4\eta_{k} \left[\mathbb{E}\Vert \nabla f_{\pi_{k}^{p}}(x_{k}^{p})\Vert^2 + \frac{1}{n}\sum_{j=1}^{p-1}\mathbb{E}\Vert \nabla f_{\pi_{k}^{j}}(x_{k}^{j})\Vert^2 + \sigma_{*}^{2}\right], \\
        \left| \mathbb{E}\Omega_{\pi_{k}^{r}}(x_{k}^{p})\right| &\leq 4\eta_{k} \left[\mathbb{E}\Vert \nabla f_{\pi_{k}^{p}}(x_{k}^{p})\Vert^2 + \frac{1}{n}\sum_{j=1}^{p-1}\mathbb{E}\Vert \nabla f_{\pi_{k}^{j}}(x_{k}^{j})\Vert^2 + \sigma_{*}^{2}\right].
    \end{align*}
\end{lemma}
\begin{proof}
    Conditioned on $\pi_{1}, \cdots, \pi_{k-1}, \pi_{k}^{1}, \pi_{k}^{2},\cdots, \pi_{k}^{p-1}$, the r.v. $\pi_{k}^{p}$ has the same distribution as $\pi_{k}^{r}$, hence
    \begin{equation*}
        \mathbb{E} \tilde{\Omega}_{\pi_{k}^{r}}(x_{k}^{p}) = \mathbb{E}\tilde{\Omega}_{\pi_{k}^{p}}(x_{k}^{p}).
    \end{equation*}
    In a way similar to the beginning of the proof of Lemma \ref{lemma, RR component vs finite-sum}, we have
    \begin{equation}
    \label{lemma, the other bound, RR component vs finite-sum, eq 1}
        \left| \mathbb{E}\tilde{\Omega}_{\pi_{k}^{p}}(x_{k}^{p}) \right| \leq \frac{1}{n}\sum_{1\leq j<p} \left| \mathbb{E}\left[g_{\pi_{k}^{p}}(x_{k}^{p}(p,j))-g_{\pi_{k}^{p}}(x_{k}^{p}) \right] \right| \leq \frac{1}{n}\sum_{1\leq j<p} \mathbb{E} \left| g_{\pi_{k}^{p}}(x_{k}^{p}(p,j))-g_{\pi_{k}^{p}}(x_{k}^{p}) \right|,
    \end{equation}
    where the second inequality is by Jensen's inequality. We have
    \begin{align*}
        \left| g_{\pi_{k}^{p}}(x_{k}^{p}(p,j))-g_{\pi_{k}^{p}}(x_{k}^{p}) \right| &\leq \Vert \nabla g_{\pi_{k}^{p}}(x_{k}^{p})\Vert \Vert x_{k}^{p}(p,j)-x_{k}^{p}\Vert + \frac{L}{2} \Vert x_{k}^{p}(p,j)-x_{k}^{p}\Vert^2 \\
        &\leq \frac{\eta_{k}}{2}\Vert \nabla g_{\pi_{k}^{p}}(x_{k}^{p})\Vert^2 + \frac{1}{2\eta_{k}}\Vert x_{k}^{p}(p,j)-x_{k}^{p}\Vert^2 + \frac{L}{2}\Vert x_{k}^{p}(p,j)-x_{k}^{p}\Vert^2 \\
        &\leq \eta_{k}\Vert \nabla g_{\pi_{k}^{p}}(x_{k}^{p})\Vert^2 + \eta_{k}\Vert q_{\pi_{k}^{p}}\Vert^2 + \frac{\eta_{k}}{2}\Vert \nabla f_{\pi_{k}^{p}}(x_{k}^{j})-\nabla f_{\pi_{k}^{j}}(x_{k}^{j})\Vert^2 \\
        & \quad + \frac{L\eta_{k}^{2}}{2}\Vert \nabla f_{\pi_{k}^{p}}(x_{k}^{j})-\nabla f_{\pi_{k}^{j}}(x_{k}^{j})\Vert^2 \\
        &\leq \eta_{k}\Vert \nabla f_{\pi_{k}^{p}}(x_{k}^{p})\Vert^2 + \eta_{k}\Vert q_{\pi_{k}^{p}}\Vert^2 + \eta_{k} \Vert \nabla f_{\pi_{k}^{p}}(x_{k}^{j})-\nabla f_{\pi_{k}^{j}}(x_{k}^{j})\Vert^2 \\
        &\leq \eta_{k}\Vert \nabla f_{\pi_{k}^{p}}(x_{k}^{p})\Vert^2 + 2\eta_{k}\Vert \nabla f_{\pi_{k}^{p}}(x_{k}^{j})\Vert^2 + 2\eta_{k} \Vert \nabla f_{\pi_{k}^{j}}(x_{k}^{j})\Vert^2 \\
        &\quad +\eta_{k}\Vert q_{\pi_{k}^{p}}\Vert^2,
    \end{align*}
    where the first inequality is by (\ref{lemma, special func property, eq 2}) (note that the proof of (\ref{lemma, special func property, eq 2}) does not require the function $h$ to achieve its minimum at $x_{*}$), the second inequality is by Young's inequality, the third inequality is by Lemma \ref{lemma, diff btw virtual and real seq}, and the fourth inequality is by the stepsize condition $L\eta_{k}\leq 1$\footnote{In the third inequality, the term $\eta_{k} \Vert q_{\pi_{k}^{p}}\Vert^2$ can actually be removed. However, since this term will appear in the upper bound of $|\mathbb{E}\Omega_{\pi_{k}^{r}}(x_{k}^{p})|$, we keep it so that we can have the same upper bounds for $|\mathbb{E}\tilde{\Omega}_{\pi_{k}^{r}}(x_{k}^{p})|$ and $|\mathbb{E}\Omega_{\pi_{k}^{r}}(x_{k}^{p})|$, which facilitates the subsequent calculations.}. Taking expectation in the above inequality gives
    \begin{equation*}
        \mathbb{E}\left| g_{\pi_{k}^{p}}(x_{k}^{p}(p,j))-g_{\pi_{k}^{p}}(x_{k}^{p}) \right| \leq 4\eta_{k}\mathbb{E} \Vert \nabla f_{\pi_{k}^{p}}(x_{k}^{p})\Vert^2 + 4\eta_{k} \mathbb{E}\Vert \nabla f_{\pi_{k}^{j}}(x_{k}^{j})\Vert^2 + 4\eta_{k}\sigma_{*}^{2},
    \end{equation*}
    where we use the fact that conditioned on $\pi_{1},\cdots,\pi_{k-1},\pi_{k}^{1},\cdots, \pi_{k}^{j-1}$, the r.v. $\pi_{k}^{j}$ has the same distribution as $\pi_{k}^{p}$ for $j<p$. Plugging the above inequality into (\ref{lemma, the other bound, RR component vs finite-sum, eq 1}) gives the first desired equation. By applying the above estimations almost verbatim, we get the upper bound of $| \mathbb{E}\Omega_{\pi_{k}^{r}}(x_{k}^{p})|$. This part of the proof is omitted for brevity.
\end{proof}

For notational convenience, define
\begin{align*}
    \tilde{\mathcal{S}}_{X} &\coloneqq \sum_{t=1}^{T}\alpha_{t}v_{t+1}\mathbb{E}G(x_{t}), \ \tilde{\mathcal{S}}_{D} \coloneqq \sum_{t=1}^{T}\alpha_{t}v_{t+1}\mathbb{E}b_{i_{t}}(z_{t+1},x_{t}), \\
    \tilde{\mathcal{S}}_{Z} &\coloneqq \sum_{t=1}^{T} \alpha_{t}v_{t+1}\mathbb{E}b_{i_{t}}(z_{t+1}), \ \tilde{\mathcal{S}}_{V} \coloneqq \sum_{t=1}^{T}\alpha_{t}^{2}v_{t+1}\mathbb{E}\Vert q_{i_{t}}\Vert^2 = \sigma_{*}^{2}\sum_{t=1}^{T}\alpha_{t}^{2}v_{t+1}.
\end{align*}
We would like to upper bound the following quantity
\begin{equation*}
    \tilde{\mathcal{B}} \coloneqq \sum_{s=1}^{T}\Delta v_{s}\sum_{t=s}^{T}\alpha_{t} \left| \mathbb{E}\tilde{\Omega}_{i_{t}}(x_{s})\right| + \sum_{t=1}^{T}\alpha_{t}v_{t+1}\left| \mathbb{E}\tilde{\Omega}_{i_{t}}(x_{t})\right|
\end{equation*}
with $\tilde{\mathcal{S}}_{X}$, $\tilde{\mathcal{S}}_{D}$, $\tilde{\mathcal{S}}_{Z}$, and $\tilde{\mathcal{S}}_{V}$.

\begin{lemma}
\label{lemma, the other bound, bounding the overall bias}
    Assume $L\eta_{k}\leq 1$ in the algorithm. Then
    \begin{equation*}
        \tilde{\mathcal{B}} \leq 60C\Lambda L\eta_{1} \left(\tilde{\mathcal{S}}_{D}+\tilde{\mathcal{S}}_{Z} \right) + 60C\Lambda \tilde{\mathcal{S}}_{V},
    \end{equation*}
    where the above upper bound of $\tilde{\mathcal{B}}$ is denoted as $\tilde{\mathcal{B}}_{\rm{up}}$. We also have
    \begin{equation*}
        \sum_{s=1}^{T}\Delta v_{s}\sum_{t=s}^{T}\alpha_{t} \left| \mathbb{E}\Omega_{i_{t}}(x_{s})\right| + \sum_{t=1}^{T}\alpha_{t}v_{t+1}\left| \mathbb{E}\Omega_{i_{t}}(x_{t})\right|\leq \tilde{\mathcal{B}}_{\rm{up}}.
    \end{equation*}
\end{lemma}
\begin{proof}
    Conditioned on $i_{1},i_{2},\cdots, i_{s-1}$, the r.v. $i_{t}$ has the same distribution as $i_{s}$ for $t$ and $s$ in the same epoch with $t\geq s$, hence in this case we have $\mathbb{E} \tilde{\Omega}_{i_{t}}(x_{s}) = \mathbb{E}\tilde{\Omega}_{i_{s}}(x_{s})$. When $t$ and $s$ ($t\geq s$) are not in the same epoch, we have $\mathbb{E}\tilde{\Omega}_{i_{t}}(x_{s}) = 0$.
    Define
    \begin{equation*}
        \tilde{\mathcal{B}}_{1} \coloneqq \sum_{s=1}^{T}\Delta v_{s}\sum_{t=s}^{T}\alpha_{t} \left| \mathbb{E}\tilde{\Omega}_{i_{s}}(x_{s})\right|, \ \tilde{\mathcal{B}}_{2} \coloneqq \sum_{t=1}^{T}\alpha_{t}v_{t+1}\left| \mathbb{E}\tilde{\Omega}_{i_{t}}(x_{t})\right|,
    \end{equation*}
    then $\tilde{\mathcal{B}} \leq \tilde{\mathcal{B}}_{1}+\tilde{\mathcal{B}}_{2}$. We have
    \begin{align*}
        \tilde{\mathcal{B}}_{1} &= \sum_{s=1}^{T}\Delta v_{s}\left| \mathbb{E}\tilde{\Omega}_{i_{s}}(x_{s})\right| \sum_{t=s}^{T}\alpha_{t} \leq  \sum_{s=1}^{T}\Delta v_{s}H_{s} \left| \mathbb{E}\tilde{\Omega}_{i_{s}}(x_{s})\right| \leq \sum_{t=1}^{T} \alpha_{t}v_{t+1} \left| \mathbb{E} \tilde{\Omega}_{i_{t}}(x_{t})\right| \\
        & \leq 4\sum_{t=1}^{T}\alpha_{t}^{2}v_{t+1}\left[\mathbb{E}\Vert \nabla f_{i_{t}}(x_{t})\Vert^2 + \frac{1}{n}\sum_{j<t, j\sim t}\mathbb{E}\Vert \nabla f_{i_{j}}(x_{j})\Vert^2 + \sigma_{*}^2 \right],
    \end{align*}
    where the second inequality is by Definition \ref{def, auxiliary seq},  the notation $j\sim t$ means the index $j$ is in the same epoch with the index $t$, and the last inequality is by Lemma \ref{lemma, the other bound, RR component vs finite-sum}. By Lemma \ref{lemma, the other bound, RR component vs finite-sum}, we also have
    \begin{equation*}
        \tilde{\mathcal{B}}_{2} \leq 4\sum_{t=1}^{T}\alpha_{t}^{2}v_{t+1}\left[\mathbb{E}\Vert \nabla f_{i_{t}}(x_{t})\Vert^2 + \frac{1}{n}\sum_{j<t, j\sim t}\mathbb{E}\Vert \nabla f_{i_{j}}(x_{j})\Vert^2 + \sigma_{*}^2 \right].
    \end{equation*}
    Combining the above two inequalities gives
    \begin{align}
    \notag
        \tilde{\mathcal{B}} &\leq 8\sum_{t=1}^{T}\alpha_{t}^{2}v_{t+1}\left[\mathbb{E}\Vert \nabla f_{i_{t}}(x_{t})\Vert^2 + \frac{1}{n}\sum_{j<t, j\sim t}\mathbb{E}\Vert \nabla f_{i_{j}}(x_{j})\Vert^2 + \sigma_{*}^2 \right] \\
    \notag
        &= 8\sum_{t=1}^{T} \alpha_{t}^{2}v_{t+1}\mathbb{E}\Vert \nabla f_{i_{t}}(x_{t})\Vert^2 + \frac{8}{n}\sum_{k=1}^{K}\eta_{k}^{2}\sum_{p=1}^{n}v_{\tau_{k,p}+1}\sum_{j=1}^{p-1}\mathbb{E}\Vert \nabla f_{\pi_{k}^{j}}(x_{k}^{j})\Vert^2 + 8\tilde{\mathcal{S}}_{V} \\
    \notag
        &= 8\sum_{t=1}^{T} \alpha_{t}^{2}v_{t+1}\mathbb{E}\Vert \nabla f_{i_{t}}(x_{t})\Vert^2 + \frac{8}{n}\sum_{k=1}^{K}\eta_{k}^{2} \sum_{j=1}^{n-1}\mathbb{E}\Vert \nabla f_{\pi_{k}^{j}}(x_{k}^{j})\Vert^2 \sum_{p=j+1}^{n}v_{\tau_{k,p}+1} + 8\tilde{\mathcal{S}}_{V} \\
    \notag
        &\leq 8\sum_{t=1}^{T} \alpha_{t}^{2}v_{t+1}\mathbb{E}\Vert \nabla f_{i_{t}}(x_{t})\Vert^2 + 8C\Lambda \sum_{k=1}^{K}\eta_{k}^{2}\sum_{j=1}^{n}v_{\tau_{k,j}+1}\mathbb{E}\Vert \nabla f_{\pi_{k}^{j}}(x_{k}^{j})\Vert^2 + 8\tilde{\mathcal{S}}_{V} \\
    \label{lemma, the other bound, bounding the overall bias, eq 1}
        &\leq 10C\Lambda \sum_{t=1}^{T}\alpha_{t}^{2}v_{t+1}\mathbb{E}\Vert \nabla f_{i_{t}}(x_{t})\Vert^2 + 10\tilde{\mathcal{S}}_{V},
    \end{align}
    where the second equality is by interchanging the order of summation, the second inequality is by Lemma \ref{lemma, aux seq property}, and the last inequality is by $C\geq 50$ and $\Lambda \geq 1$. Note that
    \begin{align*}
        \Vert \nabla f_{i_{t}}(x_{t})\Vert^2 &\leq 3\Vert \nabla f_{i_{t}}(x_{t})-\nabla f_{i_{t}}(z_{t+1})\Vert^2 + 3\Vert \nabla f_{i_{t}}(z_{t+1})-\nabla f_{i_{t}}(x_{*})\Vert^2 + 3\Vert \nabla f_{i_{t}}(x_{*})\Vert^2 \\
        &\leq 6Lb_{i_{t}}(z_{t+1},x_{t}) + 6Lb_{i_{t}}(z_{t+1}) + 3\Vert q_{i_{t}}\Vert^2,
    \end{align*}
    where the last inequality is by co-coercivity of $f_{i_{t}}$ and Definition \ref{def, basics}. Taking expectation in the above inequality gives
    \begin{equation}
    \label{lemma, the other bound, bounding the overall bias, eq 2}
        \mathbb{E}\Vert \nabla f_{i_{t}}(x_{t})\Vert^2 \leq 6L\mathbb{E}b_{i_{t}}(z_{t+1},x_{t}) + 6L\mathbb{E}b_{i_{t}}(z_{t+1}) + 3\sigma_{*}^{2}.
    \end{equation}
    Plugging (\ref{lemma, the other bound, bounding the overall bias, eq 2}) into (\ref{lemma, the other bound, bounding the overall bias, eq 1}) gives
    \begin{align*}
        \tilde{\mathcal{B}} &\leq 60C\Lambda L\eta_{1} \sum_{t=1}^{T}\alpha_{t}v_{t+1}\mathbb{E} b_{i_{t}}(z_{t+1},x_{t}) + 60C\Lambda L\eta_{1} \sum_{t=1}^{T}\alpha_{t}v_{t+1}\mathbb{E}b_{i_{t}}(z_{t+1}) + 40C\Lambda \tilde{\mathcal{S}}_{V} \\
        &\leq 60C\Lambda L\eta_{1} \left(\tilde{S}_{D}+\tilde{S}_{Z}\right) + 60C\Lambda \tilde{\mathcal{S}}_{V},
    \end{align*}
    where the first inequality is by $C \geq 50$ and $\Lambda \geq 1$. Since we have the same upper bound for $\left| \mathbb{E}\tilde{\Omega}_{\pi_{k}^{r}}(x_{k}^{p})\right|$ and $\left| \mathbb{E}\Omega_{\pi_{k}^{r}}(x_{k}^{p})\right|$ in Lemma \ref{lemma, the other bound, RR component vs finite-sum}, the second desired inequality of this lemma is proved.
\end{proof}

\begin{lemma}
\label{lemma, the other bound, bound S at Z}
    Assume the initial stepsize $\eta_{1}$ satisfies $120C\Lambda L\eta_{1}\leq 1$. Then
    \begin{align*}
        \tilde{\mathcal{B}} &\leq 120 C\Lambda L\eta_{1}\left(\tilde{\mathcal{S}}_{D} + \tilde{\mathcal{S}}_{X} \right) + 120 C\Lambda \tilde{\mathcal{S}}_{V}, \\
        \tilde{\mathcal{S}}_{Z} &\leq 2\left(\tilde{\mathcal{S}}_{D} + \tilde{\mathcal{S}}_{X} \right) + 120 C\Lambda \tilde{\mathcal{S}}_{V}.
    \end{align*}
\end{lemma}
\begin{proof}
    By (\ref{theorem, the hard bound, eq 3}), we have
    \begin{equation*}
        v_{t+1}b_{i_{t}}(z_{t+1}) \leq v_{1}b_{i_{t}}(x_{*}) + \sum_{s=1}^{t}\Delta v_{s} b_{i_{t}}(x_{s}) = \sum_{s=1}^{t}\Delta v_{s} b_{i_{t}}(x_{s}).
    \end{equation*}
    Then
    \begin{align*}
        \tilde{\mathcal{S}}_{Z} &= \sum_{t=1}^{T}\alpha_{t}v_{t+1}\mathbb{E}b_{i_{t}}(z_{t+1}) \leq \sum_{t=1}^{T}\alpha_{t}\sum_{s=1}^{t}\Delta v_{s}\mathbb{E}b_{i_{t}}(x_{s}) \\
        &= \sum_{t=1}^{T}\alpha_{t}\sum_{s=1}^{t}\Delta v_{s}(\mathbb{E}B(x_{s})+\mathbb{E}\Omega_{i_{t}}(x_{s}))  \\
        &\leq \sum_{t=1}^{T}\alpha_{t}\sum_{s=1}^{t}\Delta v_{s}\mathbb{E}G(x_{s}) + \sum_{t=1}^{T}\alpha_{t}\sum_{s=1}^{t}\Delta v_{s}\left| \mathbb{E}\Omega_{i_{t}}(x_{s})\right| \\
        &\leq \sum_{s=1}^{T}\Delta v_{s}\mathbb{E}G(x_{s})\sum_{t=s}^{T}
        \alpha_{t} + \sum_{t=1}^{T}\alpha_{t}\sum_{s=1}^{t}\Delta v_{s}\left|  \mathbb{E}\Omega_{i_{s}}(x_{s})\right| \\
        &\leq \sum_{s=1}^{T}H_{s}\Delta v_{s}\mathbb{E}G(x_{s})+\sum_{s=1}^{T}\Delta v_{s}\left|\mathbb{E}\Omega_{i_{s}}(x_{s}) \right|\sum_{t=s}^{T}\alpha_{t} \\
        &\leq a\sum_{t=1}^{T} \alpha_{t}v_{t+1}\mathbb{E}G(x_{t}) + a\sum_{t=1}^{T}\alpha_{t}v_{t+1}\left|\mathbb{E}\Omega_{i_{t}}(x_{t})\right|,
    \end{align*}
    where in the second inequality we use $B(x_{s})=G(x_{s})=f(x_{s})-f_{*}$, the third inequality is by $\mathbb{E}\Omega_{i_{t}}(x_{s})=\mathbb{E}\Omega_{i_{s}}(x_{s})$ when $t$ and $s$ ($s\leq t$) are in the same epoch and $\mathbb{E}\Omega_{i_{t}}(x_{s})=0$ when $t$ and $s$ ($s\leq t$) are not in the same epoch, and the last inequality is by Definition \ref{def, auxiliary seq}. Using the definition of $\tilde{\mathcal{S}}_{X}$ and Lemma \ref{lemma, the other bound, bounding the overall bias}, we have
    \begin{align*}
        \tilde{\mathcal{S}}_{Z} &\leq \tilde{\mathcal{S}}_{X} + \tilde{\mathcal{B}}_{\rm{up}}= \tilde{\mathcal{S}}_{X}+60C\Lambda L\eta_{1}(\tilde{\mathcal{S}}_{D}+\tilde{\mathcal{S}}_{Z})+60C\Lambda \tilde{\mathcal{S}}_{V} \\
        &\leq \tilde{\mathcal{S}}_{X} + \frac{1}{2}\tilde{\mathcal{S}}_{D} + \frac{1}{2}\tilde{\mathcal{S}}_{Z} + 60C\Lambda \tilde{\mathcal{S}}_{V},
    \end{align*}
    where the second inequality is by the setting $120C\Lambda L\eta_{1}\leq 1$. Rearranging the above inequality gives
    \begin{equation*}
        \tilde{\mathcal{S}}_{Z} \leq \tilde{\mathcal{S}}_{D} + 2\tilde{\mathcal{S}}_{X} + 120 C\Lambda \tilde{\mathcal{S}}_{V}.
    \end{equation*}
    By Lemma \ref{lemma, the other bound, bounding the overall bias} and the above equation,
    \begin{align*}
        \tilde{\mathcal{B}} &\leq 60C\Lambda L\eta_{1}\left(\tilde{\mathcal{S}}_{D}+\tilde{\mathcal{S}}_{Z} \right) + 60C\Lambda \tilde{\mathcal{S}}_{V}  \\
        &\leq 60C\Lambda L\eta_{1}\tilde{\mathcal{S}}_{D} + 60C\Lambda L\eta_{1}\tilde{\mathcal{S}}_{D} + 120 C\Lambda L\eta_{1}\tilde{\mathcal{S}}_{X} + 60 C\Lambda \tilde{\mathcal{S}}_{V} + 60C\Lambda \tilde{\mathcal{S}}_{V} \\
        &= 120C\Lambda L\eta_{1}\left(\tilde{\mathcal{S}}_{D}+\tilde{\mathcal{S}}_{X} \right) + 120C\Lambda \tilde{\mathcal{S}}_{V},
    \end{align*}
    where the second inequality is by the setting $120C\Lambda L\eta_{1}\leq 1$.
\end{proof}

\section{Proof of Theorem \ref{theorem, the other bound}}

\begin{proof}
    By Lemma \ref{lemma, the other bound, one step inequality}, for $t\in [T]$, we have
    \begin{align}
    \notag
        G(x_{t}) - G(z_{t+1}) &\leq \frac{\Vert x_{t}-z_{t+1}\Vert^2 - \Vert x_{t+1}-z_{t+1}\Vert^2}{2\alpha_{t}} + \tilde{\Omega}_{i_{t}}(z_{t+1}) - \tilde{\Omega}_{i_{t}}(x_{t}) \\
    \label{theorem, the other bound, eq 0}
        &\quad - \frac{1}{2}b_{i_{t}}(z_{t+1},x_{t}) + 3L\alpha_{t}b_{i_{t}}(z_{t+1}) + \frac{3\alpha_{t}}{2}\Vert q_{i_{t}}\Vert^2.
    \end{align}
    Multiplying both sides of (\ref{theorem, the other bound, eq 0}) by $c_{t}=\alpha_{t}v_{t+1}$ gives
    \begin{align}
    \notag
        c_{t}\left(G(x_{t}) - G(z_{t+1}) \right) &\leq \frac{v_{t+1}\Vert x_{t}-z_{t+1}\Vert^2 - v_{t+1}\Vert x_{t+1}-z_{t+1}\Vert^2}{2} + c_{t}\tilde{\Omega}_{i_{t}}(z_{t+1}) \\
    \label{theorem, the other bound, eq 2}
        &\quad - c_{t}\tilde{\Omega}_{i_{t}}(x_{t}) - \frac{c_{t}}{2}b_{i_{t}}(z_{t+1},x_{t}) + 3Lc_{t}\alpha_{t}b_{i_{t}}(z_{t+1}) + \frac{3c_{t}\alpha_{t}}{2}\Vert q_{i_{t}}\Vert^2.
    \end{align}
    By the argument used in Theorem \ref{theorem, the hard bound}, we have
    \begin{equation*}
        \frac{v_{t+1}\Vert x_{t}-z_{t+1}\Vert^2 - v_{t+1}\Vert x_{t+1}-z_{t+1}\Vert^2}{2} \leq \frac{v_{t}\Vert x_{t}-z_{t}\Vert^2 - v_{t+1}\Vert x_{t+1}-z_{t+1}\Vert^2}{2},
    \end{equation*}
    and the RHS of the above inequality will telescope in the summation w.r.t. $t$. By the definition of $\{z_{t}\}_{t}$, 
    \begin{equation}
    \label{theorem, the other bound, eq 3}
        z_{t+1} = \frac{v_{1}}{v_{t+1}}x_{*} + \sum_{s=1}^{t}\frac{\Delta v_{s}}{v_{t+1}} x_{s},
    \end{equation}
    i.e., $z_{t+1}$ is a convex combination of the points $x_{*}, x_{1}, x_{2}, \cdots, x_{t}$. Hence
    \begin{align*}
        c_{t}\tilde{\Omega}_{i_{t}}(z_{t+1}) &= \alpha_{t} v_{t+1}\tilde{\Omega}_{i_{t}}(z_{t+1}) = \alpha_{t}v_{t+1}g_{i_{t}}(z_{t+1}) - c_{t}G(z_{t+1}) \\
        &\leq \alpha_{t}v_{1}g_{i_{t}}(x_{*}) + \alpha_{t}\sum_{s=1}^{t}\Delta v_{s} g_{i_{t}}(x_{s}) - c_{t}G(z_{t+1}) \\
        &= \alpha_{t}\sum_{s=1}^{t}\Delta v_{s} g_{i_{t}}(x_{s}) - c_{t}G(z_{t+1}).
    \end{align*}
    By Lemma \ref{lemma, the other bound, one step inequality, ghost update}, we have
    \begin{equation*}
        G(x_{M})-G(z_{M+1}) \leq \frac{\Vert x_{M}-z_{M+1}\Vert^2-\Vert x_{M+1}-z_{M+1}\Vert^2}{2\beta}+\beta L G(x_{M}).
    \end{equation*}
    Multiplying the above equation by $\lambda c_{M}/2=\lambda \alpha_{M}v_{M+1}/2=\lambda\alpha_{M}/2=\lambda\beta/2$ gives (note that $v_{M+1}=1$)
    \begin{equation}
    \label{theorem, the other bound, eq for the ghost step}
        2^{-1}\lambda\beta \left( G(x_{M})-G(z_{M+1})\right) \leq \frac{v_{M+1}\Vert x_{M}-z_{M+1}\Vert^2}{2}+2^{-1}\lambda\beta^{2}LG(x_{M}),
    \end{equation}
    where the inequality is by first dropping the term $-v_{M+1}\Vert x_{M+1}-z_{M+1}\Vert^2$ and then using $0<2^{-1}\lambda<1$. Summing (\ref{theorem, the other bound, eq 2}) from $t=1$ to $T$ and adding (\ref{theorem, the other bound, eq for the ghost step}), using the above inequality, and taking expectation give us
    \begin{align}
    \notag
        \sum_{t=1}^{T}c_{t}\mathbb{E}G(x_{t}) + 2^{-1}\lambda\beta\mathbb{E}G(x_{M}) &\leq \frac{v_{1}\Vert x_{1}-x_{*}\Vert^2}{2} + \sum_{t=1}^{T}\alpha_{t}\sum_{s=1}^{t}\Delta v_{s}\mathbb{E}g_{i_{t}}(x_{s})
        - \sum_{t=1}^{T}c_{t}\mathbb{E}\tilde{\Omega}_{i_{t}}(x_{t}) \\
    \notag
        &\quad -\frac{1}{2}\sum_{t=1}^{T}c_{t}\mathbb{E}b_{i_{t}}(z_{t+1},x_{t}) + 3L\sum_{t=1}^{T}c_{t}\alpha_{t}\mathbb{E}b_{i_{t}}(z_{t+1}) \\
    \label{theorem, the other bound, eq 4}
        &\quad + \frac{3\sigma_{*}^2}{2}\sum_{t=1}^{T}c_{t}\alpha_{t} +  2^{-1}\lambda\beta \mathbb{E}G(z_{M+1}) +  2^{-1}\lambda\beta^{2}L\mathbb{E}G(x_{M}),
    \end{align}
    For the second term in the RHS of (\ref{theorem, the other bound, eq 4}), we have
    \begin{align*}
        \sum_{t=1}^{T}\alpha_{t}\sum_{s=1}^{t}\Delta v_{s}g_{i_{t}}(x_{s}) &= \sum_{t=1}^{T} \alpha_{t}\sum_{s=1}^{t}\Delta v_{s}G(x_{s}) + \sum_{t=1}^{T}\alpha_{t}\sum_{s=1}^{t}\Delta v_{s}\tilde{\Omega}_{i_{t}}(x_{s}) \\
        &= \sum_{s=1}^{T} \Delta v_{s}G(x_{s})\sum_{t=s}^{T}\alpha_{t} + \sum_{t=1}^{T}\alpha_{t}\sum_{s=1}^{t}\Delta v_{s}\tilde{\Omega}_{i_{t}}(x_{s}) \\
        &\leq \sum_{s=1}^{T}\Delta v_{s}H_{s}G(x_{s}) + \sum_{t=1}^{T}\alpha_{t}\sum_{s=1}^{t}\Delta v_{s}\tilde{\Omega}_{i_{t}}(x_{s}) \\
        &= \sum_{t=1}^{T} a c_{t} G(x_{t}) + \sum_{t=1}^{T}\alpha_{t}\sum_{s=1}^{t}\Delta v_{s}\tilde{\Omega}_{i_{t}}(x_{s}),
    \end{align*}
    where the second equality is by interchanging the order of summation, and the last equality is by Definition \ref{def, auxiliary seq}. Hence
    \begin{equation}
    \label{theorem, the other bound, eq 5}
        \mathbb{E} \sum_{t=1}^{T}\alpha_{t}\sum_{s=1}^{t}\Delta v_{s}g_{i_{t}}(x_{s}) \leq a \sum_{t=1}^{T}c_{t}\mathbb{E}G(x_{t}) + \sum_{t=1}^{T}\alpha_{t}\sum_{s=1}^{t}\Delta v_{s}\mathbb{E}\tilde{\Omega}_{i_{t}}(x_{s}).
    \end{equation}
    Define $\tilde{D}\coloneqq \Vert x_{1}-x_{*}\Vert$, plug (\ref{theorem, the other bound, eq 5}) into (\ref{theorem, the other bound, eq 4}) and note that $\lambda = 1-a$, we have
    \begin{align}
    \notag
        \lambda \tilde{\mathcal{S}}_{X} + 2^{-1}\lambda\beta\mathbb{E}G(x_{M}) &\leq \frac{v_{1}\tilde{D}^{2}}{2} + \tilde{\mathcal{B}} - \frac{1}{2}\tilde{\mathcal{S}}_{D} + 3L\eta_{1}\tilde{\mathcal{S}}_{Z} + \frac{3}{2}\tilde{\mathcal{S}}_{V} \\
    \label{theorem, the other bound, eq 6}
        &\quad + 2^{-1}\lambda\beta\mathbb{E}G(z_{M+1}) + 2^{-1}\lambda\beta^2 L\mathbb{E}G(x_{M}).
    \end{align}
    By Lemma \ref{lemma, the other bound, bound S at Z}, we have
    \begin{align*}
        \lambda \tilde{\mathcal{S}}_{X} + 2^{-1}\lambda\beta\mathbb{E}G(x_{M}) &\leq \frac{v_{1}\tilde{D}^{2}}{2} + \left(120 C\Lambda L\eta_{1} + 6L\eta_{1} -\frac{1}{2} \right)\tilde{\mathcal{S}}_{D} + \left(120 C\Lambda L\eta_{1} + 6L\eta_{1} \right)\tilde{\mathcal{S}}_{X} \\
        &\quad + \left(120 C\Lambda + 360 C\Lambda L\eta_{1} + \frac{3}{2} \right)\tilde{\mathcal{S}}_{V} + 2^{-1}\lambda\beta\mathbb{E}G(z_{M+1}) + 2^{-1}\lambda\beta^2 L\mathbb{E}G(x_{M})\\
        &\leq \frac{v_{1}\tilde{D}^{2}}{2} + \frac{\lambda}{2}\tilde{\mathcal{S}}_{X} + 128 C\Lambda \tilde{\mathcal{S}}_{V} + 2^{-1}\lambda\beta\mathbb{E}G(z_{M+1}) + 2^{-1}\lambda\beta^2 L\mathbb{E}G(x_{M}), 
    \end{align*}
    where the last inequality is by the setting $120 C\Lambda L\eta_{1}\leq \frac{\lambda}{4}$. Rearranging the terms in the above equation gives
    \begin{align}
    \notag
        \lambda \sum_{t=1}^{T}\alpha_{t}v_{t+1}\mathbb{E}G(x_{t}) + \lambda\beta\mathbb{E}G(x_{M}) &\leq v_{1}\tilde{D}^{2} + 256C\Lambda \tilde{\mathcal{S}}_{V} + \lambda\beta\mathbb{E}G(z_{M+1}) \\
    \label{theorem, the other bound, eq 7}
        &\quad + \lambda\beta^2 L\mathbb{E}G(x_{M}).
    \end{align}
    By the definition of $\{z_{t}\}_{t}$ and convexity of $G$, we have
    \begin{align}
    \notag
        \lambda\beta G(z_{M+1}) &= \lambda\beta v_{M+1}G(z_{M+1})\leq \lambda\beta \sum_{t=1}^{T}\Delta v_{t}G(x_{t}) + \lambda\beta \Delta v_{M} G(x_{M}) \\
    \notag
        &\leq \lambda \sum_{t=1}^{T}\beta\Delta v_{t} G(x_{t}) + \lambda\beta (v_{M+1}-v_{M})G(x_{M}) \\
    \notag
        &\leq \lambda \sum_{t=1}^{T} H_{t}\Delta v_{t}G(x_{t}) + \lambda\beta (1-\lambda)G(x_{M}) \\
    \label{theorem, the other bound, eq 8}
        &\leq \lambda\sum_{t=1}^{T} \alpha_{t}v_{t+1}G(x_{t}) + \lambda\beta(1-\lambda) G(x_{M}),
    \end{align}
    where the third inequality is by $H_{t}=\sum_{s=t}^{M}\alpha_{s}\geq \alpha_{M}=\beta$, and the last inequality is by Definition \ref{def, auxiliary seq}. Plugging (\ref{theorem, the other bound, eq 8}) into (\ref{theorem, the other bound, eq 7}) gives
    \begin{equation}
    \label{theorem, the other bound, eq 9}
        \lambda^{2}\beta \mathbb{E}G(x_{M}) \leq v_{1}\tilde{D}^{2} + 256C\Lambda \tilde{\mathcal{S}}_{V} + \lambda\beta^{2}L\mathbb{E}G(x_{M}).
    \end{equation}
    Since the stepsize condition $\beta=\eta_{1}\leq 1/(480 C \Lambda ^{2} L)$ implies $\lambda\beta^{2}L\leq 257^{-1}\lambda^{2}\beta$, then
    \begin{equation*}
        \frac{256}{257}\lambda^{2}\beta\mathbb{E}G(x_{M})\leq v_{1}\tilde{D}^{2} + 256C\Lambda \tilde{\mathcal{S}}_{V}.
    \end{equation*}
    By Lemma \ref{lemma, aux seq property}, we have
    \begin{align*}
        \mathbb{E}G(x_{M}) &\leq \frac{2C \Lambda \tilde{D}^{2}}{H_{1}} + 257 C^{2}\Lambda^{2} \sigma_{*}^{2}\sum_{t=1}^{T}\frac{\alpha_{t}^{2}}{H_{t}} \\
        &\leq \frac{2C\Lambda\tilde{D}^{2}}{nA_{1}} + 257 C^{2}\Lambda^{2} \sigma_{*}^{2}\sum_{k=1}^{K}\eta_{k}^{2}\sum_{p=1}^{n} \frac{1}{(n-p+1)A_{k}} \\
        &\leq \frac{2C\Lambda\tilde{D}^{2}}{nA_{1}} + 257 C^{2}\Lambda^{2} \sigma_{*}^{2}\sum_{k=1}^{K}\frac{\eta_{k}^{2}}{A_{k}}\sum_{p=1}^{n}\frac{1}{p} \\
        &\leq \frac{2C\Lambda\tilde{D}^{2}}{nA_{1}} + 257 \left(1+\log(n)\right)C^{2}\Lambda^{2} \sigma_{*}^{2}\sum_{k=1}^{K}  \frac{\eta_{k}^{2}}{A_{k}},
    \end{align*}
    where the  second inequality is by $H_{1}=nA_{1}+\beta\geq nA_{1}$ and
    \begin{equation*}
        H_{\tau_{k,p}}=(n-p+1)\eta_{k}+nA_{k+1}+\beta\geq (n-p+1)(\eta_{k}+A_{k+1})=(n-p+1)A_{k}.
    \end{equation*}
    By Definition \ref{def, the other bound, basics}, we have
    \begin{equation*}
        \mathbb{E}\left[f(x_{M})-f_{*} \right] \leq \frac{2C\Lambda\tilde{D}^{2}}{nA_{1}} + 257 C^{2}\Lambda^{3}\sigma_{*}^{2}\sum_{k=1}^{K}\frac{\eta_{k}^{2}}{A_{k}},
    \end{equation*}
    which is the desired inequality.
\end{proof}

\bibliography{sn-bibliography}

\section*{AI Disclosure}

OpenAI's GPT-$5.6$ Sol is used as an interactive research assistant to help with the mathematical explorations. The authors take full responsibility for the content of this manuscript.

\end{document}